\documentclass[a4paper,12pt, oneside, reqno]{amsart}
\usepackage{mathrsfs}
\usepackage{amsfonts}
\usepackage{amsmath,amstext,amssymb,amsthm}
\usepackage[polish, english]{babel}
\usepackage{amsxtra}
\usepackage{color}
\usepackage{dsfont}
\usepackage[margin=2cm, centering]{geometry}
\usepackage[bbgreekl]{mathbbol}
\usepackage{soul}
\usepackage{graphics}

\usepackage[colorlinks,citecolor=blue,urlcolor=blue,bookmarks=true]{hyperref}

\DeclareSymbolFontAlphabet{\mathbb}{AMSb}
\DeclareSymbolFontAlphabet{\mathbbl}{bbold}

\usepackage{accents}

\theoremstyle{plain}
\newtheorem{theorem}{Theorem}[section] 
\newtheorem{lemma}[theorem]{Lemma}
\newtheorem{proposition}[theorem]{Proposition}
\newtheorem{corollary}[theorem]{Corollary}

 \theoremstyle{definition}
\newtheorem{defn}{Definition}

\newtheorem{example}{Example} 
\newtheorem*{remark*}{Remark} 
\newtheorem{remark}[theorem]{Remark}

\newcommand{\R}{\mathbb{R}}
\newcommand{\Rd}{{\R^{d}}}

\newcommand{\ind}{\mathds{1}}
\renewcommand{\leq}{\leqslant}
\renewcommand{\geq}{\geqslant}

\def\ind{{\bf 1}}
\def\qed{{\hfill $\Box$ \bigskip}}
\def\LL{{\mathcal L}}
\def\RR{{\mathbb R}}
\def\E{{\mathbb E}}
\def\N{{\mathbb N}}
\def\pf{\noindent{\bf Proof.} }

\def \PP{\mathbb{P}}
\def \EE{\mathbb{E}}

\def\({\left(} 
\def\){\right)} 
\def\[{\left[}
\def\]{\right]} 
\def\<{\langle} 
\def\>{\rangle}

\DeclareMathOperator{\supp}{supp}

\def \PG{\rm{(P)}}
\def \Pa{\rm{(P1)}}
\def \Pb{\rm{(P2)}}
\def \Pc{\rm{(P3)}}

\def \Aa{\rm{(A1)}}
\def \Ab{\rm{(A2)}}
\def \Ac{\rm{(A4)}}
\def \Ad{\rm{(A3)}}

\def \Aaa{\rm{(A{\!}^{'}1)}}
\def \Abb{\rm{(A{\!}^{'}2)}}
\def \Acc{\rm{(A{\!}^{'}3)}}

\newcommand{\lah}{\alpha_h}
\newcommand{\uah}{\beta_h}

\newcommand{\err}[2]{\rho_{#1}^{#2}} 
\newcommand{\param}{\sigma}  
\newcommand{\lmCJ}{\gamma_0}  

\newcommand{\aF}{\mathcal{F}} 

\newcommand{\LM}{N} 
\newcommand{\LCh}{\Psi} 
\newcommand{\uLCh}{\Phi} 
\newcommand{\drf}{b} 

\newcommand{\rr}{\Upsilon} 
\newcommand{\hrr}{\widehat{\rr}} 
\newcommand{\hh}{\phi} 

\def \Ht{\zeta}

\definecolor{ks}{rgb}{0.7,0.1,0.2}

\definecolor{zm}{RGB}{255,0,255}

\allowdisplaybreaks

\title{Heat kernels of non-symmetric L{\'e}vy-type operators}

\thanks{The research was partially supported by
 the 
German Science Foundation (SFB 701)
and
National Science Centre (Poland)
grant 2016/23/B/ST1/01665.}

\author[T. Grzywny]{Tomasz Grzywny }
\address{
	Wydzia{\lll} Matematyki,
	Politechnika Wroc{\lll}awska\\
	Wyb. Wyspia\'{n}skiego 27\\
	50-370 Wroc{\lll}aw\\
	Poland
}
\email{tomasz.grzywny@pwr.edu.pl}

\author[K. Szczypkowski]{Karol Szczypkowski}
\email{karol.szczypkowski@pwr.edu.pl}

\date{}

\begin{document}

\begin{abstract}
We construct the fundamental solution (the heat kernel)
$p^{\kappa}$ to the equation
$\partial_t=\LL^{\kappa}$, where under 
certain assumptions the operator $\LL^{\kappa}$
takes one of the following forms,
\begin{align*}
\LL^{\kappa}f(x)&:= \int_{\Rd}( f(x+z)-f(x)- \ind_{|z|<1} \left<z,\nabla f(x)\right>)\kappa(x,z)J(z)\, dz \,,\\
\LL^{\kappa}f(x)&:= \int_{\R^d}( f(x+z)-f(x))\kappa(x,z)J(z)\, dz\,, \\
\LL^{\kappa}f(x)&:= \frac1{2}\int_{\R^d}( f(x+z)+f(x-z)-2f(x))\kappa(x,z)J(z)\, dz\,.
\end{align*}
In particular, $J\colon \Rd \to [0,\infty]$
is a L{\'e}vy density, i.e., $\int_{\Rd}(1\land |x|^2)J(x)dx<\infty$.
The function
$\kappa(x,z)$ is assumed to be 
Borel measurable on $\Rd\times \Rd$ satisfying
$0<\kappa_0\leq \kappa(x,z)\leq \kappa_1$,   and
$|\kappa(x,z)-\kappa(y,z)|\leq \kappa_2|x-y|^{\beta}$ for some $\beta\in (0, 1)$.

We prove the uniqueness, estimates, regularity
and other qualitative properties of $p^{\kappa}$.
\end{abstract}

\maketitle



\section{Introduction}

The goal of this paper is to extend (and improve) the results of \cite{MR3500272} and \cite{KSV16} to more general operators than therein considered. These operators will be non-symmetric and not necessarily stable-like.
On the occasion we mostly cover (excluding one case which study we postpone) a contemporaneous paper~\cite{PJ} (see also \cite{CZ-new} and \cite{MR3652202}).
Let $d\in\N$ and
$\nu:[0,\infty)\to[0,\infty]$ be a non-increasing  function  satisfying
$$\int_{\Rd}  (1\land |x|^2) \nu(|x|)dx<\infty\,.$$
We consider
$J: \Rd  \to [0, \infty]$ 
 such that for some  $\lmCJ \in [1,\infty)$ and 
 all $x\in \Rd$,
\begin{equation}\label{e:psi1}
\lmCJ^{-1} \nu(|x|)\leq J(x) \leq \lmCJ \nu(|x|)\,.
\end{equation}
Further, suppose that 
$\kappa(x,z)$ is a Borel 
function on $\Rd\times \Rd$ such that
\begin{equation}\label{e:intro-kappa}
0<\kappa_0\leq \kappa(x,z)\leq \kappa_1\, , 
\end{equation}
and 
for some $\beta\in (0,1)$,
\begin{equation}\label{e:intro-kappa-holder}
|\kappa(x,z)-\kappa(y,z)|\leq \kappa_2|x-y|^{\beta}\, .
\end{equation}
For $r>0$ we define 
$$
h(r):= \int_0^\infty \left(1\land \frac{|x|^2}{r^2}\right) \nu(|x|)dx\,,\qquad \quad
K(r):=r^{-2} \int_{|x|<r}|x|^2 \nu(|x|)dx\,.
$$
The above functions  play a prominent role in the paper.
Our main assumption is \emph{the weak  scaling condition} at the origin: there exist $\lah\in (0,2]$ and $C_h \in [1,\infty)$ such that 
\begin{equation}\label{eq:intro:wlsc}
 h(r)\leq C_h\,\lambda^{\lah}\,h(\lambda r)\, ,\quad \lambda\leq 1, r\leq 1\, .
\end{equation}
In a  similar fashion:
there exist $\uah\in (0,2]$ and $c_h\in (0,1]$ such that
\begin{equation}\label{eq:intro:wusc}
 h(r)\geq c_h\,\lambda^{\uah}\,h(\lambda r)\, ,\quad \lambda\leq 1, r\leq 1\, .\\
\end{equation}

\begin{defn}
We define the following three sets of assumptions,
\begin{enumerate}
\item[] 
\begin{enumerate}
\item[$\Pa$] \quad  \eqref{e:psi1}--\eqref{eq:intro:wlsc}
hold and  $1< \lah \leq 2$,
\item[$\Pb$] \quad \eqref{e:psi1}--\eqref{eq:intro:wusc}
hold and   $0<\lah \leq \uah <1$,
\item[$\Pc$] \quad \eqref{e:psi1}--\eqref{eq:intro:wlsc} hold, $J$ is symmetric and $\kappa(x,z)=\kappa(x,-z)$, $x,z\in\Rd$.
\end{enumerate}
\end{enumerate}
We say that $\PG$  holds if $\Pa$ or $\Pb$ or $\Pc$ is satisfied.
\end{defn}

In each case $\Pa$, $\Pb$, $\Pc$, respectively,
we consider an operator
\begin{align}
\LL^{\kappa}f(x)&:= \int_{\Rd}( f(x+z)-f(x)- \ind_{|z|<1} \left<z,\nabla f(x)\right>)\kappa(x,z)J(z)\, dz \,, \label{e:intro-operator-a1}\\
\LL^{\kappa}f(x)&:= \int_{\Rd}( f(x+z)-f(x))\kappa(x,z)J(z)\, dz \label{e:intro-operator-a2}\,, \\
\LL^{\kappa}f(x)&:= \frac1{2}\int_{\R^d}( f(x+z)+f(x-z)-2f(x))\kappa(x,z)J(z)\, dz \label{e:intro-operator-a3}\,.
\end{align}
We denote by
$\LL^{\kappa,\varepsilon}f$ 
the expressions 
\eqref{e:intro-operator-a1}, \eqref{e:intro-operator-a2} or \eqref{e:intro-operator-a3}
with $J(z)$ replaced  
by $J_{\varepsilon}(z):=J(z)\ind_{|z|>\varepsilon}$, $\varepsilon \in [0,1]$.
We apply the above operators (in a strong or weak sense) only when they are well defined according to 
the following definition.
Let $f\colon \Rd\to \R$ be a Borel measurable function.

\begin{description}
\item[Strong operator] \hfill\\
The operator $\LL^{\kappa}f$ is well defined if the corresponding integral 
converges absolutely,
and  in the case $\Pa$
the gradient $\nabla f(x)$ exists for every $x\in\Rd$.
\item[Weak operator] \hfill\\
The operator $\LL^{\kappa,0^+}f$
is well defined if the  limit 
exists for every $x\in\Rd$,
\begin{equation*}
\LL^{\kappa,0^+}f(x):=\lim_{\varepsilon \to 0^+}\LL^{\kappa,\varepsilon}f(x)\,,
\end{equation*}
where for $\varepsilon \in (0,1]$ the (strong) operators $\LL^{\kappa,\varepsilon}f$ are well defined.
\end{description} 
\noindent
The operator $\LL^{\kappa,0^+}$ is an extension of $\LL^{\kappa,0}= \LL^{\kappa}$, meaning
that if $\LL^{\kappa}f$ is well defined, then is so $\LL^{\kappa,0^+}f$ and $\LL^{\kappa,0^+}f=\LL^{\kappa}f$.
Therefore, it is desired to prove the existence of a solution to the equation $\partial_t=\LL^{\kappa}$ and the uniqueness of a solution to $\partial_t=\LL^{\kappa,0^+}$.

We  emphasize that 
in general we do not assume the symmetry of $J$.
We also point out that whenever $J$ is symmetric and $\kappa(x,z)=\kappa(x,-z)$, $x,z\in\Rd$, then
for any bounded function $f\in C^2(\Rd)$ 
the three operators \eqref{e:intro-operator-a1}--\eqref{e:intro-operator-a3}
coincide
and
\begin{align}\label{rem:symmetry}
\LL^{\kappa}f(x)
&=\lim_{\varepsilon \to 0^+} \int_{|z|>\varepsilon}
( f(x+z)-f(x))\kappa(x,z)J(z)\, dz\,.
\end{align}
The above equality may hold for other particular choices of $f$.
The assumptions on $f$ may also be relaxed
after replacing the left hand side with  $\LL^{\kappa,0^+}f(x)$.

Here are our main results.
\begin{theorem}\label{t:intro-main}
Assume $\PG$.
There is a unique function $p^{\kappa}(t,x,y)$ 
on $(0,\infty)\times \Rd\times \Rd$ 
such~that
\begin{itemize}
\item[(i)]  For all $t>0$, $x,y\in \Rd$, $x\neq y$,
\begin{equation}\label{e:intro-main-1}
\partial_t p^{\kappa}(t,x,y)=\LL_x^{\kappa,0^+}p^{\kappa}(t,x, y)\,.
\end{equation}
\item[(ii)] The function $p^{\kappa}(t,x,y)$ is jointly continuous on $(0,\infty)\times \Rd\times \Rd$
and
for any
$f\in C_c^{\infty}(\Rd)$,
\begin{equation}\label{e:intro-main-5}
\lim_{t\to 0^+}\sup_{x\in \Rd}\left| \int_{\Rd}p^{\kappa}(t,x,y)f(y)\, dy-f(x)\right|=0\, .
\end{equation}
\noindent
\item[(iii)]
 For all $0<t_0<T$ there are $c>0$ and $f_0\in L^{1}(\Rd)$ 
such that for all $t\in (t_0,T]$, $x,y\in\Rd$,
\begin{equation}\label{e:intro-main-2}
|p^{\kappa}(t,x,y)|\le c f_0(x-y)\,,
\end{equation}
and
\begin{equation}\label{e:intro-main-4}
|\LL_x^{\kappa, \varepsilon}p^{\kappa}(t,x,y)|\leq c \,,\qquad \varepsilon \in (0,1]\,.
\end{equation}
\end{itemize}
In the case $\Pa$, additionally:
\begin{itemize}
\item[(iv)]  For every $t>0$ there is $c>0$ such that for all $x,y\in\Rd$,
\begin{equation}\label{e:intro-main-a1}
|\nabla_x p^{\kappa}(t,x,y)|\leq c\,. 
\end{equation}
\end{itemize}
\end{theorem}

In the next theorem we collect more qualitative properties of $p^{\kappa}(t,x,y)$.
To this end, 
for $t>0$ and $x\in \R^d$ we define {\it the bound function},
\begin{equation}\label{e:intro-rho-def}
\rr_t(x):=\left( [h^{-1}(1/t)]^{-d}\land \frac{tK(|x|)}{|x|^{d}} \right) .
\end{equation}
\begin{theorem}\label{t:intro-further-properties}
Assume $\PG$. The following hold true.
\begin{enumerate}
\item[\rm (1)] (Non-negativity) The function $p^{\kappa}(t,x,y)$ is non-negative on $(0,\infty)\times\Rd\times\Rd$.
\item[\rm (2)] (Conservativeness) For all $t>0$, $x\in\Rd$, 
\begin{equation*}
\int_{\Rd}p^{\kappa}(t,x,y) dy =1\, .
\end{equation*}
\item[\rm (3)] (Chapman-Kolmogorov equation) For all $s,t > 0$, $x,y\in \R^d$,
\begin{equation*}
\int_{\R^d}p^{\kappa}(t,x,z)p^{\kappa}(s,z,y)\, dz =p^{\kappa}(t+s,x,y)\, .
\end{equation*}
\item[\rm (4)] (Upper estimate) For every $T>0$ there is $c>0$ such that for all $t\in (0,T]$, $x,y\in \Rd$,
\begin{equation*}
p^{\kappa}(t,x,y) \leq c \rr_t(y-x)\, .
\end{equation*}
\item[\rm (5)] (Factional derivative) For every $T>0$ there is $c>0$ such that for all
$t\in (0,T]$, $x,y\in\Rd$,
\begin{align*}
|\LL_x^{\kappa, \varepsilon} p^{\kappa}(t, x, y)|\leq c t^{-1}\rr_t(y-x)\,,\qquad \varepsilon \in [0,1]\,.
\end{align*}
\item[\rm (6)] (Gradient)
If $1-\lah<\beta\land \lah$, then for every
$T>0$ there is $c>0$ such that for all
$t\in (0,T]$, $x,y\in\Rd$, 
\begin{equation*}
\left|\nabla_x p^{\kappa}(t,x,y)\right|\leq  c\! \left[h^{-1}(1/t)\right]^{-1} \rr_t(y-x)\,. 
\end{equation*}
\item[\rm (7)] (Continuity) The function
$\LL_x^{\kappa} p^{\kappa}(t,x,y)$ is jointly continuous on $(0,\infty)\times \Rd\times\Rd$.
\item[\rm (8)] (Strong operator)
For all $t>0$, $x,y\in\Rd$,
\begin{equation*}
\partial_t p^{\kappa}(t,x,y)= \LL_x^{\kappa}\, p^{\kappa}(t,x,y)\,.
\end{equation*}
\item[\rm (9)] (H\"older continuity) For all $T>0$, $\gamma \in [0,1] \cap[0,\lah)$,
there is  $c>0$ such that for all $t\in (0,T]$ and $x,x',y\in \Rd$,
\begin{equation*}
\left|p^{\kappa}(t,x,y)-p^{\kappa}(t,x',y)\right| \leq c 
 (|x-x'|^{\gamma}\land 1) \left[h^{-1}(1/t)\right]^{-\gamma} \big( \rr_t(y-x)+ \rr_t(y-x') \big).
\end{equation*}
\item[\rm (10)] (H\"older continuity)
For all $T>0$, 
$\gamma\in [0,\beta)\cap [0,\lah)$,
there is $c>0$ such that for all $t\in (0,T]$ and $x,y,y'\in \Rd$,
\begin{equation*}
\left|p^{\kappa}(t,x,y)-p^{\kappa}(t,x,y')\right| \leq c 
(|y-y'|^{\gamma}\land 1) \left[h^{-1}(1/t)\right]^{-\gamma} \big( \rr_t(y-x)+ \rr_t(y-x') \big).
\end{equation*}
\end{enumerate}
The constants in {\rm (4) -- (6)} may be chosen to depend only on $d, \lmCJ, \kappa_0, \kappa_1, \kappa_2, \beta, \lah,  C_h, h, T$
(and $\uah$, $c_h$ in the case $\Pb$).
The same for {\rm (9)} and {\rm (10)} but with additional dependence on $\gamma$. 
\end{theorem}

For $t>0$ we define 
\begin{equation}\label{e:intro-semigroup}
P_t^{\kappa}f(x)=\int_{\Rd} p^{\kappa}(t,x,y)f(y)\, dy\, ,\quad x\in \Rd\, ,
\end{equation}
whenever the integral exists in the Lebesgue sense.
We also put $P_0^{\kappa}$ to be the identity operator.

\begin{theorem}\label{thm:onC0Lp}
Assume $\PG$. The following hold true.
\begin{enumerate}
\item[\rm (1)]  $(P^{\kappa}_t)_{t\geq 0}$ is an analytic strongly continuous positive contraction semigroup 
on \mbox{$(C_0(\Rd),\|\cdot\|_{\infty})$.}
\item[\rm (2)]  $(P^{\kappa}_t)_{t\geq 0}$ is an analytic strongly continuous  semigroup on every $(L^p(\Rd),\|\cdot\|_p)$, \mbox{$p\in [1,\infty)$.}
\item[\rm (3)] Let $(\mathcal{A}^{\kappa},D(\mathcal{A}^{\kappa}))$ be the 
generator of $(P_t^{\kappa})_{t\geq 0}$ on $(C_0(\Rd),\|\cdot\|_{\infty})$.\\
 Then
\begin{enumerate}
\item[\rm (a)] $C_0^2(\Rd) \subseteq D(\mathcal{A}^{\kappa})$ and $\mathcal{A}^{\kappa}=\LL^{\kappa}$ on $C_0^2(\Rd)$,
\item[\rm (b)] $(\mathcal{A}^{\kappa},D(\mathcal{A}^{\kappa}))$ is the closure of $(\LL^{\kappa}, C_c^{\infty}(\Rd))$,
\item[\rm (c)] the function $x\mapsto p^{\kappa}(t,x,y)$ belongs to $D(\mathcal{A}^{\kappa})$ for all $t>0$, $y\in\Rd$, and
$$
\mathcal{A}^{\kappa}_x\, p^{\kappa}(t,x,y)= \LL_x^{\kappa}\, p^{\kappa}(t,x,y)=\partial_t p^{\kappa}(t,x,y)\,,\qquad x\in\Rd\,.
$$
\end{enumerate}
\item[\rm{(4)}]  Let $(\mathcal{A}^{\kappa},D(\mathcal{A}^{\kappa}))$ be the 
generator of $(P_t^{\kappa})_{t\geq 0}$ on $(L^p(\Rd),\|\cdot\|_p)$, $p\in [1,\infty)$.\\
 Then
\begin{enumerate}
\item[\rm (a)] $C_c^2(\Rd) \subseteq D(\mathcal{A}^{\kappa})$ and $\mathcal{A}^{\kappa}=\LL^{\kappa}$ on $C_c^2(\Rd)$,
\item[\rm (b)] $(\mathcal{A}^{\kappa},D(\mathcal{A}^{\kappa}))$ is the closure of $(\LL^{\kappa}, C_c^{\infty}(\Rd))$,
\item[\rm (c)] the function $x\mapsto p^{\kappa}(t,x,y)$ belongs to $D(\mathcal{A}^{\kappa})$ for all $t>0$, $y\in\Rd$, and in $L^p(\Rd)$,
$$
\mathcal{A}^{\kappa} \, p^{\kappa}(t,\cdot,y)= \LL^{\kappa}\, p^{\kappa}(t,\cdot,y)=\partial_t p^{\kappa}(t,\cdot,y)\,.
$$
\end{enumerate}
\end{enumerate}
\end{theorem}

Finally, 
  (by probabilistic methods) we provide
a lower bound for the heat kernel $p^{\kappa}(t,x,y)$.
\begin{theorem} \label{thm:lower-bound}
Assume $\PG$. The following hold true.
\begin{itemize}
\item[(i)] There are $T_0=T_0(d,\nu,\param,\kappa_2,\beta)>0$ and $c=c(d,\nu,\param)>0$ such that for all $t\in (0,T_0]$, $x,y\in\Rd$,
\begin{equation}\label{e:intro-main-11}
p^{\kappa}(t,x,y)\geq c\left(
[h^{-1}(1/t)]^{-d}\wedge t \nu \left( |x-y|\right)\right).
\end{equation}

\item[(ii)] If additionally $\nu$ is positive, then for every $T>0$ there is $c=c(d,T,\nu,\param,\kappa_2,\beta)>0$ such that \eqref{e:intro-main-11} holds for all $t\in(0,T]$ and $x,y\in\Rd$.

\item[(iii)] If additionally there are $\bar{\beta}\in [0,2)$ and $\bar{c}>0$ such that
$\bar{c} \lambda^{d+\bar{\beta}} \nu (\lambda r) \leq \nu(r)$,
$\lambda \leq 1$, $r>0$, then  for every $T >0$
there is $c=c(d,T,\nu,\param,\kappa_2,\beta,\bar{c},\bar{\beta})>0$ such that for all $t\in(0,T]$ and $x,y\in\Rd$,
 \begin{equation}\label{e:intro-main-111}
p^{\kappa}(t,x,y)  \geq  c  \rr_t(y-x)\,.
\end{equation}
\end{itemize}
\end{theorem}

\begin{remark}\label{rem:MP}
Theorem~\ref{thm:onC0Lp} guarantees that $(P_t^{\kappa})_{t  \geq 0}$ is a Feller semigroup and therefore 
there exists 
{\it the canonical  Feller process} $X=(X_t)_{t\geq 0}$ corresponding to $(P_t^{\kappa})_{t  \geq 0}$ 
with 
trajectories that are 
c{\`a}dl{\`a}g 
functions
(see \cite[page 380]{MR1876169}).
The process $X$ 
is the unique  {\it solution to the martingale problem} for $(\LL^{\kappa},C_c^{\infty}(\Rd))$.
The latter follows from 
part (3a) of Theorem~\ref{thm:onC0Lp} and
\cite[Theorem~4.4.1]{MR838085} (see also \cite[Theorem~1.2.12 and Proposition~4.1.7]{MR838085}).
\end{remark}

\begin{remark}
The upper estimate of the heat kernel leads to a sufficient condition for a Borel measure
to belong to the Kato class with respect to $p^{\kappa}(t,x,y)$, equivalently, 
 to $X=(X_t)_{t\geq 0}$.
Similarly, the lower bound provides a necessary condition (cf. \cite[Theorem~2.7]{MR3652202}). 
Moreover, if $\PG$ and the assumption of
Theorem~\ref{thm:onC0Lp}(iii) are satisfied, then $p^{\kappa}$ is locally in time and globally in space comparable with the heat kernel $p$ of a pure-jump L{\'e}vy process $Y=(Y_t)_{t \geq 0}$
corresponding to $\nu(|x|)$ 
(see Section~\ref{sec:appA}, \cite[Remark~5.7 and Corollary~5.14]{GS-2017}).
Thus the Kato class for $X$ and $Y$ is the same.
The function Kato classes that consist of absolutely continuous measures are for L{\'e}vy processes well studied \cite{MR3713578}.
\end{remark}

\begin{remark}\label{rem:smaller_beta}
If 
\eqref{e:intro-kappa}, \eqref{e:intro-kappa-holder} hold, then $|\kappa(x,z)-\kappa(y,z)|\leq (2\kappa_1 \vee \kappa_2)|x-y|^{\beta_1}$ for every $\beta_1 \in [0,\beta]$.
\end{remark}

For the purpose of the introduction we give an example right at this moment.
\begin{example}\label{ex:1}
Our results apply if \eqref{e:psi1}
holds with
 $\nu(r)=r^{-d} [ \log (1+r^{\alpha/2})]^{-2}$, where \mbox{$\alpha\in (0,2)$.}
Indeed, 
the conditions \eqref{eq:intro:wlsc} and~\eqref{eq:intro:wusc}
are satisfied with $\lah=\uah=\alpha$, see \cite[Example~2]{GS-2017}.
Further, Theorem~\ref{thm:lower-bound}(iii) also applies.
We emphasize that such $\nu$ does not have the logarithmic moment at infinity,
$$
\int_{\Rd} \ln\left(1+ |z|^2\right) \nu(|z|)dz=\infty\,.
$$
\end{example}

The non-local integro-differential operators under our considerations
belong to the class of operators known as 
{\it L{\'e}vy-type}.
Due to the Courr{\`e}ge-Waldenfels theorem
\cite[Theorem~4.5.21]{MR1873235},
\cite[Theorem~2.21]{MR3156646}
those operators are  generic
for Feller semigroups
whose infinitesimal generator has sufficiently rich domain.
We refer the reader to
\cite{MR1873235, MR1917230, MR2158336} and
\cite{MR3156646} for a broad survey on L{\'e}vy-type operators.
Nevertheless, it is highly non-trivial to construct the semigroup from a 
given L{\'e}vy-type operator with non-constant coefficients,
and even more difficult to investigate its heat kernel.
The tool used in this paper is the parametrix method,
proposed by E. Levi \cite{zbMATH02644101} to solve elliptic Cauchy problems.
It was successfully applied
in the theory of partial differential equations 
\cite{zbMATH02629782}, 
\cite{MR1545225},
\cite{MR0003340},
\cite{zbMATH03022319}, with an
overview in the monograph \cite{MR0181836},
as well as in the theory of pseudo-differential operators \cite{MR2093219}, \cite{KSV16}, \cite{MR3652202}, \cite{FK-2017}, \cite{MR3294616}. 
In particular, operators comparable in a sense with the fractional Laplacian were intensively studied 
\cite{MR0492880},
\cite{MR616459},
\cite{MR972089},
\cite{MR1744782},
\cite{MR2093219},
also very recently
\cite{MR3500272},
\cite{PJ},
\cite{CZ-new},
\cite{KR-2017}.
More detailed historical comments on the development of the method can be found in \cite[Bibliographical Remarks]{MR0181836} and in the introductions of \cite{MR3652202} and \cite{BKS-2017}.

We will now elaborate on our assumptions in view of the literature
in terms of two selected aspects:
the admissible L{\'e}vy measures and the symmetry condition.
This will not fully exhaust the relations between all various papers, their assumptions and results.

First we focus on
the L{\'e}vy measure $J(z)dz$
and we point out
three papers \cite{MR3500272}, \cite{KSV16}, \cite{MR3652202},
two of which are
at the opposite poles. 
In the paper \cite{MR3500272} the
authors concentrate on a particular isotropic $\alpha$-stable case $J(z)=|z|^{-d-\alpha}$, $\alpha\in (0,2)$,
and, among other things, give explicit estimates of the fundamental solution.
In \cite{MR3652202} much more general not necessarily absolutely continuous L{\'e}vy measures are treated,
but  the estimates are stated in a rather implicit form of compound kernels.
Finally the paper \cite{KSV16} is situated
between those extremes.  
The authors of \cite{KSV16} follow the road-map of~\cite{MR3500272}
and consider $J(z)$
 comparable with a L{\'e}vy
density $j(|z|)$ of a subordinate Brownian motion.
In this respect our assumption is 
given by
\eqref{e:psi1}
and stands for the comparability of 
$J(z)$ with an isotropic unimodal L{\'e}vy density~$\nu(|z|)$,
which allows for much larger class of L{\'e}vy measures
than in \cite{KSV16}.
In particular,
we can consider compactly supported L\'{e}vy measures.
With this in mind it locates us between \cite{KSV16} and \cite{MR3652202}.

Another assumption on the L{\'e}vy measure is 
the weak scaling~\eqref{eq:intro:wlsc},
which naturally generalizes the scaling property 
of the isotropic $\alpha$-stable case \cite{MR3500272},
and is also present in \cite{KSV16} and~\cite{MR3652202}.
More precisely,
the condition \cite[(1.4)]{KSV16} is
equivalent to  \eqref{eq:intro:wlsc}
due to \eqref{ineq:comp_unimod} and \eqref{eq:hcompPhi}, while
under \eqref{e:psi1} the condition
\cite[A1]{MR3652202}
 is equivalent to \eqref{eq:intro:wlsc}.
 The latter
is a consequence of the equivalence
of conditions ${\rm (C3)}$ and ${\rm (C4)}$
in
\cite[Theorem~3.1]{GS-2017},
${\rm (A1)}$ and ${\rm (A3)}$ in \cite[Lemma~2.3]{GS-2017}
and \eqref{ineq:comp_unimod} below.
In other words, here
our assumptions coincide with those of \cite{MR3652202} restricted to absolutely continuous L{\'e}vy measures satisfying \eqref{e:psi1}.
In fact, in $\Pb$ we also need one more weak scaling \eqref{eq:intro:wusc}, 
but this case is not in question of any of the papers \cite{MR3500272}, \cite{KSV16}, \cite{MR3652202}.

Furthermore, 
in comparison with \cite{KSV16}
we avoid two more technical assumptions \cite[(1.5) and (1.9)]{KSV16} on the behavior of the L{\'e}vy measure at infinity.
This is achieved by the
choice of the form of the bound function $\rr_t(x)$ supported by outcomes of \cite{GS-2017},
and the formulation of the maximum principle in 
Theorem~\ref{t:nonlocal-max-principle}. 
We note for instance that the
L{\'e}vy measure
in Example~\ref{ex:1}, which
is admissible by our assumptions,
does not satisfy \cite[(1.5)]{KSV16},
see \eqref{ineq:comp_unimod} and Lemma~\ref{lem:log_LM},
so the result of \cite{KSV16} cannot be applied in that case.

The assumptions \eqref{e:intro-kappa} and \eqref{e:intro-kappa-holder} 
on the function $\kappa(x,z)$
are common.
In both papers \cite{MR3500272} and \cite{KSV16} also the symmetry condition, i.e., the symmetry of $J$ and $\kappa(x,z)=\kappa(x,-z)$
 for all
 $x,z\in\Rd$, is required.
We cover
such situation in the case $\Pc$.
We note in passing 
 that this is a different symmetry than the one used in the theory of Dirichlet
forms \cite{MR2778606}.
In the cases $\Pa$ and $\Pb$ 
the symmetry condition is absent.
As explained before
 \eqref{rem:symmetry} the symmetry enables to represent the operator $\LL^{\kappa}$ in 
various equivalent forms, which facilitates calculations.
In the non-symmetric case 
the
intrinsic drift $\int_{|z|<1} z \kappa(x,z)J(z)dz$ may not be  negligible and
one has to be more specific in the choice of the operator.
In two recent papers 
\cite{PJ} and \cite{CZ-new}
the authors investigate the non-symmetric case for 
$J(z)=|z|^{-d-\alpha}$
and they consider the operator
(a) \eqref{e:intro-operator-a1} if $\alpha\in(1,2)$; (b) \eqref{e:intro-operator-a1} if $\alpha=1$ and $\int_{r<|z|<R}z \kappa(x,z)J(z)dz=0$; (c) \eqref{e:intro-operator-a2} if  $\alpha\in (0,1)$. 
The cases (a) and (b) are covered in the present paper by cases $\Pa$ and $\Pb$.
The case (b) with extensions is a subject of our forthcoming paper.
In \cite{MR3652202}, except for the symmetric case, also (a) and $\Pa$ are included in the discussion (with the presence of a bounded H{\"o}lder continuous first order term).

Finally we devote a few words to qualitative improvements that we make even in the cases discussed in \cite{MR3500272} and~\cite{KSV16}.
First of all in Theorem~\ref{t:intro-main} we significantly simplify the formulation of the uniqueness of $p^{\kappa}$. 
In Theorem~\ref{e:intro-rho-def}  we extend the range of $\lah$ and $\beta$ for which the gradient $\nabla p^{\kappa}$ exists,
we prove 
joint continuity of $\LL^{\kappa} p^{\kappa}$ and H{\"o}lder continuity in the second spatial coordinate of $p^{\kappa}$.
In Theorem~\ref{thm:onC0Lp}
we provide more detailed analysis of the semigroup $P^{\kappa}_t$ and its generator on various spaces. 
As a consequence in part (8) of Theorem~\ref{e:intro-rho-def} we have that $p^{\kappa}(t,x,y)$ solves the equation $\partial_t=\LL_x^{\kappa}$
(and 
 $\partial_t=\LL_x^{\kappa,0^+}$) for all $t>0$, $x,y\in\Rd$, without the restriction $x\neq y$ 
(cf. \cite[(1.7)]{MR3500272}, \cite[(1.10)]{KSV16}).
Up to our knowledge the solvability of the equation with the strong operator $\LL_x^{\kappa}$  is a novelty, and
demands many technical reinforcements.

To sum up, we utterly generalize \cite{MR3500272} and \cite{KSV16}
by restricting very general assumptions of \cite{MR3652202} to L{\'e}vy measures satisfying \eqref{e:psi1} (the case $\Pb$ is not considered in \cite{MR3652202}).
Moreover, we strengthen certain results even for the isotopic $\alpha$-stable case \cite{MR3500272}
and we propose new outcomes.
We also extend the core parts of \cite{PJ} and \cite{CZ-new}  
for the non-symmetric case
(excluding non-symmetry with $\alpha=1$, time-dependence and small Kato drift).
Other closely related papers treat for instance
(symmetric)
singular L{\'e}vy measures \cite{BKS-2017}, \cite{KR-2017} or 
 (symmetric) exponential L{\'e}vy measures \cite{KJ-2018}.
Our contribution is that under relatively weak assumptions, and with a satisfactory generality
that allows for non-symmetric L{\'e}vy measures,
we obtain explicit results, which are a proper extension of the $\alpha$-stable case. 
To avoid ambiguity we give full proofs of all statements.
We also refer the reader to \cite{CZ-survey}
for partial survey and correction of
 certain gaps  of~\cite{MR3500272}.

In order to start the procedure of constructing
the solution to the L{\'e}vy-type operator
one needs certain knowledge about
the solution to the operator with frozen coefficients
which leads back to the L{\'e}vy case.
This initial information usually determines
the results accessible by the parametrix method. 
Therefore we observe pairs of papers like 
\cite{MR2008600, MR3500272}, \cite{MR3357585, KSV16}, \cite{MR3235175,MR3652202}, \cite{MR3139314,MR3353627},
\cite{MR2320691,BKS-2017}.
In our case we base on the results of \cite{GS-2017}, which has roots in \cite{MR3357585}.
Another important ingredient of the preliminaries are the so-called convolution inequalities
used to deal with multiply iterated integrals that appear in the construction. For the $\alpha$-stable case they can be found for instance in \cite[Lemma~5]{MR972089}. In Lemma~\ref{l:convolution} we propose a refined version 
motivated by \cite[Lemma 2.6]{KSV16} with more parameters and for function $\err{\beta}{\gamma}$ defined 
by means of the bound function.

There exist other methods to associate semigroup and heat kernel to an operator.
Some rely on the symbolic calculus 
\cite{MR0367492}, \cite{MR0499861}, \cite{MR666870},
\cite{MR1659620}, \cite{MR1254818}, \cite{MR1917230},
\cite{MR2163294}, \cite{MR2456894},
other on Dirichlet forms \cite{MR2778606}, \cite{MR898496}, \cite{MR2492992}, \cite{MR2443765}, \cite{MR2806700}
or perturbation series \cite{MR1310558}, \cite{MR2283957}, \cite{MR2643799}, \cite{MR2876511}, \cite{MR3550165}, \cite{MR3295773}.
For probabilistic methods and applications we refer the reader to \cite{MR3022725}, 
\cite{MR3544166}, \cite{MR1341116},
 \cite{MR3765882}, \cite{K-2015}.

The reminder of the paper is organized as follows. In Section~\ref{sec:analysis_LL} we
use the results of \cite{GS-2017}
 as a starting point
to establish further uniform 
properties of the heat kernel $p^{\mathfrak{K}}(t,x,y)$ of the L{\'e}vy operator $\LL^{\mathfrak{K}}$.
In Section~\ref{sec:construction} we carry out the construction of $p^{\kappa}(t,x,y)$
and we prove its primary properties.
According to the parametrix method we anticipate that
\begin{align*}
p^{\kappa}(t,x,y)=
p^{\mathfrak{K}_y}(t,x,y)+\int_0^t \int_{\Rd}p^{\mathfrak{K}_z}(t-s,x,z)q(s,z,y)\, dzds\,,
\end{align*}
where $q(t,x,y)$ solves the equation
\begin{align*}
q(t,x,y)=q_0(t,x,y)+\int_0^t \int_{\Rd}q_0(t-s,x,z)q(s,z,y)\, dzds\,,
\end{align*}
and $q_0(t,x,y)=\big(\LL_x^{{\mathfrak K}_x}-\LL_x^{{\mathfrak K}_y}\big) p^{\mathfrak{K}_y}(t,x,y)$.
Here $p^{\mathfrak{K}_w}$ is the heat kernel of the L{\'e}vy operator $\LL^{\mathfrak{K}_w}$
obtained from the operator $\LL^{\kappa}$ by freezing its coefficients: $\mathfrak{K}_w(z)=\kappa(w,z)$.
In Section~\ref{subsec:p_freeze} we examine $p^{\mathfrak{K}_y}(t,x,y)$. In
Section~\ref{subsec:q} we define $q(t,x,y)$ explicitly
 via the perturbation series and we study its properties. In Section~\ref{subsec:phi}
we investigate
$\phi_y(t,x)=\int_0^t \int_{\Rd}p^{\mathfrak{K}_z}(t-s,x,z)q(s,z,y)\, dzds$,
which is the most technical part, and several improvements that we make there affect the eventual results.
Finally, in Section~\ref{subsec:p^K} we collect initial properties of $p^{\kappa}$ that follow directly from the construction.
In Section~\ref{sec:main} we establish a nonlocal maximum principle, analyze the semigroup $(P_t^{\kappa})_{t\geqq 0}$, complement the fundamental properties of $p^{\kappa}$ and prove Theorems~\ref{t:intro-main}--\ref{thm:onC0Lp}.
In Section~\ref{sec:appA} and~\ref{sec:appB}
we store auxiliary results such as features of the bound function,
3G-type inequalities, convolution inequalities.

We end this section with comments on the notation.
Throughout the article
$\omega_d=2\pi^{d/2}/\Gamma(d/2)$ is the surface measure of the unit sphere in $\R^d$.
By $c(d,\ldots)$ we denote a generic
 positive constant that depends only on the listed parameters $d,\ldots$. 
As usual $a\land b=\min\{a,b\}$ and $a\vee b = \max\{a,b\}$.
We use ``$:=$" to denote a definition.
In what follows the constants
$\lmCJ$, $\kappa_0$, $\kappa_1$, $\kappa_2$, $\beta$, $\lah$,  $C_h$, $\uah$, $c_h$
can be regarded as fixed.

Excluding Section~\ref{sec:appA} and~\ref{sec:appB}
{\bf we assume in the whole paper that $\PG$ holds}. 
However, in theorems and propositions we explicitly formulate all assumptions.
If needed we make a restriction to $\Pa$ or $\Pb$ or $\Pc$.

\section*{Acknowledgment}
The authors thank K. Bogdan, A. Kulik, Z. Vondra\v{c}ek
for discussions and helpful comments.

\section{Analysis of the  heat kernel of $\LL^{\mathfrak{K}}$}\label{sec:analysis_LL}

In this section we assume that
$\mathfrak{K}\colon \Rd \to [0,\infty)$ is such that
$$
0<\kappa_0 \leq \mathfrak{K}(z) \leq \kappa_1\,.
$$
In the case $\Pa$, $\Pb$, $\Pc$ 
we consider an operator 
$\LL^{\mathfrak{K}}$
defined by taking $\kappa(x,z)=\mathfrak{K}(z)$ in
\eqref{e:intro-operator-a1}, \eqref{e:intro-operator-a2}, \eqref{e:intro-operator-a3}, respectively.
The operator uniquely determines a L{\'e}vy process
and its density
$p^{\mathfrak{K}}(t,x,y)=p^{\mathfrak{K}}(t,y-x)$ (see Section~\ref{sec:appB}. 
In particular, \eqref{ogolne:zal1} holds by \eqref{eq:intro:wlsc}, \eqref{eq:hcompPhi}, \eqref{ineq:comp_unimod} and \eqref{e:psi1}).
To simplify the notation we introduce
\begin{align}\label{e:delta-f-def}
\delta_1^{\mathfrak{K}} (t,x,y;z)&:=p^{\mathfrak{K}}(t,x+z,y)-p^{\mathfrak{K}}(t,x,y)-\ind_{|z|<1}\left< z,\nabla_x p^{\mathfrak{K}}(t,x,y)\right>\,,\\
\delta_2^{\mathfrak{K}} (t,x,y;z)&:=p^{\mathfrak{K}}(t,x+z,y)-p^{\mathfrak{K}}(t,x,y)\,,\\ 
\delta_3^{\mathfrak{K}} (t,x,y;z)&:=\frac1{2}(p^{\mathfrak{K}}(t,x+z,y)+p^{\mathfrak{K}}(t,x-z,y)-2p^{\mathfrak{K}}(t,x,y))\, .
\end{align}
Thus we have 
\begin{align}\label{eq:delta_gen}
\LL_x^{\mathfrak{K}_1} \,p^{\mathfrak{K}_2}(t,x,y)=\int_{\Rd}\delta^{\mathfrak{K}_2} (t,x,y;z)\, \mathfrak{K}_1(z)J(z)dz\,,
\end{align}
where $\delta^{\mathfrak{K}}$ is one of the above functions appropriate to the case under consideration.
We also introduce the sets of parameters
$\param_1 = (\lmCJ,\kappa_0,\kappa_1,\lah, C_h,h)$,
$\param_2= (\lmCJ,\kappa_0,\kappa_1,\lah,\uah, C_h, c_h,h)$,
$\param_3= (\lmCJ,\kappa_0,\kappa_1,\lah,C_h,h)$,
and we write shortly $\param$ 
if the case is clear from the context.

The result below is the initial point of the whole paper.

\begin{proposition}\label{prop:gen_est}
Assume $\PG$. For every $T>0$ and $\bbbeta\in \mathbb{N}_0^d$ there exists a constant $c=c(d,T,\bbbeta,\param)$
such that for all $t\in (0,T]$, $x,y\in\Rd$,
\begin{align*}
|\partial_x^{\bbbeta} p^{\mathfrak{K}}\left(t,x,y\right)|\leq 
c \left[h^{-1}(1/t) \right]^{-|\bbbeta|} \rr_t(y-x)\,.
\end{align*}
\end{proposition}
\pf
In the case $\Pa$, $\Pb$ and $\Pc$ 
the result follows 
from \cite[Section 5.2]{GS-2017}.
\qed

\begin{lemma}\label{prop:gen_est_low}
Assume $\PG$. For every $T,\theta>0$ there exists a constant $\tilde{c}=\tilde{c}(d,T,\theta,\nu,\param)$
such that for all $t\in (0,T]$ and $|x-y|\leq \theta h^{-1}(1/t)$,
\begin{align*}
p^{\mathfrak{K}}\left(t,x,y\right)\geq \tilde{c} \left[ h^{-1}(1/t)\right]^{-d}\,.
\end{align*}
\end{lemma}
\pf
In the case $\Pc$
the estimate follows from \cite[Corollary~5.11]{GS-2017}.
In the cases $\Pa$ and $\Pb$ we also use
\cite[Corollary~5.11]{GS-2017} but with
with $x-y- t\drf_{[h_0^{-1}(1/t)]}$ in place of $x$
as we have that $|t\drf_{[h_0^{-1}(1/t)]}|\leq a h_0^{-1}(1/t)$ for $a=a(d,T,\param)$, see proof of Proposition~5.9 and~5.10 in  \cite{GS-2017}.
\qed

\subsection{Increments and integrals of 
 $p^{\mathfrak{K}}(t,x,y)$}

We simplify the notation by introducing the following expressions. For $t>0$, $x,y,z\in\Rd$,
\begin{align*}
\aF_{1}&:=\rr_t(y-x-z)\ind_{|z|\geq h^{-1}(1/t)}+ \left[ \left(\frac{|z|}{h^{-1}(1/t)} \right)^2 \land \left(\frac{|z|}{h^{-1}(1/t)} \right) \right]  \rr_t(y-x),\\
\aF_{2}&:=\rr_t(y-x-z)\ind_{|z|\geq h^{-1}(1/t)}+ \left[ \left(\frac{|z|}{h^{-1}(1/t)}\right)\wedge 1\right] \rr_t(y-x),\\
\aF_{3}&:=\rr_t(y-x\pm z)\ind_{|z|\geq h^{-1}(1/t)} 
+ \left[ \left(\frac{|z|}{h^{-1}(1/t)}\right)^2 \land 1\right] \rr_t(y-x).
\end{align*}
In the last line we use $f(x\pm z)$ in place of $f(x+z)+f(x-z)$.
Hereinafter we add arguments $(t,x,y;z)$ when referring to functions defined above.

\begin{lemma}\label{lem:diff-HK}
Assume $\PG$. For every $T>0$ there exists a constant $c=c(d,T,\param)$
such that for all $t\in (0,T]$, $x,y,z\in\Rd$ we have
$\left|p^{\mathfrak{K}}(t,x+z,y)-p^{\mathfrak{K}}(t,x,y)\right|\leq c\,
\aF_2(t,x,y;z)$.
\end{lemma}
\pf 
If $|z|\ge h^{-1}(1/t)$, the result follows from Proposition~\ref{prop:gen_est}. If $|z|< h^{-1}(1/t)$, we use 
Proposition~\ref{prop:gen_est}
and
\begin{equation*}
p^{\mathfrak{K}}(t,x+z,y)-p^{\mathfrak{K}}(t,x,y)=\int_0^1 \left< z, \nabla_x p^{\mathfrak{K}}(t, x+\theta z,y)\right> d\theta\,,
\end{equation*}
to obtain
$
|p^{\mathfrak{K}}(t,x+z,y)-p^{\mathfrak{K}}(t,x,y)|\le  c_1 (|z|/h^{-1}(1/t))  \int_0^1  \rr_t(y-x-\theta z)
d \theta\, .$
Since $\theta|z|\le h^{-1}(1/t)$, we get from 
Corollary~\ref{cor:small_shift}
that
$$
|p^{\mathfrak{K}}(t,x+z,y)-p^{\mathfrak{K}}(t,x,y)|\leq  c_2 \left(\frac{|z|}{h^{-1}(1/t)}\right) \rr_t(y-x)\,.
$$
\qed

\begin{lemma}\label{lem:est_delta_1}
Assume $\Pa$. For every $T>0$ there exists a constant $c=c(d,T,\param_1)$ such that for all $t\in (0,T]$, $x,y,z\in\Rd$ we have
$
|\delta_1^{\mathfrak{K}}(t,x,y;z)|\leq c \big(
\aF_{1}(t,x,y;z)\ind_{|z|<1}+ \aF_{2}(t,x,y;z)\ind_{|z|\geq 1}\big)
$.
\end{lemma}
\pf
For $|z|\geq 1$ we apply Lemma~\ref{lem:diff-HK}.
Let $|z|<1$.
If $|z|\geq h^{-1}(1/t)$, then by Proposition~\ref{prop:gen_est},
$$
|\delta_1^{\mathfrak{K}}(t,x,y;z)|\leq  c \left( \rr_t(y-x-z) + \left( \frac{|z|}{h^{-1}(1/t)}\right) \rr_t(y-x)\right)\,.
$$
If $|z|<h^{-1}(1/t)$, by Proposition~\ref{prop:gen_est} and Corollary~\ref{cor:small_shift}  we have 
\begin{align*}
|\delta_1^{\mathfrak{K}} (t,x,y;z)|\leq |z|^2 \sum_{|\bbbeta|=2} \int_0^1\int_0^1 |\partial_x^{\bbbeta} p^{\mathfrak{K}}(t,x+\theta' \theta z,y)| d\theta' d\theta
\leq c \left(\frac{|z|}{h^{-1}(1/t)}\right)^2 \rr_t (y-x)\,.
\end{align*}
\qed

\begin{lemma}\label{lem:diff_delta_1}
Assume $\Pa$. For every $T>0$ there exists a constant $c=c(d,T,\param_1)$ such that for all $t\in (0,T]$, $x,x',y,z\in\Rd$ satisfying $|x'-x|<h^{-1}(1/t)$ we have
$$
|\delta_1^{\mathfrak{K}}(t,x',y;z)-\delta_1^{\mathfrak{K}}(t,x,y;z)|
\leq c\left(\frac{|x'-x|}{h^{-1}(1/t)}\right)  \big(
\aF_{1}(t,x,y;z)\ind_{|z|<1} + \aF_{2}(t,x,y;z)\ind_{|z|\geq 1}\big)
\,.
$$ 
\end{lemma}
\pf
We denote $w=x'-x$ and
we use Proposition~\ref{prop:gen_est} and Corollary~\ref{cor:small_shift}
repeatedly.
Note that $|w|<h^{-1}(1/t)$ and
\begin{align}\label{e:difference-delta-grad}
\delta_1^{\mathfrak{K}}(t,x+w,y;z)-\delta_1^{\mathfrak{K}}(t,x,y;z) = \int_0^1 \left< w, \nabla_x  \delta_1^{\mathfrak{K}}(t,x+\theta w,y;z)\right> d\theta\,.
\end{align}
For $|z|\geq 1$, if $|z|\geq h^{-1}(1/t)$, we apply \eqref{e:difference-delta-grad}
to get
\begin{align*}
|\delta_1^{\mathfrak{K}}(t,x+w,y;z)-\delta_1^{\mathfrak{K}}(t,x,y;z)| 
\leq c \left(\frac{|w|}{h^{-1}(1/t)}\right) \big( \rr_t(y-x-z)+ \rr_t(y-x) \big).
\end{align*}
If  $|z|<h^{-1}(1/t)$, we have
\begin{align*}
|\delta_2^{\mathfrak{K}}(t,x+w,y;z)-\delta_2^{\mathfrak{K}}(t,x,y;z)|
&\leq |z| |w|\sum_{|\bbbeta|=2} \int_0^1\int_0^1 |\partial_x^{\bbbeta} p^{\mathfrak{K}}(t,x+\theta w+\theta ' z,y)|\,d\theta ' d\theta\\
&\leq c \left( \frac{|w|}{h^{-1}(1/t)}\right)\left(\frac{|z|}{h^{-1}(1/t)}\right) \rr_t(y-x)\,.
\end{align*}
Let $|z|<1$. 
If $|z|\geq h^{-1}(1/t)$, then 
we use \eqref{e:difference-delta-grad} to obtain
\begin{align*}
|\delta_1^{\mathfrak{K}}(t,x+w,y;z)-\delta_1^{\mathfrak{K}}(t,x,y;z)| 
\leq c \left(\frac{|w|}{h^{-1}(1/t)}\right) \left( \rr_t(y-x-z)+ \left(\frac{|z|}{h^{-1}(1/t)}\right)\rr_t(y-x) \right).
\end{align*}
If $|z|<h^{-1}(1/t)$, then
\begin{align*}
|\delta_1^{\mathfrak{K}}(t,x+w,y;z)-\delta_1^{\mathfrak{K}}(t,x,y;z)| 
&\leq  |w||z|^2\sum_{|\bbbeta|=3} \int_0^1 \int_0^1 \int_0^1 |\partial_x^{\bbbeta} p^{\mathfrak{K}}(t,x+\theta w+ \theta'' \theta' z ,y)|\,d\theta'' d\theta' d\theta\\
&\leq c \left(\frac{|w|}{h^{-1}(1/t)}\right)\left(\frac{|z|}{h^{-1}(1/t)}\right)^2 \rr_t(y-x)\,.
\end{align*}
\qed

We note that the estimate for $|\delta_2^{\mathfrak{K}} (t,x,y;z)|\leq c \,\aF_2(t,x,y;z)$ is given in Lemma~\ref{lem:diff-HK}.
\begin{lemma}\label{lem:diff_delta_2}
Assume $\Pb$. For every $T>0$ there exists a constant $c=c(d,T,\param_2)$
such that for all $t\in (0,T]$, $x,x',y,z\in \Rd$ satisfying $|x'-x|<h^{-1}(1/t)$ we have
$$
|\delta_2^{\mathfrak{K}}(t,x',y;z)-\delta_2^{\mathfrak{K}}(t,x,y;z)|
\leq c \left(\frac{|x'-x|}{h^{-1}(1/t)}\right) \aF_2(t,x,y;z)\,.
$$
\end{lemma}
\pf
The proof is the same as for the case $|z|\geq h^{-1}(1/t)$ in the proof of Lemma~\ref{lem:diff_delta_1}.
\qed

\begin{lemma}\label{lem:est_delta_3}
Assume $\Pc$. For every $T>0$ there exists a constant $c=c(d,T,\param_3)$
such that for all $t\in(0,T]$, $x,y,z\in\Rd$ we have
$
|\delta_3^{\mathfrak{K}}(t,x,y;z)|\leq c\, \aF_3 (t,x,y;z)
$.
\end{lemma}
\pf
If $|z|\geq h^{-1}(1/t)$ we apply Proposition~\ref{prop:gen_est}. If $|z|<h^{-1}(1/t)$,
by Proposition~\ref{prop:gen_est} and
Corollary~\ref{cor:small_shift} we get
\begin{align*}
|\delta_3^{\mathfrak{K}}(t,x,y;z)|\leq |z|^2 \sum_{|\bbbeta|=2} \int_0^1 \int_{-1}^1 |\partial_x^{\bbbeta} p^{\mathfrak{K}}(t,x+\theta' \theta z,y)|\,d\theta ' d\theta
\leq c \left(\frac{|z|}{h^{-1}(1/t)}\right)^2 \rr_t(y-x)\,.
\end{align*}
\qed

\begin{lemma}\label{lem:diff_delta_3}
Assume $\Pc$.
For every $T>0$ the exists a constant $c=c(d,T,\param_3)$ such that for all 
$t\in(0,T]$, $x,x',y,z\in\Rd$ satisfying $|x'-x|<h^{-1}(1/t)$ we have
$$
|\delta_3^{\mathfrak{K}}(t,x',y;z)-\delta_3^{\mathfrak{K}}(t,x,y;z)|\leq 
c \left(\frac{|x'-x|}{h^{-1}(1/t)} \right) \aF_3(t,x,y;z)\,.
$$
\end{lemma}
\pf
Let $w=x'-x$. We use Proposition~\ref{prop:gen_est} and Corollary~\ref{cor:small_shift} repeatedly.
We have 
\begin{align*}
\delta_3^{\mathfrak{K}}(t,x+w,y;z)-\delta_3^{\mathfrak{K}}(t,x,y;z) = \int_0^1 \left< w, \nabla_x  \delta_3^{\mathfrak{K}}(t,x+\theta w,y;z)\right> d\theta\,.
\end{align*}
If $|z|\geq h^{-1}(1/t)$, then
\begin{align*}
|\delta_3^{\mathfrak{K}}(t,x+w,y;z)-\delta_3^{\mathfrak{K}}(t,x,y;z)| 
\leq c \left( \frac{|w|}{h^{-1}(1/t)}\right)
\left(\rr_t(y-x\pm z)+\rr_t(y-x) \right).
\end{align*}
If $|z|<h^{-1}(1/t)$, then
\begin{align*}
|\delta_3^{\mathfrak{K}}(t,x+w,y;z)-\delta_3^{\mathfrak{K}}(t,x,y;z)| 
&\leq |w||z|^2 \sum_{|\bbbeta|=3}\int_0^1 \int_0^1\int_{-1}^{1}|\partial_x^{\bbbeta} p^{\mathfrak{K}}(t,x+\theta w+\theta ''\theta ' z,y )|\, d\theta '' d\theta ' d\theta\\
&\leq c \left(\frac{|w|}{h^{-1}(1/t)}\right)\left(\frac{|z|}{h^{-1}(1/t)}\right)^2 \rr_t(y-x)\,.
\end{align*}
\qed

The next result is the counterpart of \cite[Theorem~2.4]{MR3500272} and \cite[Theorem~3.4]{KSV16}.

\begin{theorem}\label{thm:delta}
Assume $\PG$. For every $T>0$ there exists a constant $c=c(d,T,\param)$
such that for all $t\in(0,T]$, $x,x',y\in\Rd$,
\begin{align*}
&\int_{\Rd} |\delta^{\mathfrak{K}} (t,x,y;z)|\, \nu(|z|)dz \leq c t^{-1} \rr_t(y-x)\,, \quad and\\
\int_{\Rd} |\delta^{\mathfrak{K}} (t,x',y;z)-&\delta^{\mathfrak{K}} (t,x,y;z)|\, \nu(|z|)dz \leq c \left(\frac{|x'-x|}{h^{-1}(1/t)} \land 1\right)t^{-1} \big( \rr_t(y-x') + \rr_t(y-x)\big). 
\end{align*}
\end{theorem}
\pf
The first statement follows immediately from
Lemma~\ref{lem:est_delta_1}, \ref{lem:diff-HK} and~\ref{lem:est_delta_3}
supported by
Lemma~\ref{lem:int_J} and~\ref{lem:int_rr_J}.
We prove the second part.
If $|x'-x|\geq h^{-1}(1/t)$, then
\begin{align*}
\int_{\Rd} \left(|\delta^{\mathfrak{K}} (t,x',y;z)|+|\delta^{\mathfrak{K}} (t,x,y;z)|\right)\nu(|z|)dz
\leq c t^{-1} \left( \rr_t(y-x')+\rr_t(y-x) \right)\,.
\end{align*}
If $|x'-x|< h^{-1}(1/t)$, we rely on Lemma~\ref{lem:diff_delta_1}, \ref{lem:diff_delta_2} and~\ref{lem:diff_delta_3}
as well as Lemma~\ref{lem:int_J} and~\ref{lem:int_rr_J}.
\qed

\subsection{Continuous dependence of heat kernels with respect to $\mathfrak{K}$}

We discuss $\mathfrak{K}$, $\mathfrak{K}_1$, $\mathfrak{K}_2$
as introduced at the beginning of Section~\ref{sec:analysis_LL}.
In what follows $\|\cdot\|=\|\cdot\|_{\infty}$.

\begin{lemma}\label{lem:rozne_1}
Assume $\PG$.
For all $t>0$, $x,y\in\Rd$ and $s\in (0,t)$,
\begin{align*}
\frac{d}{d s} \int_{\Rd} &p^{\mathfrak{K}_1}(s,x,z) p^{\mathfrak{K}_2}(t-s,z,y)\,dz\\
&= \int_{\Rd} \LL_x^{\mathfrak{K}_1}p^{\mathfrak{K}_1}(s,x,z) \, p^{\mathfrak{K}_2}(t-s,z,y)\,dz
- \int_{\Rd} p^{\mathfrak{K}_1}(s,x,z)\, \LL_z^{\mathfrak{K}_2} p^{\mathfrak{K}_2}(t-s,z,y) \,dz\,,
\end{align*}
and
\begin{align*}
\int_{\Rd} \LL^{\mathfrak{K}}_x p^{\mathfrak{K}_1}(s,x,z)
p^{\mathfrak{K}_2}(t-s,z,y)\,dz
= &\int_{\Rd} p^{\mathfrak{K}_1}(s,x,z) \, \LL_z^{\mathfrak{K}} p^{\mathfrak{K}_2}(t-s,z,y)\,dz\,.
\end{align*}
\end{lemma}
\pf
Note that the difference quotient equals
\begin{align*}
\int_{\Rd} &\frac1{h}\left[ p^{\mathfrak{K}_1}(s+h,x,z) -p^{\mathfrak{K}_1}(s,x,z)\right] p^{\mathfrak{K}_2}(t-s-h,z,y)\,dz\\
&+
\int_{\Rd} p^{\mathfrak{K}_1}(s,x,z) \frac1{h} \left[ p^{\mathfrak{K}_2}(t-s-h,z,y)-p^{\mathfrak{K}_2}(t-s,z,y)\right] dz\,.
\end{align*}
If $2|h|<(t-s)\land s$,
by
Lemma~\ref{l:partial-time}(a), 
\eqref{eq:delta_gen},
Theorem~\ref{thm:delta},
Proposition~\ref{prop:gen_est}
and \eqref{e:nonincrease-t}
the first integrand is bounded by 
$s^{-1}\rr_{s/2}(z-x)\rr_{(t-s)/2}(y-z)$
up to multiplicative constant. Therefore we can use the dominated convergence theorem. Similarly, we deal with the second integral.
Next, by  \eqref{eq:delta_gen},
 Theorem~\ref{thm:delta} and  Proposition~\ref{prop:gen_est} 
the following integrals converge absolutely
and thus the change of the order of integration is justified,
\begin{align*}
\int_{\Rd}&\left(\int_{\Rd} \delta^{\mathfrak{K}_1}(s,x,z;w) \,\mathfrak{K}(w)J(w)dw\right) p^{\mathfrak{K}_2}(t-s,z,y)dz\\
&\quad=\int_{\Rd}\left( \int_{\Rd} \delta^{\mathfrak{K}_1}(s,x,z;w) p^{\mathfrak{K}_2}(t-s,z,y)dz\right) \mathfrak{K}(w)J(w)dw\\
&\quad=\int_{\Rd}\left(\int_{\Rd}
p^{\mathfrak{K}_1}(s,x,z) \delta^{\mathfrak{K}_2}(t-s,z,y;w)dz\right)  \mathfrak{K}(w)J(w)dw\\
&\quad=\int_{\Rd}
p^{\mathfrak{K}_1}(s,x,z) \left(\int_{\Rd}\delta^{\mathfrak{K}_2}(t-s,z,y;w)\,  \mathfrak{K}(w)J(w)dw\right) dz\,.
\end{align*}
In the third equality in the case $\Pa$ we used integration by parts.
\qed

The following result is the counterpart of \cite[Theorem 2.5]{MR3500272} and \cite[Theorem 3.5]{KSV16}.

\begin{theorem}\label{thm:cont_kappa}
Assume $\PG$. For every $T>0$ there exists a constant $c=c(d,T,\param)$
such that
for all $t\in (0,T]$, $x,y,z\in\Rd$,
\begin{align*}
|p^{\mathfrak{K}_1}(t,x,y)-p^{\mathfrak{K}_2}(t,x,y)| &\leq c \|\mathfrak{K}_1 - \mathfrak{K}_2\|  \rr_t(y-x)\,,\\ 
|\nabla_x p^{\mathfrak{K}_1}(t,x,y)-\nabla_x p^{\mathfrak{K}_2}(t,x,y)| &\leq c \|\mathfrak{K}_1 - \mathfrak{K}_2\| \left[h^{-1}(1/t)\right]^{-1} \rr_t(y-x) \,,\\ 
\int_{\Rd} |\delta^{\mathfrak{K}_1} (t,x,y;z)-\delta^{\mathfrak{K}_2} (t,x,y;z)| \,\nu(|z|)dz &\leq c \|\mathfrak{K}_1 - \mathfrak{K}_2\| t^{-1}\rr_t(y-x) \,. 
\end{align*}
\end{theorem}
\pf
(i) The first equality below follows 
from the strong continuity of the semigroup of a L{\'e}vy process and
Lemma~\ref{l:partial-time}(b).
Then by Lemma~\ref{lem:rozne_1},
\begin{align*}
p^{\mathfrak{K}_1}(t,x,y)-p^{\mathfrak{K}_2}(t,x,y)
&= 
\lim_{\varepsilon_1,\varepsilon_2 \to 0^+ } \int_{\varepsilon_1}^{t-\varepsilon_2}
\frac{d}{ds} \left( \int_{\Rd} p^{\mathfrak{K}_1}(s,x,z) p^{\mathfrak{K}_2}(t-s,z,y)dz\right) ds\\
&= 
\lim_{\varepsilon_1 \to 0^+ } \int_{\varepsilon_1}^{t/2}
\int_{\Rd} p^{\mathfrak{K}_1}(s,x,z) \left( \LL_z^{\mathfrak{K}_1} 
-  \LL_z^{\mathfrak{K}_2}\right) p^{\mathfrak{K}_2}(t-s,z,y)\,dzds\\
&+
\lim_{\varepsilon_2\to 0^+ } \int_{t/2}^{t-\varepsilon_2}
\int_{\Rd} \left( \LL_x^{\mathfrak{K}_1} 
-  \LL_x^{\mathfrak{K}_2}\right) p^{\mathfrak{K}_1}(s,x,z)  p^{\mathfrak{K}_2}(t-s,z,y)\,dzds
\,.
\end{align*}
By Proposition~\ref{prop:gen_est}, \eqref{eq:delta_gen}, Theorem~\ref{thm:delta}, Corollary~\ref{cor:3Pclass} and Lemma~\ref{lem:integr_rr},
\begin{align*}
&\int_{\varepsilon}^{t/2}
\int_{\Rd}  p^{\mathfrak{K}_1}(s,x,z)\,  | \!\left( \LL_z^{\mathfrak{K}_1} 
-  \LL_z^{\mathfrak{K}_2}\right) p^{\mathfrak{K}_2}(t-s,z,y)|\,dzds\\
&\leq
c \|\mathfrak{K}_1-\mathfrak{K}_2 \|
\int_{\varepsilon}^{t/2}
\int_{\Rd} \rr_s (z-x) \left( \int_{\Rd} | \delta^{\mathfrak{K}_2}(t-s,z,y;w)| \nu(|w|)dw\right)dzds\\
&\leq
c \|\mathfrak{K}_1-\mathfrak{K}_2 \|
\int_{\varepsilon}^{t/2}
\int_{\Rd} \rr_s (z-x)\, (t-s)^{-1}\rr_{t-s}(y-z) \,dzds\\
&\leq c \|\mathfrak{K}_1-\mathfrak{K}_2 \|
\rr_t(y-x) \int_{\varepsilon}^{t/2} t^{-1}ds\,.
\end{align*}
Similarly,
\begin{align*}
&\int_{t/2}^{t-\varepsilon}
\int_{\Rd} |\!\left( \LL_x^{\mathfrak{K}_1} 
-  \LL_x^{\mathfrak{K}_2}\right) p^{\mathfrak{K}_1}(s,x,z) |\,  p^{\mathfrak{K}_2}(t-s,z,y)\,dzds\\
&\leq
c \|\mathfrak{K}_1-\mathfrak{K}_2 \|
\int_{t/2}^{t-\varepsilon}
\int_{\Rd} \left( \int_{\Rd} | \delta^{\mathfrak{K}_1}(s,x,z;w)| \nu(|w|)dw\right) \rr_{t-s} (y-z) \,dzds\\
&\leq
c \|\mathfrak{K}_1-\mathfrak{K}_2 \|
\int_{t/2}^{t-\varepsilon}
\int_{\Rd} s^{-1}\rr_s(z-x) \rr_{t-s}(y-z) \,dzds
\leq c \|\mathfrak{K}_1-\mathfrak{K}_2 \|
\rr_t(y-x)\,.
\end{align*}
(ii) Define $\widehat{\mathfrak{K}}_i(z)=\mathfrak{K}_i(z)-\kappa_0/2$, $i=1,2$. By the construction of the L{\'e}vy process we have
\begin{align}\label{eq:przez_k_0}
p^{\mathfrak{K}_i}(t,x,y)=\int_{\Rd} p^{\kappa_0/2}(t,x,w)
p^{\widehat{\mathfrak{K}}_i}(t,w,y)\,dw\,.
\end{align}
By Proposition~\ref{prop:gen_est} we can differentiate under the integral in \eqref{eq:przez_k_0}.
Together with Corollary~\ref{cor:3Pclass} and Lemma~\ref{lem:integr_rr}
we obtain
\begin{align*}
|\nabla_x p^{\mathfrak{K}_1}(t,x,y)-\nabla_x p^{\mathfrak{K}_2}(t,x,y)| & =  \left|\int_{\Rd}\nabla_x p^{\kappa_0/2}(t, x,w) \left(p^{\widehat{\mathfrak{K}}_1}(t,w,y)-p^{\widehat{{\mathfrak{K}}}_2}(t,w,y)\right)dw\right|\\
& \leq   c \|\mathfrak{K}_1-\mathfrak{K}_2\| \int_{\Rd}
 \left[h^{-1}(1/t)\right]^{-1} \rr_t(w-x)\rr_t (y-w)\, dw\\
&\leq c \|{\mathfrak{K}}_1-{\mathfrak{K}}_2\|  \left[h^{-1}(1/t)\right]^{-1} \rr_{2t}(y-x)\\
&\leq c \|{\mathfrak{K}}_1-{\mathfrak{K}}_2\| \left[h^{-1}(1/t)\right]^{-1} 2 \rr_t(y-x)\,.
\end{align*}
(iii) By \eqref{eq:przez_k_0} we have
\begin{align*}
|\delta^{\mathfrak{K}_1} (t,x,y;z)-\delta^{\mathfrak{K}_2} (t,x,y;z)| 
= 
\left| \int_{\Rd} \delta^{\kappa_0/2}(t,x,w;z) 
\left(p^{\widehat{\mathfrak{K}}_1}(t,w,y)-p^{\widehat{{\mathfrak{K}}}_2}(t,w,y)\right)
 dw\right|.
\end{align*}
Then by Theorem~\ref{thm:delta}, Corollary~\ref{cor:3Pclass} and Lemma~\ref{lem:integr_rr},
\begin{align*}
\int_{\Rd} &|\delta^{\mathfrak{K}_1} (t,x,y;z)-\delta^{\mathfrak{K}_2} (t,x,y;z)| \,\nu(|z|)dz \\
&\leq 
 \int_{\Rd} \left( \int_{\Rd}  |\delta^{\kappa_0/2}(t,x,w;z)| \,\nu(|z|)dz\right)
\left| p^{\widehat{\mathfrak{K}}_1}(t,w,y)-p^{\widehat{{\mathfrak{K}}}_2}(t,w,y)\right|
 dw\\
&\leq 
c \|{\mathfrak{K}}_1-{\mathfrak{K}}_2\| \int_{\Rd} t^{-1}\rr_t(w-x)
\rr_t(y-w)\, dw \leq c \|{\mathfrak{K}}_1-{\mathfrak{K}}_2\| t^{-1} \rr_{2t}(y-x)\,.
\end{align*}
\qed

\section{Levi's construction of heat kernels}\label{sec:construction}

For a fixed $w\in \R^d$, define $\mathfrak{K}_w(z)=\kappa(w,z)$
and let $p^{\mathfrak{K}_w}(t,x,y)$ be the heat kernel of the operator $\LL^{{\mathfrak K}_w}$ as introduced in Section~\ref{sec:analysis_LL}.
For all $t>0$, $x,y,w\in\Rd$,
\begin{equation}\label{eq:p_gen_klas}
\partial_t p^{\mathfrak{K}_w}(t,x,y)= \LL_x^{\mathfrak{K}_w} p^{\mathfrak{K}_w}(t,x,y)\,,
\end{equation}
where for every $v\in\Rd$ we have
$
\LL_x^{\mathfrak{K}_v} p^{\mathfrak{K}_w}(t,x,y) =\int_{\Rd} \delta^{\mathfrak{K}_w}(t,x,y;z)  \kappa(v,z)J(z)dz
$
(see~\eqref{eq:delta_gen}).

\subsection{Properties of $p^{\mathfrak{K}_y}(t,x,y)$}\label{subsec:p_freeze}

\begin{lemma}\label{l:jcontoffzkernel}
The functions $p^{\mathfrak{K}_w}(t,x,y)$ and $\nabla_x p^{\mathfrak{K}_w}(t,x,y)$ are jointly continuous in $(t, x, y,w) \in (0,\infty)\times (\Rd)^3$.
\end{lemma}
\pf
Recall that $p^{\mathfrak{K}_w}(t,x,y)=p^{\mathfrak{K}_w}(t,y-x)$. By the triangle inequality,
$$
|p^{\mathfrak{K}_{w}}(t,x)-p^{\mathfrak{K}_{w_0}}(t_0,x_0)|\leq 
|p^{\mathfrak{K}_{w}}(t,x)-p^{\mathfrak{K}_{w_0}}(t,x)|
+
|p^{\mathfrak{K}_{w_0}}(t,x)-p^{\mathfrak{K}_{w_0}}(t_0,x_0)|\,.
$$
The first term is small by Theorem~\ref{thm:cont_kappa} and \eqref{e:intro-kappa-holder}, and the second by Lemma~\ref{l:partial-time}.
Similar proof is valid for
$\nabla_x p^{\mathfrak{K}_w}(t,x,y)$.
\qed

\begin{lemma}\label{lem:cont_Lv_pw}
The function $\LL_x^{\mathfrak{K}_{v}} p^{\mathfrak{K}_{w}}(t,x,y)$ is jointly continuous in $(t,x,y,w,v)\in (0,\infty)\times  (\Rd)^4$.
\end{lemma}
\pf
By Lemma \ref{l:jcontoffzkernel} the function
$\delta^{\mathfrak{K}_{w}}(t,x,y; z)$ is jointly continuous in $(t,x,y,w)$.
Recall that $\kappa(v,z)$ is continuous in $v$ and bounded.
We let 
$(t_n,x_n,y_n,w_n,v_n) \to (t,x,y,w,v)$ such that $0<\varepsilon \leq t_n\leq T$.
Further, by Lemma~\ref{lem:est_delta_3}, \ref{lem:diff-HK}, \ref{lem:est_delta_1}
we have respectively,
\begin{align*}
|\delta_3^{\mathfrak{K}_{w_n}}(t_n,x_n,y_n;z)|&\leq c \rr_{\varepsilon}(0) 
\left[ \left(\frac{|z|}{h^{-1}(1/\varepsilon)}\right)^2 \land 1\right],\\
|\delta_2^{\mathfrak{K}_{w_n}}(t_n,x_n,y_n;z)|&\leq c \rr_{\varepsilon}(0) 
\left[ \left(\frac{|z|}{h^{-1}(1/\varepsilon)}\right) \land 1\right],\\
|\delta_1^{\mathfrak{K}_{w_n}}(t_n,x_n,y_n;z)|&\leq c \rr_{\varepsilon}(0) 
\left[ \ind_{|z|\geq 1\land h^{-1}(1/\varepsilon)}+\left(\frac{|z|}{h^{-1}(1/\varepsilon)} \right)^2 \ind_{|z|\leq 1} \right].
\end{align*}
Thus the sequence $\delta^{\mathfrak{K}_{w_n}}(t_n,x_n,y_n;z) \kappa(v_n,z) \nu(|z|)$
is bounded by an integrable function and we can use the dominated convergence theorem.
\qed

For $\gamma,\beta\in \R$ we introduce the following function (see Appendix~\ref{subsec:conv})
\begin{align}\label{def:err}
\err{\gamma}{\beta}(t,x):= \left[h^{-1}(1/t)\right]^{\gamma} \left(|x|^{\beta}\land 1\right) t^{-1} \rr_t(x)\,.
\end{align}

\begin{lemma}\label{lem:pkw_holder}
For every $T>0$ there exists a constant
$c=c(d,T,\param)$
such that for all $t\in(0,T]$, $x,x',y,w \in \Rd$ and $\gamma\in [0,1]$,
\begin{align*}
|p^{\mathfrak{K}_w}(t,x,y)-p^{\mathfrak{K}_w}(t,x',y) | 
\leq c (|x-x'|^{\gamma}\land 1) \,t \left( \err{-\gamma}{0} (t,x-y)+ \err{-\gamma}{0}(t,x'-y) \right).
\end{align*}
\end{lemma}

\pf
We use Lemma~\ref{lem:diff-HK} and $( |x-x'|/h^{-1}(1/t) \land 1)\leq (|x-x'|^{\gamma}\land 1) \left[ h^{-1}(1/t)\right]^{-\gamma}
\left[ h^{-1}(1/T)\vee 1\right]$.

\qed

The following result is the counterpart of \cite[Lemma 3.2 and 3.3]{MR3500272}.
\begin{lemma}\label{l:some-estimates-3}
Let $\beta_1\in [0,\beta]\cap [0,\lah)$. For every $T>0$ there exists a constant $c=c(d,T,\param,\kappa_2,\beta_1)$
such that for all $t\in (0,T]$, $x,y,w\in\Rd$,
\begin{align}
\int_{\Rd} |\delta^{\mathfrak{K}_y} (t,x,y;z)|\,  \kappa(w,z) J(z)dz
&\leq c  \err{0}{0}(t,x-y)\,,  \label{l:some-estimates-3a} \\
\int_{\Rd} \left|\int_{\Rd} \delta^{\mathfrak{K}_y} (t,x,y;z) \,dy \right| \kappa(x,z) J(z)dz
&\leq c  t^{-1}\left[h^{-1}(1/t)\right]^{\beta_1}, \label{l:some-estimates-3b}
\end{align}
\begin{equation}\label{e:some-estimates-2bb}
 \left|\int_{\Rd} \nabla_x  p^{\mathfrak{K}_y} (t,x,y)\,dy \right|
\leq c\! \left[h^{-1}(1/t)\right]^{-1+\beta_1}\,.
\end{equation}
Furthermore,
\begin{align}\label{e:some-estimates-2c}
\lim_{t \to 0^+ } \sup_{x\in\Rd} \left| \int_{\Rd} p^{\mathfrak{K}_y}(t,x,y)\, dy -1\right|=0
\end{align}
\end{lemma}
\pf
The inequality \eqref{l:some-estimates-3a} follows from \eqref{e:intro-kappa}, \eqref{e:psi1} and Theorem~\ref{thm:delta}.
Let $I$ be the expression on the left hand side of \eqref{l:some-estimates-3b}. 
Since $\int_{\Rd}p^{\mathfrak{K}_w}(t,x,y)dy=1$ and
$\int_{\Rd} \partial_{x_i} p^{\mathfrak{K}_w}(t,x,y)dy=\partial_{x_i}\int_{\Rd}p^{\mathfrak{K}_w}(t,x,y)dy=0$ (see Lemma~\ref{lem:diff-HK}) we have
\begin{align}\label{eq:gen_zero}
\int_{\Rd} \delta^{\mathfrak{K}_w}(t,x,y;z)\,dy=0\,,\qquad x,w,z\in\Rd\,.
\end{align}
Then by \eqref{eq:gen_zero}, \eqref{e:psi1},
Theorem~\ref{thm:cont_kappa} and
Remark~\ref{rem:smaller_beta},
\begin{align*}
I&= \int_{\Rd} \left|\int_{\Rd} \left( \delta^{\mathfrak{K}_y} (t,x,y;z) - \delta^{\mathfrak{K}_x} (t,x,y;z)\right)dy \right| \kappa(x,z) J(z)dz\\
&\leq \int_{\Rd} \int_{\Rd} \left| \delta^{\mathfrak{K}_y} (t,x,y;z) - \delta^{\mathfrak{K}_x} (t,x,y;z)\right|   \kappa(x,z) J(z)dz dy\\
&\leq c
\int_{\Rd}  \|\kappa(y,\cdot)-\kappa(x,\cdot) \|  t^{-1} \rr_t(y-x)\, dy\\
&\leq c
\int_{\Rd}  (|y-x|^{\beta_1}\land 1)   t^{-1} \rr_t(y-x)\, dy\,.
\end{align*}
The result follows now from Lemma~\ref{l:convolution}(a).
For \eqref{e:some-estimates-2bb} by Theorem~\ref{thm:cont_kappa} and Lemma~\ref{l:convolution}(a),
\begin{align*}
&\left| \int_{\Rd} \nabla_x  p^{\mathfrak{K}_y} (t,x,y)\,dy \right|
=\left| \int_{\Rd} \left( \nabla_x  p^{\mathfrak{K}_y} (t,x,y) - \nabla p^{\mathfrak{K}_x} (t,\cdot,y)(x) \right)dy \right|\\
&\quad \leq c \int_{\Rd} \| \kappa(y,\cdot)-\kappa(x,\cdot) \| \left[h^{-1}(1/t)\right]^{-1} \rr_t(y-x) \,dy\\
&\quad \leq c \left[h^{-1}(1/t)\right]^{-1}  \int_{\Rd} (|y-x|^{\beta_1}\land 1) \rr_t(y-x)\,dy
\leq c \left[h^{-1}(1/t)\right]^{-1+\beta_1}\,.
\end{align*}
Eventually, by Theorem~\ref{thm:cont_kappa} and Lemma~\ref{l:convolution}(a),
\begin{align*}
&\sup_{x\in\Rd} \left| \int_{\Rd} p^{\mathfrak{K}_y}(t,x,y)\, dy -1\right|
\leq 
\sup_{x\in\Rd} \int_{\Rd} \left| p^{\mathfrak{K}_y}(t,x,y)- p^{\mathfrak{K}_x}(t,x,y) \right|dy \\
&\leq c 
\sup_{x\in\Rd} \int_{\Rd} \|\kappa(y,\cdot)-\kappa(x,\cdot) \| \rr_t(y-x) dy\\
&\leq c \sup_{x\in\Rd} \int_{\Rd} (|y-x|^{\beta_1}\land 1) \rr_t(y-x)\, dy
\leq c \left[h^{-1}(1/t)\right]^{\beta_1} \to 0\,,
\end{align*}
as $t\to 0^+$. This ends the proof.
\qed

\subsection{Construction of $q(t,x,y)$}\label{subsec:q}

For $(t,x,y)\in (0,\infty)\times \Rd\times \Rd$ define
\begin{eqnarray}
q_0(t,x,y):= \int_{\Rd}\delta^{\mathfrak{K}_y}(t,x,y;z)\left(\kappa(x,z)-\kappa(y,z)\right)J(z)dz\nonumber = \big(\LL_x^{{\mathfrak K}_x}-\LL_x^{{\mathfrak K}_y}\big) p^{\mathfrak{K}_y}(t,x,y)\, .\label{e:q0-definition}
\end{eqnarray}
Directly from Lemma~\ref{lem:cont_Lv_pw} we have the following result.
\begin{lemma}\label{l:joint-cont-q0}
The function $q_0(t,x,y)$ is jointly continuous in $(t,x,y)\in (0,\infty)\times \Rd \times \Rd$.
\end{lemma}
In the next lemma we collect estimates on $q_0$. 
\begin{lemma}\label{l:estimates-q0} 
For every $T>0$  there exists a constant $c=c(d,T,\param,\kappa_2)\geq 1$ such that for all 
$\beta_1\in [0,\beta]$,
$t\in (0,T]$ and $x,x',y,y'\in\Rd$
\begin{align}\label{e:q0-estimate}
|q_0(t,x,y)|\leq c (|y-x|^{\beta_1}\land 1) t^{-1}\rr_t(y-x)=c \err{0}{\beta_1}(t,y-x)\,,
\end{align}
and for every $\gamma\in [0,\beta_1]$, 
\begin{align}
&|q_0(t,x,y)-q_0(t,x',y)|\nonumber\\
&\leq c \left(|x-x'|^{\beta_1-\gamma}\land 1\right)\left\{\left(\err{\gamma}{0}+\err{\gamma-\beta_1}{\beta_1}\right)(t,x-y)
+\left(\err{\gamma}{0}+\err{\gamma-\beta_1}{\beta_1}\right)(t,x'-y)\right\},\label{e:estimate-step3}
\end{align}
and
\begin{align}
&|q_0(t,x,y)-q_0(t,x,y')|\nonumber \\
&\leq c \left(|y-y'|^{\beta_1-\gamma}\land 1\right)\left\{\left(\err{\gamma}{0}+\err{\gamma-\beta_1}{\beta_1}\right)(t,x-y)
+\left(\err{\gamma}{0}+\err{\gamma-\beta_1}{\beta_1}\right)(t,x-y')\right\}.
\label{e:estimate-q0-2}
\end{align}
\end{lemma}
\pf
(i) \eqref{e:q0-estimate} follows from 
\eqref{e:psi1}, Remark~\ref{rem:smaller_beta}
and Theorem~\ref{thm:delta}.\\
(ii) For $|x-x'|\geq 1$ the inequality holds by \eqref{e:q0-estimate} and \eqref{e:nonincrease-gamma}:
$$
|q_0(t,x,y)|\leq c \err{0}{\beta_1}(t,y-x) \leq c \left[ h^{-1}(1/T)\vee 1\right]^{\beta_1-\gamma}\err{\gamma-\beta_1}{\beta}(t,y-x)\,.
$$
For $1\geq |x-x'|\geq h^{-1}(1/t)$ the result follows from \eqref{e:q0-estimate} and
$$
|q_0(t,x,y)|\leq c \err{0}{\beta_1}(t,y-x) = c \left[ h^{-1}(1/t)\right]^{\beta_1-\gamma} \err{\gamma-\beta_1}{\beta_1}(t,y-x)\leq c |x-x'|^{\beta_1-\gamma} \err{\gamma-\beta_1}{\beta_1}(t,y-x)\,.
$$
Now, 
\eqref{e:psi1},  Remark~\ref{rem:smaller_beta} and Theorem~\ref{thm:delta} provide that
\begin{align*}
 & |q_0(t,x,y)-q_0(t,x',y)|=\left|\int_{\Rd} \delta^{\mathfrak{K}_y} (t,x,y;z)(\kappa(x,z)-\kappa(y,z))\,J(z)dz\right.\\
&  \hspace{0.1\linewidth} -\left. \int_{\Rd}\delta^{\mathfrak{K}_y}(t,x',y;z)(\kappa(x',z)-\kappa(y,z))\,J(z)dz\right|\\
& \hspace{0.05\linewidth} \leq  \lmCJ \int_{\Rd}|\delta^{\mathfrak{K}_y}(t,x,y;z)-\delta^{\mathfrak{K}_y}(t,x',y;z)|\, |\kappa(x,z)-\kappa(y,z)|\,\nu(|z|)dz\\
& \hspace{0.1\linewidth} + \lmCJ \int_{\Rd}| \delta^{\mathfrak{K}_y} (t,x',y;z)|\,|\kappa(x,z)-\kappa(x',z)|\, \nu(|z|)dz\\
& \hspace{0.05\linewidth} \leq c \left(|x-y|^{\beta_1}\land 1\right)\int_{\Rd}|\delta^{\mathfrak{K}_y}(t,x,y;z)-\delta^{\mathfrak{K}_y}(t,x',y;z)|\, \nu(|z|)dz\\
& \hspace{0.1\linewidth}  + c \left(|x-x'|^{\beta_1}\land 1\right)\int_{\Rd}|\delta^{\mathfrak{K}_y}(t,x',y;z)|\,\nu(|z|)dz\\
\leq c &\left(|x-y|^{\beta_1}\land 1\right) 
\left(\frac{|x-x'|}{h^{-1}(1/t)} \land 1\right)
\big(\err{0}{0} (t,x-y)+\err{0}{0}(t,x'-y)\big)+ c \left(|x-x'|^{\beta}\land 1\right) \err{0}{0}(t,x'-y).
\end{align*}
Applying 
$(|x-y|^{\beta_1}\land 1)\leq (|x-x'|^{\beta_1}\land 1) + (|x'-y|^{\beta_1}\land 1)$
we obtain
\begin{align*}
|q_0(t,x,y)-q_0(t,x',y)|\leq \ &c \left(\frac{|x-x'|}{h^{-1}(1/t)} \land 1\right)
\big(\err{0}{\beta_1} (t,x-y)+\err{0}{\beta_1}(t,x'-y)\big)\\
&+c \left(|x-x'|^{\beta_1}\land 1\right) \err{0}{0}(t,x'-y).
\end{align*}
Thus in the last case $|x-x'|\leq  h^{-1}(1/t)\land 1$ we have
$|x-x'|/ h^{-1}(1/t)\leq |x-x'|^{\beta_1-\gamma} \left[h^{-1}(1/t)\right]^{\gamma-\beta_1}$
and $|x-x'|^{\beta_1}\leq |x-x'|^{\beta_1 -\gamma} \left[h^{-1}(1/t)\right]^{\gamma}$.\\
(iii)
We treat the cases $|y-y'|\geq 1$ and $1\geq |y-y'|\geq h^{-1}(1/t)$ like in part (ii).
Now note that by
 \eqref{e:psi1},
 Remark~\ref{rem:smaller_beta}, 
$\delta^{\mathfrak{K}}(t,x,y;z)=\delta^{\mathfrak{K}}(t,-y,-x;z)$
 and Theorem~\ref{thm:delta} and~\ref{thm:cont_kappa},
\begin{align*}
&q_0(t,x,y)-q_0(t,x,y')\\
&=\int_{\Rd} \delta^{\mathfrak{K}_y}(t,x,y;z)\left(\kappa(y',z)-\kappa(y,z)\right)J(z)dz \\
& \ \ \ +\int_{\Rd}\left(\delta^{\mathfrak{K}_y}(t,x,y;z)-\delta^{\mathfrak{K}_y}(t,x,y';z)\right)\left(\kappa(x,z)-\kappa(y',z)\right)J(z)dz\\
&\ \ \ +\int_{\Rd}\left(\delta^{\mathfrak{K}_y}(t,x,y';z)-\delta^{\mathfrak{K}_{y'}}(t,x,y';z)\right)\left(\kappa(x,z)-\kappa(y',z)\right) J(z)dz\\
&\leq c \left( |y-y'|^{\beta_1}\land 1\right)\err{0}{0}(t,x-y)\\
&\quad +c  \left( |x-y'|^{\beta_1}\land 1\right)  \left(\frac{|y-y'|}{h^{-1}(1/t)} \land 1\right) \left(\err{0}{0}(t,x-y)+\err{0}{0}(t,x-y')\right)\\
&\quad +  c \left( |y-y'|^{\beta_1}\land 1\right)  \err{0}{0}(t,x-y') \,.
\end{align*}
Applying 
$(|x-y'|^{\beta_1}\land 1)\leq (|x-y|^{\beta_1}\land 1) + (|y-y'|^{\beta_1}\land 1)$
we obtain
\begin{align*}
|q_0(t,x,y)-q_0(t,x,y')|\leq \ & c
 \left(\frac{|y-y'|}{h^{-1}(1/t)} \land 1\right)
\big(\err{0}{\beta_1} (t,x-y)+\err{0}{\beta_1}(t,x-y')\big)\\
&+c \left(|y-y'|^{\beta_1}\land 1\right) \big( \err{0}{0}(t,x-y)+\err{0}{0}(t,x-y')\big).
\end{align*}
This proves \eqref{e:estimate-q0-2} in the case $|y-y'|\leq  h^{-1}(1/t)\land 1$.
\qed

For $n\in \N$ and $(t,x,y)\in (0, \infty)\times \Rd \times \Rd$ we inductively define
\begin{equation}\label{e:qn-definition}
q_n(t,x,y):=\int_0^t \int_{\Rd}q_0(t-s,x,z)q_{n-1}(s,z,y)\, dzds\, .
\end{equation}
The following result is the counterpart of \cite[Theorem 3.1]{MR3500272}
and \cite[Theorem~4.5]{KSV16}.

\begin{theorem}\label{t:definition-of-q}Assume \PG.
The series $q(t,x,y):=\sum_{n=0}^{\infty}q_n(t,x,y)$ is absolutely and locally uniformly convergent on $(0, \infty)\times \Rd \times \Rd$ and solves the integral equation
\begin{align}\label{e:integral-equation}
q(t,x,y)=q_0(t,x,y)+\int_0^t \int_{\Rd}q_0(t-s,x,z)q(s,z,y)\, dzds\, .
\end{align}
Moreover, 
for every $T> 0$ and $\beta_1\in (0,\beta]\cap (0,\lah)$ there is a constant $c=c(d,T,\param,\kappa_2, \beta_1)$
 such that on $ (0, T]\times \Rd \times \Rd$,
\begin{align}\label{e:q-estimate}
|q(t,x,y)|\leq c \big(\err{0}{\beta_1}+\err{\beta_1}{0}\big)(t,x-y)\,, 
\end{align}
and for any $\gamma\in (0,\beta_1]$ and $T>0$ there is a constant $c=c(d,T,\param,\kappa_2,\beta_1,\gamma)$ such that
on $(0, T]\times \Rd \times \Rd$,
\begin{align}
&|q(t,x,y)-q(t,x',y)|\nonumber\\
&\leq c \left(|x-x'|^{\beta_1-\gamma}\land 1\right)
\left\{\big(\err{\gamma}{0}+\err{\gamma-\beta_1}{\beta_1}\big)(t,x-y)+\big(\err{\gamma}{0}+\err{\gamma-\beta_1}{\beta_1}\big)(t,x'-y)\right\}\,,
\label{e:difference-q-estimate}
\end{align}
and
\begin{align}
&|q(t,x,y)-q(t,x,y')|\nonumber\\
&\leq c \left(|y-y'|^{\beta_1-\gamma}\land 1\right)
\left\{\big(\err{\gamma}{0}+\err{\gamma-\beta_1}{\beta_1}\big)(t,x-y)+\big(\err{\gamma}{0}+\err{\gamma-\beta_1}{\beta_1}\big)(t,x-y')\right\}\,.
\label{e:difference-q-estimate_1}
\end{align}

\end{theorem}
\pf
Let $T>0$ be fixed. In what follows $t\in (0,T]$, $x,y\in\Rd$.
We also fix $M>0$ such that $M < \beta_1$ and $\beta_1+M<\lah\land 1$.
Furthermore, we  set $\beta_0=\beta_1+M$ when using Lemma~\ref{l:convolution}.
For clarity we denote
$C_1=5 c  c_2$ and $C_2=h^{-1}(1/T)\vee 1$,
where $c$ and $c_2$ are taken from Lemma~\ref{l:estimates-q0} and~\ref{l:convolution}(b), respectively.

\noindent {\it Step 1.} 
First we justify that
\begin{align}\label{e:bdonq_n}
|q_n(t,x,y)|
\leq \gamma_n \big(\err{\beta_1+nM}{0}+\err{n M}{\beta_1}\big)(t,x-y)\,,
\end{align}
where
\begin{align*}
\gamma_n:= C_1^{n+1} C_2^{(\beta_1-M)n} \prod_{j=1}^n  B\left({M}/{2}, {jM}/{2}\right)
=C_1 \Gamma(M/2) \frac{\left(C_1 C_2^{\beta_1-M}  \Gamma(M/2)\right)^n}{\Gamma \left((n+1)M/2\right)}.
\end{align*}
For $n=1$ by \eqref{e:q0-estimate} and Lemma~\ref{l:convolution}(c) with $n_1=n_2=\beta_1+M$, $m_1=m_2=M$,
we have
\begin{align*}
|q_1(t,x,y)|&\leq c^2 \int_0^t\int_{\Rd} \err{0}{\beta_1} (t-s,x-z) \err{0}{\beta_1}(s,z-y)\,dzds\\
&\leq 2 c^2  c_2 \,B\!\left(M/2,M/2\right) \big( \err{\beta_1+M}{0}+\err{M}{\beta_1} \big)(t,x-y)\,.
\end{align*}
Further, assuming \eqref{e:bdonq_n} for $n\in \N$ 
we get by \eqref{e:q0-estimate},
Lemma~\ref{l:convolution}(c) with $n_1=n_2=m_1=M$, $m_2=0$
and
$n_1=n_2=\beta_1+M$, $m_1=m_2=M$,
by
the monotonicity of Beta function
and~\eqref{e:nonincrease-gamma},
\begin{align*}
|q_{n+1}(t,x,y)|&\leq c \,\gamma_n \int_0^t\int_{\Rd} \err{0}{\beta_1}(t-s,x,z) \big(\err{\beta_1+nM}{0}+\err{n M}{\beta_1} \big)(s,z,y)\,dzds\\
&\leq c \,\gamma_n  c_2\, B\!\left(M/2,(\beta_1+nM)/2\right)  \big( 3 \err{\beta_1+(n+1)M}{0}+\err{\beta_1+n M}{\beta_1} \big)(t,x-y)\\
&\quad+ c \,\gamma_n  c_2\, B\!\left(M/2,(n+1)M/2\right) \big( 2 \err{\beta_1+(n+1)M}{0}+2\err{(n+1) M}{\beta_1} \big)(t,x-y)\\
&\leq \gamma_n  5 c c_2 C_2^{\beta_1-M}\, B\!\left(M/2,(n+1)M/2\right)  \big( \err{\beta_1+(n+1)M}{0}+\err{(n+1) M}{\beta_1} \big)(t,x-y)\\
&\leq  \gamma_{n+1} \big( \err{\beta_1+(n+1)M}{0}+\err{(n+1) M}{\beta_1} \big)(t,x-y)\,.
\end{align*}
Thus \eqref{e:bdonq_n} follows by induction. Then by \eqref{e:nonincrease-gamma} we have
\begin{align*}
|q_n(t,x,y)|
\leq 
\gamma_n  \left[ h^{-1}(1/T)\right]^{nM} \big(\err{\beta_1}{0}+\err{0}{\beta_1}\big)(t,x-y)\,.
\end{align*}
Finally,
\begin{align*}
\sum_{n=0}^{\infty} |q_n(t,x,y)|
&\leq\left(C_1\Gamma(M/2) \sum_{n=0}^{\infty} \frac{ \left(C_1 C_2^{\beta_1} \,\Gamma(M/2)\right)^n}{\Gamma\left( (n+1)M/2\right)} \right) \big(\err{\beta_1}{0}+\err{0}{\beta_1}\big)(t,x-y) \,.
\end{align*}
Now, the series defining $q$ is absolutely and uniformly convergent on $[\varepsilon,T]\times\Rd\times\Rd$
and has the desired bound \eqref{e:q-estimate}. The equation \eqref{e:integral-equation} follows from the definition of $q_n$.

\noindent {\it Step 2.}
By \eqref{e:estimate-step3}, \eqref{e:q-estimate} and  Lemma~\ref{l:convolution}(c) with the usual parameters and once with
$n_1=n_2=m_1=m_2=\beta_1$,
\begin{align*}
&\int_0^t\int_{\Rd} |q_0(t-s,x,z)-q_0(t-s,x',z)| |q(s,z,y)| \,dzds
\leq c \left(|x-x'|^{\beta_1-\gamma}\land 1\right)\\
&\quad\times \int_0^t\int_{\Rd}
\left\{\left(\err{\gamma}{0}+\err{\gamma-\beta_1}{\beta_1}\right)(t,x-z)
+\left(\err{\gamma}{0}+\err{\gamma-\beta_1}{\beta_1}\right)(t,x'-z)\right\}
\big(\err{0}{\beta_1}+\err{\beta_1}{0}\big)(t,z-y)\,dzds\\
&\leq c \left(|x-x'|^{\beta_1-\gamma}\land 1\right)
\left\{\left(\err{\gamma+\beta_1}{0}+\err{\gamma}{0}+\err{\gamma}{\beta_1}\right)(t,x-y)
+\left(\err{\gamma+\beta_1}{0}+\err{\gamma}{0}+\err{\gamma}{\beta_1}\right)(t,x'-y)\right\}.
\end{align*}
Finally, we use \eqref{e:integral-equation},  \eqref{e:estimate-step3} and the above together with \eqref{e:nonincrease-gamma}.

\noindent {\it Step 3.}
In order to prove
\eqref{e:difference-q-estimate_1}, similarly to {\it Step 1}, using induction we get by 
\eqref{e:estimate-q0-2}, Lemma~\ref{l:convolution}(c) and 
\eqref{e:nonincrease-gamma},
\begin{align*}
&|q_n(t,x,y)-q_n(t,x,y')|\\
&\leq \gamma_n'  \left(|y-y'|^{\beta_1-\gamma}\land 1\right)
 \left\{ \big(\err{\gamma+nM}{0}+\err{\gamma-\beta_1 +nM}{\beta_1} \big)(t,x-y)+\big(\err{\gamma+nM}{0}+\err{\gamma-\beta_1 +nM}{\beta_1} \big)(t,x-y')  \right\},
\end{align*}
where $\gamma_n'= C_1^{n+1}  C_2^{(\beta_1-M)n} (C_hC_2^2 )^{n(\beta_1-\gamma)/\lah}  \prod_{j=1}^n  B\left(\beta_1/{2},{(\gamma+(j-1)M})/{2}\right)$.
\qed

\subsection{Properties of $\phi_y(t,x)$}\label{subsec:phi}
Let
\begin{equation}\label{e:phi-y-def}
\phi_y(t,x,s):=\int_{\Rd} p^{\mathfrak{K}_z}(t-s,x,z)q(s,z,y)\, dz, \quad x \in \Rd, \, 
0\leq s<t\,,
\end{equation}
and
\begin{equation}\label{e:def-phi-y-2}
\phi_y(t,x):=\int_0^t \phi_y(t,x,s)\, ds =\int_0^t \int_{\Rd}p^{\mathfrak{K}_z}(t-s,x,z)q(s,z,y)\, dzds\, .
\end{equation}

\begin{lemma}\label{lem:phi_cont_xy}
Let $\beta_1\in (0,\beta]\cap (0,\lah)$.
For every $T>0$ there exists a constant $c=c(d,T,\param,\kappa_2,\beta_1)$
such that for all $t\in (0,T]$, $x,y\in\Rd$,
\begin{align*}
|\phi_y(t,x)|\leq c t \big(\err{0}{\beta_1}+\err{\beta_1}{0}\big)(t,x-y)\,.
\end{align*}
For any $T>0$ and $\gamma \in [0,1]\cap [0,\lah)$ there exists a constant $c=c(d,T,\param,\kappa_2,\beta_1,\gamma)$ such that
for all $t\in (0,T]$, $x,x',y\in \Rd$,
\begin{align*}
|\phi_{y}(t,x)-\phi_{y}(t,x')|&\leq c (|x-x'|^{\gamma}\land 1) \, t \left\{ \big( \err{\beta_1-\gamma}{0}+\err{-\gamma}{\beta_1}\big)(t,x-y)+ \big( \err{\beta_1-\gamma}{0}+\err{-\gamma}{\beta_1}\big)(t,x'-y) \right\}.
\end{align*}
For any $T>0$ and $\gamma \in (0,\beta_1]$ there exists a constant $c=c(d,T,\param,\kappa_2,\beta_1,\gamma)$ such that
for all $t\in (0,T]$, $x,y,y'\in \Rd$,
\begin{align*}
|\phi_{y}(t,x)-\phi_{y'}(t,x)|&\leq c (|y-y'|^{\beta_1-\gamma}\land 1)\, t \left\{ \big( \err{\gamma}{0}+\err{\gamma-\beta_1}{\beta_1}\big)(t,x-y)+ \big( \err{\gamma}{0}+\err{\gamma-\beta_1}{\beta_1}\big)(t,x-y') \right\}.
\end{align*}
\end{lemma}

\pf
By Proposition~\ref{prop:gen_est} and~\eqref{e:q-estimate},
\begin{align*}
|\phi_y(t,x)|\leq c \int_0^t \int_{\Rd} (t-s)\err{0}{0}(t-s,x-z) \big(\err{0}{\beta_1}+\err{\beta_1}{0}\big)(s,z-y)  \,.
\end{align*}
By Lemma~\ref{lem:pkw_holder} and \eqref{e:q-estimate},
\begin{align*}
&|\phi_{y}(t,x)-\phi_{y}(t,x')|
\leq
\int_0^t \int_{\Rd} \left| p^{\mathfrak{K}_z}(t-s,x,z)-p^{\mathfrak{K}_z}(t-s,x',z) \right| q(s,z,y) \,dzds\\
&\leq c (|x-x'|^{\gamma}\land 1) \int_0^t \int_{\Rd} 
(t-s) \left( \err{-\gamma}{0} (t-s,x-z)+ \err{-\gamma}{0}(t-s,x'-z) \right)\\
&\hspace{0.55\linewidth}
\big(\err{0}{\beta_1}+\err{\beta_1}{0}\big)(s,z-y)  \,dzds\,.
\end{align*}
By Proposition~\ref{prop:gen_est} and  \eqref{e:difference-q-estimate_1}
\begin{align*}
&|\phi_{y}(t,x)-\phi_{y'}(t,x)|\leq
 \int_0^t \int_{\Rd}  p^{\mathfrak{K}_z}(t-s,x,z)  | q(s,z,y)- q(s,z,y')| \,dzds\\
&\leq c (|y-y'|^{\beta_1-\gamma}\land 1) \int_0^t\int_{\Rd} (t-s)\err{0}{0}(t-s,x-z) 
\left\{\big(\err{\gamma}{0}+\err{\gamma-\beta_1}{\beta_1}\big)(s,z-y)\right.\\
&\hspace{0.53\linewidth}\left.+\big(\err{\gamma}{0}+\err{\gamma-\beta_1}{\beta_1}\big)(s,z-y')\right\}dzds
\end{align*}
Finally, the results follow from Lemma~\ref{l:convolution}(c).
\qed

\begin{lemma}\label{lem:phi_cont_joint}
The function $\phi_y(t,x)$ is jointly continuous in $(t,x,y)\in (0,\infty)\times \Rd \times \Rd$.
\end{lemma}

\pf
First we prove the continuity in $t$ variable for fixed $x,y\in\Rd$.
We have for all $\varepsilon \in (0,t)$ and $|h|<\varepsilon/2$,
\begin{align*}
\phi_y(t+h,x)-\phi_y(t,x)
&=\int_0^{t-\varepsilon}\left( \phi_y(t+h,x,s)- \phi_y(t,x,s)\right)ds\\
&\quad + \int_{t-\varepsilon}^{t+h} \phi_y(t+h,x,s)\,ds
- \int_{t-\varepsilon}^t\phi_y(t,x,s)\,ds\,.
\end{align*}
By Proposition~\ref{prop:gen_est}, \eqref{e:q-estimate} and \eqref{e:nonincrease-t},
for $s\in (0,t-\varepsilon)$ we get
\begin{align*}
p^{\mathfrak{K}_z}(t-s+h,x,z) |q(s,z,y)|
\leq c \,2t \,\err{0}{0}(\varepsilon/2,0) \big(\err{0}{\beta_1}+\err{\beta_1}{0}\big)(s,z-y)\,.
\end{align*}
The right hand side is by Lemma~\ref{l:convolution}(a) and~\ref{l:convoluton-inequality} integrable over $(0,t-\varepsilon)\times \Rd$ in $dzds$.
Thus
by Lemma~\ref{l:jcontoffzkernel}
and
the dominated convergence theorem we have 
$\lim_{h\to 0} \int_0^{t-\varepsilon} \phi_y(t+h,x,s)\,ds=
\int_0^{t-\varepsilon} \phi_y(t,x,s)\,ds$
for every $\varepsilon\in (0,t)$.
Next,
we show that given $\varepsilon_1>0$ there exists $\varepsilon \in(0,t/3)$ such that for all $r\in\RR$ satisfying $|r-t|< \varepsilon/2$,
\begin{align*}
\int_{t-\varepsilon}^r \left| \phi_y(r,x,s)\right| ds <\varepsilon_1\,.
\end{align*}
Indeed, by \eqref{e:q-estimate}, the monotonicity of $h^{-1}$
and \eqref{e:nonincrease-t} in the first inequality, and  Proposition~\ref{prop:gen_est} and Lemma~\ref{lem:integr_rr}
in the second inequality, we have
\begin{align*}
\int_{t-\varepsilon}^r \left| \phi_y(r,x,s)\right| ds 
\leq c \left(\int_{t-\varepsilon}^r \int_{\Rd} p^{\mathfrak{K}_z}(r-s,x,z)\,dz ds\right) \err{0}{0}(t/2,0) 
\leq   c\, 2\varepsilon \, \err{0}{0}(t/2,0)\,.
\end{align*}
This ends the proof of the continuity in $t>0$.
The joint continuity follows from 
$|\phi_y(t,x)-\phi_{y_0}(t_0,x_0)|\leq |\phi_y(t,x)-\phi_y (t,x_0)|+|\phi_{y}(t,x_0)-\phi_{y_0}(t,x_0)|+|\phi_{y_0}(t,x_0)-\phi_{y_0}(t_0,x_0)|$
and Lemma~\ref{lem:phi_cont_xy}.

\qed

The following result is the counterpart of \cite[Lemma 3.6]{MR3500272}.
\begin{lemma}\label{l:gradient-phi-y}
Assume that 
$1-\lah<\beta\land \lah$.
For every $T>0$ 
there exists a constant $c=c(d,T,\param,\kappa_2,\beta_1)$ such that
for all  $t \in(0,T]$, $x,y\in\Rd$,
\begin{align}\label{e:gradient-phi-y}
&\nabla_x\phi_y(t,x)=\int_0^t \int_{\Rd} \nabla_x p^{\mathfrak{K}_z}(t-s,x,z) q(s,z,y)\, dzds\,,\\
\nonumber\\
\label{e:gradient-phi-y-estimate}
&\left|\nabla_x\phi_y(t,x) \right|\leq c \!\left[ h^{-1}(1/t)\right]^{-1} t \,\err{0}{0}(t,x-y)\,.
\end{align}
\end{lemma}
\pf
Let $\beta_1\in (0,\beta]$ such that $1-\lah<\beta_1<\lah$ and $t\in (0,T]$.
We set $\beta_0=\beta_1$ when using Lemma~\ref{l:convolution}.
We first note that by Lemma~\ref{lem:diff-HK}, \eqref{e:q-estimate} and Lemma~\ref{l:convolution}(b)
for  $s\in (0,t)$,
\begin{align}\label{eq:grad_phi_pomoc}
\nabla_x \phi_y(t,x,s)=\int_{\Rd} \nabla_x p^{\mathfrak{K}_z}(t-s,x,z)q(s,z,y)\, dz\,,
\end{align}
and the above function is integrable in $x$ over $\Rd$.
Now, let 
$|\varepsilon| \leq h^{-1}(3/t)$
 and $\widetilde{x}=x+\varepsilon\theta e_i$.
We have
\begin{align*}
&{\rm I}_0=\left| \frac{1}{\varepsilon} (\phi_y(t,x+\varepsilon e_i,s)-\phi_y(t,x,s))\right|
= \left| \int_0^1 \int_{\Rd} \frac{\partial}{\partial x_i} p^{\mathfrak{K}_z}(t-s,\widetilde{x},z) q(s,z,y)\, dzd\theta \right|.
\end{align*}
For $s\in (0,t/2]$ we first use  Proposition~\ref{prop:gen_est} and \eqref{e:q-estimate}, then  Lemma~\ref{l:convolution}(b)
(once with $n_1=m_1=0$, $n_2=m_2=\beta_1$) and \eqref{e:nonincrease-beta}, and finally  the monotonicity of $h^{-1}$, $\Ab$ of Lemma~\ref{lem:equiv_scal_h} and Proposition~\ref{prop:small_shift}
to get
\begin{align*}
{\rm I}_0&\leq c \int_0^1  \int_{\Rd} (t-s) 
\err{-1}{0}(t-s, \widetilde{x}-z) 
 \big(\err{0}{\beta_1}+\err{\beta_1}{0}\big)(s,z-y)  \,dzd\theta\\
&\leq c \left[h^{-1}(1/(t-s))\right]^{-1} \left(\int_0^1 \err{0}{0}(t, \widetilde{x}-y) \,d\theta \right)
\left(1 
+\left[h^{-1}(1/s)\right]^{\beta_1}
 + (t-s) s^{-1}\left[h^{-1}(1/s)\right]^{\beta_1}\right) \\
&\leq c \left[ h^{-1}(1/t) \right]^{-1}
\err{0}{0}(t,x-y)
\left(1+ t \,s^{-1}\left[h^{-1}(1/s)\right]^{\beta_1}\right).
\end{align*}
For $s\in (t/2,t)$ 
we fix $\gamma>0$ such that $\beta_1-\gamma >(1-\lah)\vee 0$. Then
by Proposition~\ref{prop:gen_est}, \eqref{e:difference-q-estimate},
\eqref{e:some-estimates-2bb} and \eqref{e:q-estimate}
we have
\begin{align*}
{\rm I}_0& \leq \int_0^1 \int_{\Rd} \left|  \frac{\partial}{\partial x_i} p^{\mathfrak{K}_z}(t-s,\widetilde{x},z)\right| |q(s,z,y)-q(s,\widetilde{x},y)|\, dzd\theta\\
&\quad+ \int_0^1 \left| \int_{\Rd}   \frac{\partial}{\partial x_i} p^{\mathfrak{K}_z}(t-s,\widetilde{x},z)\,dz\right| |q(s,\widetilde{x},y)| \,d\theta\\
&\leq 
c \int_0^1 \int_{\Rd} (t-s) \err{-1}{\beta_1-\gamma} (t-s,\widetilde{x}-z) 
\big(\err{\gamma}{0}+\err{\gamma-\beta_1}{\beta_1}\big)(s,z-y)\,
 dzd\theta\\
&\quad+c \int_0^1 \left( \int_{\Rd} (t-s) \err{-1}{\beta_1-\gamma} (t-s,\widetilde{x}-z)\,dz\right) \big(\err{\gamma}{0}+\err{\gamma-\beta_1}{\beta_1}\big)(s,\widetilde{x}-y)\, d\theta
\\
&\quad+c \left[h^{-1}(1/(t-s))\right]^{-1+\beta_1} \int_0^1 \big(\err{0}{\beta_1}+\err{\beta_1}{0}\big)(s,\widetilde{x}-y)\,d\theta
=: {\rm I}_1+{\rm I}_2+{\rm I}_3\,.
\end{align*}
By  Lemma~\ref{l:convolution}(b) with $(n_1,n_2)=(\beta_1-\gamma,0)$ and $(n_1,n_2)=(\beta_1-\gamma,\beta_1)$, \eqref{e:nonincrease-beta}, 
the monotonicity of $h^{-1}$, $\Ab$ of Lemma~\ref{lem:equiv_scal_h} and Proposition~\ref{prop:small_shift} we get
\begin{align*}
{\rm I}_1&\leq c 
\left( \left[h^{-1}(1/(t-s)) \right]^{-1+\beta_1-\gamma}\left[h^{-1}(1/s) \right]^{\gamma}
+(t-s)\left[h^{-1}(1/(t-s))\right]^{-1}  s^{-1} \left[h^{-1}(1/s)\right]^{\gamma}
\right.\\
&\quad \left. 
+\left[ h^{-1}(1/(t-s)) \right]^{-1+\beta_1-\gamma}\left[h^{-1}(1/s) \right]^{\gamma-\beta_1}\right) \left(\int_0^1 \err{0}{0}(t,\widetilde{x}-y)\,d\theta\right) \\
&\leq c 
\left( \left[h^{-1}(1/(t-s)) \right]^{-1+\beta_1-\gamma}
+(t-s)\left[h^{-1}(1/(t-s))\right]^{-1}  t^{-1}\right.\\
&\quad \left.+\left[ h^{-1}(1/(t-s)) \right]^{-1+\beta_1-\gamma}\left[h^{-1}(1/t) \right]^{\gamma-\beta_1}\right) \err{0}{0}(t,x-y).
\end{align*}
Next, by  Lemma~\ref{l:convolution}(a),
 \eqref{e:nonincrease-beta}, \eqref{e:nonincrease-t}, the monotonicity of $h^{-1}$, $\Ab$ of Lemma~\ref{lem:equiv_scal_h} and Proposition~\ref{prop:small_shift},
\begin{align*}
{\rm I}_2
&\leq c 
 \left[h^{-1}(1/(t-s)) \right]^{-1+\beta_1-\gamma}\left( \left[h^{-1}(1/s) \right]^{\gamma} + \left[h^{-1}(1/s) \right]^{\gamma-\beta_1}\right)
\left(\int_0^1 \err{0}{0}(t/2,\widetilde{x}-y)\,d\theta\right)\\
&\leq c 
 \left[h^{-1}(1/(t-s)) \right]^{-1+\beta_1-\gamma} \left(1+\left[h^{-1}(1/t)
 \right]^{\gamma-\beta_1}\right)
\err{0}{0}(t,x-y).
\end{align*}
Similarly,
$
{\rm I}_3\leq c \left[h^{-1}(1/(t-s)) \right]^{-1+\beta_1}\err{0}{0}(t,x-y)
$.
Finally, the expression ${\rm I}_0$ is bounded by a function independent of $\varepsilon$, which 
by Lemma~\ref{l:convoluton-inequality} is integrable over $(0,t)$ in~$s$, since $\alpha_h>1/2$. 
Then \eqref{e:gradient-phi-y} follows by the dominated convergence theorem and~\eqref{eq:grad_phi_pomoc}.
More precisely, Lemma~\ref{l:convoluton-inequality} assures that uniformly in $\varepsilon$ we have
for $t\in (0,T]$,
$$
\int_0^t {\rm I}_0 \,ds\leq c\,  t \left[ h^{-1}(1/t)\right]^{-1} \left(1+
\left[h^{-1}(1/t) \right]^{\beta_1-\gamma}+
\left[h^{-1}(1/t) \right]^{\beta_1}
\right) \err{0}{0}(t,x-y) \,,
$$
which proves
\eqref{e:gradient-phi-y-estimate} due to monotonicity of $h^{-1}$.
\qed

\begin{lemma}\label{lem:some-est_gen_phi_xy}
Let $\beta_1\in (0,\beta]\cap (0,\lah)$.
For all $T>0$, $\gamma \in(0,\beta_1]$ 
there exist  constants $c_1=c_1(d,T,\param,\kappa_2,\beta_1)$ and $c_2=c_2(d,T,\param,\kappa_2,\beta_1,\gamma)$ such that
for all  $0<s<t\leq T$, $x,y\in\Rd$,
\begin{align}
\int_{\Rd}\left(\int_{\Rd} |\delta^{\mathfrak{K}_z} (t-s, x,z;w)||q(s,z, y)|
\,dz\right) \kappa(x,w)J(w) dw\nonumber \hspace{0.15\linewidth}\\  
 \leq c_1\int_{\Rd}\err{0}{0}(t-s, x-z)
\big(\err{0}{\beta_1}+\err{\beta_1}{0}\big)(s,z-y)\,dz\,,   \label{e:Fubini1} \\
\nonumber \\
\int_{\Rd}  \left| \int_{\Rd} \delta^{\mathfrak{K}_z}(t-s,x,z;w)q(s,z,y)\,dz \right| \kappa(x,w)J(w)dw
\leq c_2 \big( {\rm I}_1+{\rm I}_2+{\rm I}_3 \big), \label{ineq:some-est_gen_phi_xy}
\end{align}
where 
\begin{align*}
{\rm I}_1+{\rm I}_2+{\rm I}_3:=
& \int_{\Rd} \err{0}{\beta_1-\gamma}(t-s,x-z) \big(\err{\gamma}{0}+\err{\gamma-\beta_1}{\beta_1}\big)(s,z-y) \,dz \\
& + 
(t-s)^{-1} \left[h^{-1}(1/(t-s))\right]^{\beta_1-\gamma}
\big(\err{\gamma}{0}+\err{\gamma-\beta_1}{\beta_1}\big)(s,x-y) \\
& +  \, (t-s)^{-1}\left[h^{-1}(1/(t-s))\right]^{\beta_1}  \big(\err{0}{\beta_1}+\err{\beta_1}{0}\big)(s,x-y)\,.
\end{align*}
\end{lemma}
\pf
The inequality
 \eqref{e:Fubini1} follows from 
\eqref{l:some-estimates-3a}
 and \eqref{e:q-estimate}.
Next, 
let ${\rm I}_0$ be the left hand side of \eqref{ineq:some-est_gen_phi_xy}.
By
\eqref{e:difference-q-estimate},
Lemma~\ref{l:some-estimates-3}, \eqref{e:q-estimate},
and 
Lemma~\ref{l:convolution}(a),
\begin{align*}
{\rm I}_0&\leq \int_{\Rd}  \int_{\Rd} |\delta^{\mathfrak{K}_z}(t-s,x,z;w)| |q(s,z,y)-q(s,x,y)|\,dz\, \kappa(x,w)J(w)dw\\
&\quad + \int_{\Rd}  \left| \int_{\Rd} \delta^{\mathfrak{K}_z}(t-s,x,z;w) \,dz\right|\! \kappa(x,w)J(w)dw \, |q(s,x,y)|\\
&\leq c \int_{\Rd} \left(  \int_{\Rd} |\delta^{\mathfrak{K}_z}(t-s,x,z;w)|\kappa(x,w)J(w)dw \right) \left(|x-z|^{\beta_1-\gamma}\land 1\right) \big(\err{\gamma}{0}+\err{\gamma-\beta_1}{\beta_1}\big)(s,z-y) \,dz \\
&\quad + c \int_{\Rd} \left(  \int_{\Rd} |\delta^{\mathfrak{K}_z}(t-s,x,z;w)|\kappa(x,w)J(w)dw \right) \left(|x-z|^{\beta_1-\gamma}\land 1\right) dz\, \big(\err{\gamma}{0}+\err{\gamma-\beta_1}{\beta_1}\big)(s,x-y) \\
&\quad + c  \, (t-s)^{-1}\left[h^{-1}(1/(t-s))\right]^{\beta_1}  \big(\err{0}{\beta_1}+\err{\beta_1}{0}\big)(s,x-y)
\leq c ({\rm I}_1+{\rm I}_2+{\rm I}_3)\,.
\end{align*}
\qed

\begin{lemma}\label{lem:some-est_gen_phi_xy_1}
Let $\beta_1\in (0,\beta]\cap (0,\lah)$.
For every $T>0$
there exists a constant $c=c(d,T,\param,\kappa_2,\beta_1)$ such that
for all  $t\in(0,T]$, $x,y\in\Rd$,
\begin{align}
\int_{\Rd}  \int_0^t &\left| \int_{\Rd} \delta^{\mathfrak{K}_z}(t-s,x,z;w)q(s,z,y)\,dz \right| ds\,\kappa(x,w)J(w)dw 
\leq c \err{0}{0}(t,x-y)\,, \label{ineq:I_0_oszagorne} \\
\int_{\Rd}\int_{\Rd}  \int_0^t &\left| \int_{\Rd} \delta^{\mathfrak{K}_z}(t-s,x,z;w)q(s,z,y)\,dz \right| ds\,\kappa(x,w)J(w)dw \,dy
\leq c t^{-1} \left[h^{-1}(1/t)\right]^{\beta_1}\,. \label{ineq:I_0_oszagorne_1}
\end{align}
\end{lemma}
\pf
Let ${\rm I}_0$ be the left hand side of \eqref{ineq:some-est_gen_phi_xy}.
For $s\in (0,t/2]$ we use \eqref{e:Fubini1}, Lemma~\ref{l:convolution}(b) and the monotonicity of $h^{-1}$ to get
\begin{align*}
{\rm I}_0&\leq 
c  \left((t-s)^{-1}+(t-s)^{-1} \left[h^{-1}(1/s)\right]^{\beta_1} +s^{-1}\left[h^{-1}(1/s)\right]^{\beta_1}\right) \err{0}{0}(t,x-y)\\
&\leq c  \left(t^{-1}+s^{-1}\left[h^{-1}(1/s)\right]^{\beta_1}\right) \err{0}{0}(t,x-y)\,.
\end{align*}
For $s\in (t/2,t)$ we fix  $\gamma \in(0,\beta_1)$ and we use  \eqref{ineq:some-est_gen_phi_xy}.
Then by Lemma~\ref{l:convolution}(b)
with $(n_1,n_2)=(\beta_1-\gamma,0)$ and $(n_1,n_2)=(\beta_1-\gamma,\beta_1)$, \eqref{e:nonincrease-beta}, 
the monotonicity of $h^{-1}$ and $\Ab$ of Lemma~\ref{lem:equiv_scal_h},
\begin{align*}
{\rm I}_1 &\leq  c \left( (t-s)^{-1} \left[h^{-1}(1/(t-s))\right]^{\beta_1-\gamma}  \left[h^{-1}(1/s)\right]^{\gamma}
+s^{-1} \left[h^{-1}(1/s)\right]^{\gamma}\right.\\
&\quad  \left.+(t-s)^{-1}\left[h^{-1}(1/(t-s))\right]^{\beta_1-\gamma} \left[h^{-1}(1/s)\right]^{\gamma-\beta_1}\right) \err{0}{0}(t,x-y)\\
& \leq c \left( t^{-1} 
+(t-s)^{-1}\left[h^{-1}(1/(t-s))\right]^{\beta_1-\gamma} \left[h^{-1}(1/t)\right]^{\gamma-\beta_1}\right) \err{0}{0}(t,x-y)\,.
\end{align*}
Next, 
like above
\begin{align*}
{\rm I}_2 &\leq c \, (t-s)^{-1} \left[h^{-1}(1/(t-s))\right]^{\beta_1-\gamma} \left( \left[h^{-1}(1/s)\right]^{\gamma}+ \left[h^{-1}(1/s)\right]^{\gamma-\beta_1}\right) \err{0}{0}(s,x-y)\\
&\leq c\, (t-s)^{-1} \left[h^{-1}(1/(t-s))\right]^{\beta_1-\gamma} \left[h^{-1}(1/t)\right]^{\gamma-\beta_1} \err{0}{0}(t,x-y)\,.
\end{align*}
Similarly, 
${\rm I}_3\leq c \, (t-s)^{-1} \left[h^{-1}(1/(t-s))\right]^{\beta_1} \err{0}{0}(t,x-y)$.
Finally, by
Lemma~\ref{l:convoluton-inequality},
\begin{align*}
\int_0^t {\rm I}_0\,ds \leq c \left(1 +\left[h^{-1}(1/t)\right]^{\beta_1}\right) \err{0}{0}(t,x-y)\,,
\end{align*}
which proves \eqref{ineq:I_0_oszagorne}.
Now, by \eqref{ineq:some-est_gen_phi_xy}, Lemma~\ref{l:convolution}(a) and the monotonicity of $h^{-1}$,
\begin{align*}
\int_{\Rd}\big({\rm I}_1+{\rm I}_2+{\rm I}_3 \big)\, dy \leq c  (t-s)^{-1}\left[h^{-1}(1/(t-s))\right]^{\beta_1-\gamma} s^{-1}\left[h^{-1}(1/s)\right]^{\gamma}\,.
\end{align*}
The result follows by integration in $s$ and application of Lemma~\ref{l:convoluton-inequality}.
\qed

\begin{lemma}\label{l:L-on-phi-y}
We have for all  $t >0$, $x,y\in\Rd$,
\begin{equation}\label{e:L-on-phi-y}
\LL_x^{\mathfrak{K}_x} \phi_y(t,x)= \int_0^t \int_{\Rd} \LL_x^{\mathfrak{K}_x} p^{\mathfrak{K}_z}(t-s,x,z) q(s,z,y)\, dzds\,.
\end{equation}
Further, 
$\LL_x^{\mathfrak{K}_x} \phi_y(t,x)$  is jointly continuous in   $(t,x,y)\in (0,\infty)\times \Rd\times \Rd$.
\end{lemma}
\pf
(a)
By \eqref{e:phi-y-def}, and
\eqref{eq:grad_phi_pomoc} in the case $\Pa$,  
\begin{align}
\LL_x^{\mathfrak{K}_x}\phi_y(t,x,s)
=\int_{\Rd}  \left(\int_{\Rd} \delta^{\mathfrak{K}_z} (t-s, x,z;w) q(s,z, y)
\,dz\right) \kappa(x,w)J(w) dw\,. \label{e:L-on-phi-y2-first}
\end{align}
By finiteness of \eqref{e:Fubini1} and Fubini's theorem,
\begin{align}\label{e:L-on-phi-y2}
\LL_x^{\mathfrak{K}_x}\phi_y(t,x,s)
=\int_{\Rd} \LL_x^{\mathfrak{K}_x}p^{\mathfrak{K}_z}(t-s,x,z)q(s,z,y)\, dz\,.
\end{align}
Finally, by \eqref{e:def-phi-y-2}, and \eqref{e:gradient-phi-y} in the case $\Pa$,
in the first equality, and  \eqref{ineq:I_0_oszagorne}, \eqref{e:Fubini1} in the second
(allowing to change the order of integration twice) we prove \eqref{e:L-on-phi-y} as follows
\begin{align*}
\LL_x^{\mathfrak{K}_x} \phi_y(t,x)
&=\int_{\Rd} \left( \int_0^t \int_{\Rd} \delta^{\mathfrak{K}_z}(t-s,x,z;w)q(s,z,y)\,dzds\right) \kappa(x,w)J(w)dw\\
&=  \int_0^t \int_{\Rd}\left( \int_{\Rd} \delta^{\mathfrak{K}_z}(t-s,x,z;w)\, \kappa(x,w)J(w)dw\right) q(s,z,y)\,dzds\,.
\end{align*}
(b)
Fix $t_0>0$, $x_0,y_0\in\Rd$.
For clarity we define $f(t,x,y):=\LL_x^{\mathfrak{K}_x} \phi_y(t,x)$.
Note that by \eqref{e:L-on-phi-y} and \eqref{e:L-on-phi-y2}
we have
$
f(t,x,y)=
I(t,x,y)
+\int_0^{t_0-\varepsilon_0} \LL_x^{\mathfrak{K}_x}  \phi_y(t,x,s) \,ds
$, where 
$I(t,x,y):=\int_{t_0-\varepsilon_0}^t  \LL_x^{\mathfrak{K}_x} \phi_y(t,x,s) ds$.
First we show that given $\varepsilon >0$ there exists $\varepsilon_0 \in(0,t_0/3)$ such that for all $r\in\RR$ satisfying $|r-t_0|< \varepsilon_0/2$, and all $x,y\in\Rd$,
\begin{align}\label{ineq:small_cut}
|I(r,x,y)|<\varepsilon\,.
\end{align}
Indeed, by
\eqref{e:L-on-phi-y2-first},
\eqref{ineq:some-est_gen_phi_xy}, the estimates of ${\rm I}_1$, ${\rm I}_2$, ${\rm I}_3$ from the proof of 
Lemma~\ref{lem:some-est_gen_phi_xy_1},
the monotonicity of $h^{-1}$ and
\eqref{e:nonincrease-t},
\begin{align*}
|I(r,x,y)|&\leq c \int_{t_0-\varepsilon_0}^r \left(r^{-1}+ (r-s)^{-1} \left[h^{-1}(1/(r-s))\right]^{\beta_1-\gamma} 
\left[h^{-1}(1/r)\right]^{\gamma-\beta_1}
\right) ds \,\err{0}{0}(r,x-y)\\
&\leq c 
\int_0^{2\varepsilon_0} \left(2/t_0+ u^{-1} \left[h^{-1}(1/u)\right]^{\beta_1-\gamma} 
\left[h^{-1}(2/t_0)\right]^{\gamma-\beta_1}
\right) du \,\err{0}{0}(t_0/2,0)< \varepsilon\,.
\end{align*}
Using \eqref{ineq:small_cut},  \eqref{l:some-estimates-3a}, 
\eqref{e:nonincrease-t}, \eqref{e:difference-q-estimate_1} and Lemma~\ref{l:convolution}(a) we get for all $t>0$
satisfying $|t-t_0|<\varepsilon_0/2$, and all $x,y\in\Rd$,
\begin{align*}
|f(t,x,y)-f(t,x,y_0)|
&\leq 2\varepsilon+ \int_0^{t_0-\varepsilon_0}\int_{\Rd} \left| \LL_x^{\mathfrak{K}_x} p^{\mathfrak{K}_z}(t-s,x,z)\right||q(s,z,y)-q(s,z,y_0)|\, dzds\\
&\leq 2\varepsilon+  c \err{0}{0}(\varepsilon_0/2,0) \,(|y-y_0|^{\beta_1-\gamma}\land 1) \int_0^{t_0} s^{-1}\left[h^{-1}(1/s)\right]^{\gamma}ds\,.
\end{align*}
Again by \eqref{ineq:small_cut} we have 
for $t>0$, $|t-t_0|<\varepsilon_0/2$, $x\in\Rd$,
\begin{align*}
|f(t,x,y_0)-f(t_0,x_0,y_0)|\leq 2\varepsilon+\left| \int_0^{t_0-\varepsilon_0} \LL_x^{\mathfrak{K}_x}  \phi_{y_0}(t,x,s)\,ds -\int_0^{t_0-\varepsilon_0}\LL_x^{\mathfrak{K}_{x_0}}  \phi_{y_0}(t_0,x_0,s)\, ds \right|.
\end{align*}
By \eqref{l:some-estimates-3a},
\eqref{e:q-estimate}
 and \eqref{e:nonincrease-t}, 
for $s\in (0,t_0-\varepsilon_0)$, $x\in\Rd$ we get
$$
\left| \LL_x^{\mathfrak{K}_x} p^{\mathfrak{K}_z}(t-s,x,z) \right| |q(s,z,y_0)|
\leq c \,\err{0}{0}(\varepsilon_0/2,0) \big(\err{0}{\beta_1}+\err{\beta_1}{0}\big)(s,z-y_0)\,.
$$
By Lemma~\ref{l:convolution}(a) and~\ref{l:convoluton-inequality} the right hand side is  integrable over $(0,t_0-\varepsilon_0)\times \Rd$ in $dzds$. Thus
by \eqref{e:L-on-phi-y2}, Lemma~\ref{lem:cont_Lv_pw} and the dominated convergence theorem
$\lim_{t\to t_0} \int_0^{t_0-\varepsilon_0}\LL_x^{\mathfrak{K}_x} \phi_{y_0}(t,x,s)ds=
\int_0^{t_0-\varepsilon_0}\LL_x^{\mathfrak{K}_{x_0}} \phi_{y_0}(t_0,x_0,s)ds$.
Finally,
if ${(t,x,y)\to (t_0,x_0,y_0)}$, then
$$\lim|f(t,x,y)-f(t_0,x_0,y_0)|\leq \lim \(|f(t,x,y)-f(t,x,y_0)| +|f(t,x,y_0)-f(t_0,x_0,y_0)|\) \leq 4\varepsilon.$$
\qed

The following result is the counterpart of \cite[Lemma 3.5]{MR3500272}
and \cite[Lemma~4.6]{KSV16}.
\begin{lemma}\label{l:phi-y-abs-cont}
For all $t>0$, $x,y\in \Rd$, $x\neq y$, we have 
\begin{align}\label{e:phi-y-partial}
\phi_y(t,x) 
=\int_0^t \left(q(r,x,y)+  \int_0^r \int_{\Rd} \LL_x^{\mathfrak{K}_z} p^{\mathfrak{K}_z}(r-s,x,z) q(s,z,y)\, dzds\right) dr\,.
\end{align}
\end{lemma}
\pf
{\it Step 1:}
Note that for every $s\in (0,t)$ and all $x,y\in\Rd$,
\begin{align}\label{e:phi-y-partial-2}
\partial_t \phi_y(t,x,s)=\int_{\Rd} \partial_t p^{\mathfrak{K}_z}(t-s,x,z) q(s,z,y)\, dz\,.
\end{align}
The above follows from \eqref{e:phi-y-def}, the mean value theorem, \eqref{eq:p_gen_klas}, 
\eqref{l:some-estimates-3a}, \eqref{e:nonincrease-t}
and integrability of $| q(s,z,y)|$ in $z$ (see \eqref{e:q-estimate})
which justifies
the use the dominated convergence theorem.\\
{\it Step 2:} Let $T>0$. We prove that there exists $c>0$ such that for all $t\in (0,T]$, $x,y\in\Rd$,
\begin{align}\label{e:phi-y-partial-3}
\int_0^t \int_0^r |\partial_r \phi_y(r,x,s)|\,ds\,dr\leq c t \frac{K(|x-y|)}{|x-y|^d}\,.
\end{align}
By \eqref{e:phi-y-partial-2}, \eqref{eq:p_gen_klas} and \eqref{e:L-on-phi-y2} we get
$
\partial_r \phi_y(r,x,s)= \LL_x^{\mathfrak{K}_x}\phi_y(r,x,s)-\int_{\Rd} q_0(r-s,x,z) q(s,z,y)\, dz
$.
Next, applying \eqref{e:L-on-phi-y2-first} and \eqref{ineq:I_0_oszagorne} we have
$
\int_0^t \int_0^r |\LL_x^{\mathfrak{K}_x}\phi_y(r,x,s)|dsdr\leq c t K(|x-y|)/ |x-y|^d
$.
By \eqref{e:q0-estimate}, \eqref{e:q-estimate}, Lemma~\ref{l:convolution}(c) (once with $n_1=m_1=n_2=m_2=\beta_1$) and
by \eqref{e:nonincrease-gamma}, \eqref{e:nonincrease-beta},
$$
\int_0^t\left( \int_0^r \int_{\Rd} | q_0(r-s,x,z) q(s,z,y)|\, dzds\right)dr\leq c \int_0^t \err{0}{0}(r,x-y)\,dr\leq c t K(|x-y|)/ |x-y|^d\,.
$$
{\it Step 3:} We claim that for fixed $s>0$, $x,y\in\Rd$,
\begin{align}\label{e:phi-y-partial-4}
\lim_{t\to s^+}\phi_y(t,x,s)=q(s,x,y)\, .
\end{align}
In view of \eqref{e:phi-y-def} and \eqref{e:some-estimates-2c} it suffices to consider the following expression for $\delta>0$ as $t\to s^+$,
\begin{align*}
&\left| \int_{\Rd} p^{\mathfrak{K}_z}(t-s,x,z)\big( q(s,z,y)-q(s,x,y)\big)dz\right|\\
&\leq \int_{|x-z|<\delta} p^{\mathfrak{K}_z}(t-s,x,z)|q(s,z,y)-q(s,x,y)|\,dz\\
&+\int_{|x-z|\geq \delta} p^{\mathfrak{K}_z}(t-s,x,z)\big(|q(s,z,y)|+|q(s,x,y)|\big)dz=: {\rm I}_1+{\rm I}_2\,.
\end{align*}
By \eqref{e:difference-q-estimate} for any $\varepsilon>0$ there is $\delta>0$ such that $|q(s,z,y)-q(s,x,y)|<\varepsilon$ if $|z-x|<\delta$. Together with Proposition~\ref{prop:gen_est} and Lemma~\ref{lem:integr_rr}
we get ${\rm I}_1\leq c \varepsilon$.
By \eqref{e:q-estimate} there is $c>0$ such that $|q(s,z,y)|\leq c$ for all $z\in\Rd$. Then
by Proposition~\ref{prop:gen_est} we have
$${\rm I}_2\leq c (t-s) \int_{|x-z|\geq \delta}K(|x-z|)|x-z|^{-d}dz \leq c(t-s) h(\delta)\xrightarrow {t\to s^+}0 \,.$$
{\it Step 4:} Let $x\neq y$. By \eqref{eq:p_gen_klas} and \eqref{e:phi-y-partial-2} in the first equality,
\eqref{e:phi-y-partial-3} and Fubini's theorem in the second, \eqref{e:phi-y-partial-3} that allows to put the limit in the third equality, and \eqref{e:def-phi-y-2}, \eqref{e:phi-y-partial-4}, \cite[Theorem~7.21]{MR924157}  in the last,
\begin{align*}
&\int_0^t \int_0^r \int_{\Rd} \LL_x^{\mathfrak{K}_z} p^{\mathfrak{K}_z}(r-s,x,z) q(s,z,y)\, dzdsdr
=\int_0^t \int_0^r  \partial_r \phi_y(r,x,s)\, dsdr\\
&= \int_0^t \int_s^t  \partial_r \phi_y(r,x,s)\, drds= \int_0^t \lim_{\varepsilon\to 0+} \int_{s+\varepsilon}^t  \partial_r \phi_y(r,x,s)\, drds
=\phi_y(t,x) -\int_0^t q(s,x,y)\, ds\,.
\end{align*}
This ends the proof.
\qed
\begin{corollary}
For all $x,y\in\Rd$, $x\neq y$, the function $\phi_y(t,x)$ is differentiable in $t>0$  and
\begin{align}\label{e:phi-y-partial_1}
 \partial_t \phi_y(t,x)  =
q_0(t,x,y)+  \LL_x^{\mathfrak{K}_x} \phi_y (t,x)\,.
\end{align}
\end{corollary}
\pf
By \eqref{e:phi-y-partial}, \eqref{e:integral-equation} and
\eqref{e:L-on-phi-y}
we have
\begin{align*}
\phi_y(t,x) 
&=\int_0^t \left(q_0(r,x,y)+  \int_0^r \int_{\Rd} \LL_x^{\mathfrak{K}_x} p^{\mathfrak{K}_z}(r-s,x,z) q(s,z,y)\, dzds\right) dr\\
&=\int_0^t \left(q_0(r,x,y)+  \LL_x^{\mathfrak{K}_x} \phi_y (r,x) \right) dr\,.
\end{align*}
Lemma~\ref{l:joint-cont-q0} and~\ref{l:L-on-phi-y} assure that the integrand is continuous and the result follows.
\qed

\subsection{Properties of $p^\kappa(t, x, y)$}\label{subsec:p^K}

Now we define and study the function 
\begin{align}\label{e:p-kappa}
p^{\kappa}(t,x,y):=p^{\mathfrak{K}_y}(t,x,y)+\phi_y(t,x)=p^{\mathfrak{K}_y}(t,x,y)+\int_0^t \int_{\Rd}p^{\mathfrak{K}_z}(t-s,x,z)q(s,z,y)\, dzds\,.
\end{align}
\begin{lemma}\label{lem:some-est_p_kappa}
Let $\beta_1\in (0,\beta]\cap (0,\lah)$.
For every $T>0$ there exists a constant $c=c(d,T,\param,\kappa_2,\beta_1)$ such that for all $t\in (0,T]$, $x,y\in\Rd$,
\begin{align}
\int_{\Rd} | \delta^{\kappa}(t,x,y;z) | \,\kappa(x,z)J(z)dz &\leq c \err{0}{0}(t,x-y)\,, \label{ineq:some-est_p_kappa} \\
\int_{\Rd} \left|\int_{\Rd} \delta^{\kappa}(t,x,y;z)\, dy \right|\kappa(x,z)J(z)dz& \leq c t^{-1} \left[h^{-1}(1/t)\right]^{\beta_1}\,. \label{ineq:some-est_p_kappa_1}
\end{align}
\end{lemma}
\pf
By \eqref{e:p-kappa}, and \eqref{e:gradient-phi-y} in the case $\Pa$,
\begin{align*}
\delta^{\kappa}(t,x,y;w)=\delta^{\mathfrak{K}_y}(t,x,y;w)+\int_0^t\int_{\Rd} \delta^{\mathfrak{K}_z}(t-s,x,z;w)q(s,z,y)\,dzds\,.
\end{align*}
The inequalities result from Lemma~\ref{l:some-estimates-3} and~\ref{lem:some-est_gen_phi_xy_1}.
\qed

The following result is the counterpart  of \cite[Lemma~3.7 and~4.2]{MR3500272}.
\begin{lemma}\label{l:p-kappa-difference} 
\noindent 
(a)
The function $p^{\kappa}(t,x,y)$ 
is jointly continuous   
on $(0, \infty)\times \Rd \times \Rd$. 

\noindent
(b)  For every $T> 0$ there is a constant $c=c(d,T,\param,\kappa_2,\beta)$ such that for all $t\in (0,T]$ and $x,y\in \Rd$,
$$
|p^{\kappa}(t,x,y)|\leq c t \err{0}{0}(t,x-y).
$$

\noindent
(c)  For all $t>0$, $x,y\in\Rd$, $x\neq y$,  
$$
\partial_t p^{\kappa}(t,x,y)= \LL_x^{\kappa}\, p^{\kappa}(t,x,y)\,.
$$

\noindent
(d)
For every $T>0$ there is a constant $c=c(d,T,\param,\kappa_2,\beta)$ such that for all
$t\in (0,T]$, $x,y\in\Rd$ and 
 $\varepsilon \in [0,1]$,
\begin{align}\label{e:fract-der-p-kappa-1b}
|\LL_x^{\kappa, \varepsilon} p^{\kappa}(t, x, y)|\leq c \err{0}{0}(t,x-y)\,,
\end{align}
and if $1-\lah<\beta\land \lah$, then
 \begin{align}\label{e:fract-der-p-kappa-2}
\left|\nabla_x p^{\kappa}(t,x,y)\right|\leq  c\! \left[h^{-1}(1/t)\right]^{-1} t \err{0}{0}(t,x-y)\,. 
\end{align}

\noindent
(e) For all $T>0$, $\gamma \in [0,1]\cap [0,\lah)$,
there is a constant $c=c(d,T,\param,\kappa_2,\beta,\gamma)$ such that for all $t\in (0,T]$ and $x,x',y\in \Rd$,
\begin{align*}
\left|p^{\kappa}(t,x,y)-p^{\kappa}(t,x',y)\right| \leq c 
 (|x-x'|^{\gamma}\land 1) \,t \left( \err{-\gamma}{0} (t,x-y)+ \err{-\gamma}{0}(t,x'-y) \right).
\end{align*}

\noindent
For all $T>0$, $\gamma \in [0,\beta)\cap [0,\lah)$,
there is a constant $c=c(d,T,\param,\kappa_2,\beta,\gamma)$ such that for all $t\in (0,T]$ and $x,y,y'\in \Rd$,
\begin{align*}
\left|p^{\kappa}(t,x,y)-p^{\kappa}(t,x,y')\right| \leq c 
(|y-y'|^{\gamma}\land 1)\, t \left(  \err{-\gamma}{0}(t,x-y)+  \err{-\gamma}{0}(t,x-y') \right).
\end{align*}

\noindent\\
(f) The function $\LL_x^{\kappa}p^{\kappa}(t,x,y)$  is jointly continuous on $(0,\infty)\times \Rd\times \Rd$.
\end{lemma}
\pf 
The statement of (a) follows from Lemma~\ref{l:jcontoffzkernel} and~\ref{lem:phi_cont_joint}.
Part  (b) is a result of
Proposition~\ref{prop:gen_est}
and Lemma~\ref{lem:phi_cont_xy}.
 The equation  in (c) is a consequence of \eqref{e:p-kappa}, \eqref{eq:p_gen_klas} and \eqref{e:phi-y-partial_1}:
$\partial_t p^{\kappa}(t,x,y)=\LL_x^{\mathfrak{K}_x} p^{\mathfrak{K}_y}(t,x,y)+ \LL_x^{\mathfrak{K}_x} \phi_y(t,x)=\LL_x^{\mathfrak{K}_x} p^{\kappa}(t,x,y)$.
We get
\eqref{e:fract-der-p-kappa-1b} 
by \eqref{ineq:some-est_p_kappa}.
For the proof of \eqref{e:fract-der-p-kappa-2} we use Proposition~\ref{prop:gen_est} and  \eqref{e:gradient-phi-y-estimate}.
The first inequality of part  (e) follows from
 Lemma~\ref{lem:pkw_holder} and~\ref{lem:phi_cont_xy}, and \eqref{e:nonincrease-gamma}, \eqref{e:nonincrease-beta}.
The same argument suffices for the second inequality when supported by
$$
|p^{\mathfrak{K}_y}(t,x,y)-p^{\mathfrak{K}_{y'}}(t,x,y')|
\leq |p^{\mathfrak{K}_y}(t,-y,-x)-p^{\mathfrak{K}_{y}}(t,-y',-x)|
+ |p^{\mathfrak{K}_y}(t,x,y')-p^{\mathfrak{K}_{y'}}(t,x,y')|
$$
and Theorem~\ref{thm:cont_kappa}.
Part (f)
follows from 
Lemma~\ref{lem:cont_Lv_pw} and~\ref{l:L-on-phi-y}.
\qed

\section{Main Results}\label{sec:main}

\subsection{A nonlocal maximum principle}\label{sec:max}

Recall that $\LL^{\kappa,0^+}f:=\lim_{\varepsilon \to 0^+}\LL^{\kappa,\varepsilon}f$ is an extension of $\LL^{\kappa}f:=\LL^{\kappa,0}f$.
Moreover, in the case~$\Pa$, the well-posedness of those operators require the existence of the gradient $\nabla f$. 
\begin{theorem}\label{t:nonlocal-max-principle}
Assume $\PG$.
Let $T>0$ and   $u\in C([0,T]\times \Rd)$ be such that
\begin{align}\label{e:nonlocal-max-principle-1}
\| u(t,\cdot)-u(0,\cdot) \|_{\infty} \xrightarrow {t\to 0^+} 0\,, \qquad \qquad \sup_{t\in [0,T]} \|  u(t,\cdot)\ind_{|\cdot|\geq r} \|_{\infty} \xrightarrow {r\to \infty}0\,.
\end{align}
Assume that  $u(t,x)$ satisfies the following equation: for all $(t,x)\in (0,T]\times \Rd$,
\begin{align}\label{e:nonlocal-max-principle-4}
\partial_t u(t,x)=\LL_x^{\kappa,0^+}u(t,x)\, .
\end{align}
If $\sup_{x\in\Rd} u(0,x)\geq 0$, then
for every $t\in (0,T]$,
\begin{align}\label{e:nonlocal-max-principle-5}
\sup_{x\in \R^d}u(t,x)\leq \sup_{x\in \Rd}u(0,x)\, .
\end{align}
\end{theorem}
\pf
For arbitrary $\lambda >0$ we consider $\widetilde{u}(t,x)=e^{-\lambda t}u(t,x)$. Then for all $(t,x)\in (0,T]\times \Rd$,
\begin{align*}
\partial_t \widetilde{u}(t,x)=\big(\!\! -\lambda +\LL_x^{\kappa,0^+} \big) \, \widetilde{u}(t,x)\,.
\end{align*}
By letting $\lambda \to 0^+$, if suffices to prove 
that
\begin{align}\label{e:nonlocal-max-principle-51}
\sup_{x\in\Rd} \widetilde{u}(t,x) \leq \sup_{x\in\Rd} \widetilde{u}(0,x)= \sup_{x\in\Rd} u(0,x), \qquad \mbox{for every}\, t\in (0,T]\,.
\end{align}
Suppose that \eqref{e:nonlocal-max-principle-51} does not hold. Then $\widetilde{u}(t',x')>\sup_{x\in\Rd} \widetilde{u}(0,x)\geq 0$ for some $(t',x')\in (0,T]\times\Rd$. Thus by continuity and \eqref{e:nonlocal-max-principle-1}
the function $\widetilde{u}$ attains a positive maximum at some $(t_0,x_0)\in (0,T]\times \Rd$.
Consequently, $\partial_t \widetilde{u}(t_0,x_0)\geq 0$, $\LL_x^{\kappa,0^+}\widetilde{u}(t_0,x_0)\leq 0$ and
$$
0\leq \partial_t \widetilde{u}(t_0,x_0) = \big(\!\! -\lambda +\LL_x^{\kappa,0^+} \big) \, \widetilde{u}(t_0,x_0) \leq - \lambda  \widetilde{u}(t_0,x_0)\,,
$$
which is a contradiction.
\qed
\begin{corollary}\label{cor:jedn_max}
If $u_1, u_2 \in C([0,T]\times \Rd)$ satisfy  \eqref{e:nonlocal-max-principle-1}, \eqref{e:nonlocal-max-principle-4} 
 and $u_1(0,x)=u_2(0,x)$, then $u_1\equiv u_2$ on $[0,T]\times \Rd$.
\end{corollary}

\subsection{Properties of the semigroup  $(P^{\kappa}_t)_{t\ge 0}$}

Define
$$
P_t^{\kappa}f(x)=\int_{\R^d}p^\kappa(t,x, y)f(y)dy.
$$
We first collect some properties of $\rr_t*f$.
\begin{remark}\label{rem:conv_Lp}
We have $\rr_t*f \in C_b(\Rd)$
for any $f\in L^p(\Rd)$, $p\in [1,\infty]$.
Moreover,
$\rr_t*f\in C_0(\Rd)$ 
for any $f\in L^p(\Rd)\cup C_0(\Rd)$, $p\in [1,\infty)$.
Further, there is $c=c(d)$ such that
$\|\rr_t*f \|_p\leq c \|f\|_p$ for all $t>0$, $p\in [1,\infty]]$.
The above follows from $\rr_t\in L^1(\Rd)\cap L^{\infty}(\Rd)\subseteq L^q(\Rd)$ for every $q\in [1,\infty]$
(see  Lemma~\ref{lem:integr_rr}),
and from properties of the convolution.

\end{remark}
\begin{lemma}\label{lem:bdd_cont}
(a)
We have $P_t^{\kappa} f \in C_b(\Rd)$ for any $f\in L^p(\Rd)$, $p\in [1,\infty]$.
Moreover, $P_t^{\kappa} f \in C_0(\Rd)$
for any $f\in L^p(\Rd)\cup C_0(\Rd)$, $p\in [1,\infty)$.
For every $T>0$ there exists a constant $c=c(d,T,\param,\kappa_2,\beta)$ such that for all $t\in(0,T]$ we get
$$
\|P^{\kappa}_t f\|_p\leq c \|f\|_p\,.
$$
(b) $P^{\kappa}_t\colon C_0(\Rd)\to C_0(\Rd)$, $t>0$, and for any bounded uniformly continuous function $f$,
$$
\lim_{t\to 0^+} \|P^{\kappa}_t f -f \|_{\infty}=0\,.
$$
(c)
$P^{\kappa}_t\colon L^p(\Rd)\to L^p(\Rd)$, $t>0$, $p\in [1,\infty)$, and for any $f\in L^p(\Rd)$,
$$
\lim_{t\to 0^+} \|P_t^{\kappa}f -f \|_p=0\,.
$$
\end{lemma}
\pf
Part (a) follows from Remark~\ref{rem:conv_Lp} and 
Lemma~\ref{l:p-kappa-difference}.
It remains to prove the continuity as $t\to 0^+$.
We fix $T>0$ and let $t\in (0,T]$.
By Lemma~\ref{lem:phi_cont_xy}, Young's inequality  and Lemma~\ref{l:convolution}(a)
we have $$\|\int_{\Rd} \phi_y (t,\cdot)f (y)dy\|_p\leq c\left[h^{-1}(1/t)\right]^{\beta_1} \|f\|_p \xrightarrow {t\to 0^+}0\,.$$
Then by \eqref{e:p-kappa},
\begin{align*}
\|P^{\kappa}_tf-f\|_{p}&\leq \|\int_{\Rd} p^{\mathfrak{K}_y}(t,\cdot,y)\big[f(y)-f(\cdot)\big] dy \|_p
+ \sup_{x\in\Rd} \left| 
 \int_{\Rd}p^{\mathfrak{K}_y}(t,x,y) dy -1  \right| 
\|f\|_p\\
&\quad + c \left[ h^{-1}(1/t)\right]^{\beta_1}\|f\|_p\,,
\end{align*}
and by \eqref{e:some-estimates-2c} it suffices to consider the first term.
By Proposition~\ref{prop:gen_est}, Minkowski's integral inequality and Lemma~\ref{lem:integr_rr}, for $f_z(x):=f(x+z)$, we have
\begin{align*}
&\|\int_{\Rd} p^{\mathfrak{K}_y}(t,\cdot,y)\big[f(y)-f(\cdot)\big] dy \|_p
\leq c \int_{\Rd} \rr_t (z)  \|  f_z-f \|_p \, dz\\
& \leq c \int_{|z|< \delta} \rr_t(z) \|  f_z-f \|_p \, dz + 2 c\|f\|_p  \int_{|z|\geq \delta}t  K(|x|)|x|^{-d}\,dz
 \leq c \left( \varepsilon+t h(\delta) \|f\|_p \right) ,
\end{align*}
where $\delta>0$ is such that $\|f_z-f\|_p<\varepsilon$ for $|z|<\delta$. This ends the proof.
\qed
\begin{lemma}\label{lem:grad_Pt}
Assume $\Pa$. For any $f\in L^p(\Rd)$, $p\in [1,\infty]$,  we have for all $t>0$, $x\in\Rd$,
\begin{align}\label{eq:grad_Pt}
\nabla_x \,P_t^{\kappa} f(x)= \int_{\Rd} \nabla_x\, p^{\kappa}(t,x,y) f(y)dy\,,
\end{align}
and for any $f\in L^{\infty}(\Rd)$ and all $t>0$, $x\in\Rd$,
\begin{align}\label{eq:grad_Pt_1}
\nabla_x \left( \int_0^t P^{\kappa}_s f(x)\,ds \right)= \int_0^t \nabla_x  P^{\kappa}_s f(x)\,ds\,.
\end{align}
\end{lemma}
\pf
 By  \eqref{e:fract-der-p-kappa-2} 
and Corollary~\ref{cor:small_shift}
for $|\varepsilon|<h^{-1}(1/t)$,
\begin{align*}
 \left|  \frac1{\varepsilon}( p^{\kappa}(t,x+\varepsilon e_i,y)-p^{\kappa}(t,x,y)) \right| |f(y)| \leq c \left[h^{-1}(1/t)\right]^{-1}  \rr_t (x-y) |f(y)|\,.
\end{align*}
The right hand side is integrable by Remark~\ref{rem:conv_Lp}. We can use the dominated convergence theorem, which gives
\eqref{eq:grad_Pt}. Now, for $f\in L^{\infty}(\Rd)$,
\begin{align*}
&\int_{\Rd}
 \left| \frac1{\varepsilon}( p^{\kappa}(s,x+\varepsilon e_i,y)-p^{\kappa}(s,x,y)) \right| |f(y)|dy
\leq \int_{\Rd} \int_0^1\left| \partial_{x_i} p^{\kappa}(s,x+\theta\varepsilon e_i,y) \right| d\theta\, |f(y)|dy\\
&\leq 
c \left[h^{-1}(1/s)\right]^{-1} \int_0^1 \big( \rr_s * |f|\big) (x+\theta\varepsilon e_i)\,d\theta
\leq c \left[h^{-1}(1/s)\right]^{-1} \| \rr_s * |f| \|_{\infty}\,.
\end{align*}
The right hand side is bounded by
$c \left[h^{-1}(1/s)\right]^{-1} \|f\|_{\infty}$ (Remark~\ref{rem:conv_Lp}), which is intergrable over $(0,t)$ by
$\Ab$ of Lemma~\ref{lem:equiv_scal_h}
 and $\lah>1$. Finally, \eqref{eq:grad_Pt_1} follows by dominated convergence theorem.
\qed

\begin{lemma}\label{l:L-int-commute0}
For any function $f\in L^p(\Rd)$, $p\in [1,\infty]$, and all $t>0$, $x\in\Rd$,
\begin{align}\label{e:L-int-commute-2}
\LL_x^{\kappa}P_t^{\kappa} f(x)=\int_{\Rd}\LL_x^{\kappa} \,p^{\kappa}(t,x, y)f(y)dy\, .
\end{align}
Further, for every $T>0$ there exists a constant $c>0$ such that
for all $f\in L^p(\Rd)$, $t\in (0,T]$, 
\begin{align}\label{e:LP-p-estimate}
\| \LL^{\kappa}P_t^{\kappa} f\|_p\leq c t^{-1} \|f\|_p\,.
\end{align}
\end{lemma}
\pf
By the definition, and \eqref{eq:grad_Pt} in the case $\Pa$,
\begin{align}\label{eq:LPf}
\LL_x^{\kappa} P_t^{\kappa} f(x)
=\int_{\Rd} \left(  \int_{\Rd} \delta^{\kappa}(t,x,y;z)  f(y)dy \right)  \kappa(x,z)J(z)dz\,.
\end{align}
The equality follows from Fubini's theorem justified by \eqref{ineq:some-est_p_kappa} and Remark~\ref{rem:conv_Lp}.
The inequality follows then from \eqref{e:L-int-commute-2}, \eqref{e:fract-der-p-kappa-1b} and again Remark~\ref{rem:conv_Lp}.
\qed

\begin{lemma}\label{lem:for_max}
Let $f\in C_0(\Rd)$. For $t>0$, $x\in\Rd$ we define
$u(t,x)=P^{\kappa}_t f(x)$ and $u(0,x)=f(x)$. 
Then $u\in C([0,T]\times \Rd)$,
\eqref{e:nonlocal-max-principle-1} holds and $\partial_t u(t,x)=\LL_x^{\kappa}u(t,x)$ for all $t,T>0$, $x\in\Rd$.
\end{lemma}
\pf
First we show that 
(i), (ii), (iii) (and (iv) in the case $\Pa$)  of Theorem~\ref{t:intro-main} hold true.
Indeed, it follows from Lemma~\ref{l:p-kappa-difference}, Lemma~\ref{lem:bdd_cont}(b) and \eqref{e:nonincrease-t}.
Moreover, part (iii) holds with $f_0=\rr_{t_0}$.
Except the last part (and one use of Lemma~\ref{lem:bdd_cont}(b)) we base the proof solely on the properties from Theorem~\ref{t:intro-main}.
Note that $u(t,x)=\int_{\Rd} p^{\kappa}(t,x,x-z)f(x-z)\,dz$
and
we have $|p^{\kappa}(t,x,x-z)f(x-z)| \leq  c f_0(z)\|f\|_{\infty}$ for all $t\in [t_0,T]$, $x\in\Rd$.
Thus we can use the dominated convergence theorem and the joint continuity  to get $u\in C( (0,T]\times \Rd)$.
The first part of the statement follows by
combining the latter with 
$\|u(t,\cdot)-u(0,\cdot) \|_{\infty}\to 0$, $t\to 0^+$ (see Lemma~\ref{lem:bdd_cont}(b) and~\eqref{e:intro-main-5}).
Let $\varepsilon>0$. By previous line there is $t_0>0$ such that $|u(t,x) | \leq |f(x)|+\varepsilon$ for all $t\in [0,t_0]$, $x\in\Rd$, while
for  $t\in [t_0,T]$, $x\in\Rd$ we have $|u(t,x)|\leq c( f_0*|f|)(x)$,
which is an element of $C_0(\Rd)$.
This finishes the proof of \eqref{e:nonlocal-max-principle-1}.
Finally, 
we prove the last part.
By
the mean value theorem,
Lemma~\ref{l:p-kappa-difference}(c), 
\eqref{e:fract-der-p-kappa-1b},
\eqref{e:nonincrease-t}
and the dominated convergence theorem
$\partial_t u(t,x)= \int_{\Rd} \partial_t p^{\kappa}(t,x,y) f(y)dy$.
Then we apply Lemma~\ref{l:p-kappa-difference}(c) and Lemma~\ref{l:L-int-commute0}.
\qed

The following result is the counterpart  of \cite[Lemma 4.3]{MR3500272}.
\begin{lemma}\label{l:L-int-commute}
For any bounded (uniformly) H\"older continuous function 
$f \in C^\eta_b(\R^d)$, $\eta>0$, 
and all $t>0$, $x\in\Rd$,
we have $\int_0^t | \LL_x^{\kappa} P_s^{\kappa}f(x)|ds <\infty$ and
\begin{align}\label{e:L-int-commute}
\LL_x^{\kappa}\left( \int_0^t P_s^{\kappa}f(x)\,ds\right) =\int_0^t \LL_x^{\kappa} P_s^{\kappa}f(x)\,ds\,.
\end{align}
\end{lemma}
\pf
By the definition, and Lemma~\ref{lem:grad_Pt} in the case $\Pa$,
\begin{align*}
\LL_x^{\kappa} \int_0^t P_s^{\kappa}f(x)\,ds
&=\int_{\Rd} \left( \int_0^t \int_{\Rd} \delta^{\kappa} (s,x,y;z)   f(y)dy ds \right) \kappa(x,z)J(z)dz\,.
\end{align*}
Note that by \eqref{eq:LPf} the poof will be finished if we can change the order of integration from $dsdz$ to $dzds$. To this end we use Fubini's theorem justified by the following. We have $|f(y)-f(x)|\leq c (|y-x|^{\eta} \land 1)$ and we can assume that $\eta<\lah$. Then
\begin{align*}
\int_{\Rd}  \int_0^t &\left| \int_{\Rd} \delta^{\kappa} (s,x,y;z)   f(y)dy \right| ds \, \kappa(x,z)J(z)dz\\
&\leq \int_{\Rd}  \int_0^t \left| \int_{\Rd}  \delta^{\kappa} (s,x,y;z) \big[f(y)-f(x)\big] dy\right| ds \,\kappa(x,z)J(z)dz\\
&\quad+\int_{\Rd}  \int_0^t \left| \int_{\Rd} \delta^{\kappa} (s,x,y;z) f(x)  dy\right| ds\, \kappa(x,z)J(z)dz=: {\rm I}_1+{\rm I}_2\,.
\end{align*}
By \eqref{ineq:some-est_p_kappa} we have
${\rm I}_1\leq c \int_0^t \int_{\Rd}  \err{0}{\eta}(s,y-x)  dyds$, while by \eqref{ineq:some-est_p_kappa_1}
${\rm I}_2\leq c \int_0^t s^{-1} \left[h^{-1}(1/s)\right]^{\beta_1}ds$.
The integrals are finite by Lemma~\ref{l:convolution}(a) and~\ref{l:convoluton-inequality}.
\qed

\begin{proposition}\label{lem:gen_sem_step1}
Assume \PG.
For any $f\in C_b^{2}(\Rd)$ and all $t>0$, $x\in\Rd$,
\begin{align}\label{eq:gen_sem_step1}
P_t^{\kappa}f(x)-f(x)=\int_0^t P_s^{\kappa}\LL^{\kappa} f(x)\,ds\,.
\end{align}
\end{proposition}
\pf
(i) Note that $\LL^{\kappa}f \in C_0(\Rd)$ for any $f\in C_0^2(\Rd)$.

\noindent
(ii) We will show that if $f\in C_0^{2,\varepsilon}(\Rd)$,
then $\LL^{\kappa}f$ is (uniformly) H\"older continuous. 
To this end we use \cite[Theorem~5.1]{MR2555009}. 
 For $x,z\in\Rd$ define
$$
E_zf(x)=f(x+z)-f(x)\,,\qquad F_zf(x)=f(x+z)-f(x)-\left<z,\nabla f(x)\right>\,.
$$
We only consider the cases \Pa{} and \Pc. The case $\Pb$ is similar (see Lemma~\ref{lem:int_J}). 
Then $\LL^{\mathfrak{K}_y}f(x)=\int_{|z|<1}F_z f(x)\kappa(y,z)J(z)dz+\int_{|z|\geq 1}E_zf(x)\kappa(y,z)J(z)dz$. Using
\eqref{e:intro-kappa}, \eqref{e:intro-kappa-holder}, \eqref{e:psi1} and \cite[Theorem~5.1(b) and (e)]{MR2555009},
\begin{align*}
&|\LL^{\kappa}f(x) - \LL^{\kappa}f(y)|
\leq |\LL^{\mathfrak{K}_x}f(x)-\LL^{\mathfrak{K}_y}f(x)|+|\LL^{\mathfrak{K}_y}f(x)-\LL^{\mathfrak{K}_y}f(y)|\\
& \leq c |x-y|^{\beta}+\int_{|z|<1}|F_zf(x)-F_zf(y)|\,\kappa(y,z)J(z)dz+ \int_{|z|\geq 1}|E_zf(x)-E_zf(y)| \,\kappa(y,z)J(z)dz\\
& \leq c |x-y|^{\beta}+c |x-y|^{\varepsilon} \int_{|z|<1}|z|^2\nu(|z|)dz+c |x-y| \int_{|z|\geq 1}\nu(|z|)dz\,.
\end{align*}

\noindent
(iii) We will prove that \eqref{eq:gen_sem_step1} holds if $f\in C_0^{2,\varepsilon}(\Rd)$. 
Let $u_0$ and $u_1$ be defined as in Lemma~\ref{lem:for_max} such that $u_0(0,x)=\LL^{\kappa}f(x)$ and $u_1(0,x)=f(x)$. Further, let
$$
u_2(t,x):=f(x)+\int_0^t P_s^{\kappa}\LL^{\kappa} f(x)ds=f(x)+\int_0^t u_0(s,x)ds\,.
$$ 
By Lemma~\ref{lem:for_max} for $u_0$ we get that $u_2\in C([0,T]\times \Rd)$, \eqref{e:nonlocal-max-principle-1} holds
for $u_2$
 and
$\partial_t u_2(t,x)=u_0(t,x)$.
Using  (ii), Lemma~\ref{l:L-int-commute}, Lemma~\ref{lem:for_max} for $u_0$
and \cite[Theorem~7.21]{MR924157} 
we have
\begin{align*}
\LL_x^{\kappa} u_2(t,x)&= \LL^{\kappa} f(x) + \int_0^t \LL_x^{\kappa} u_0(s,x)\,ds
=\LL^{\kappa} f(x) + \int_0^t \partial_s u_0(s,x)\,ds\\
&=\LL^{\kappa} f(x)+\lim_{\varepsilon\to 0^+} \int_\varepsilon^t \partial_s u_0(s,x)\,ds= u_0(t,x)=\partial_t u_2(t,x)\,.
\end{align*}
Thus we can apply Corollary~\ref{cor:jedn_max} to $u_1$ and $u_2$, which implies the claim.

\noindent
(iv) We will extend \eqref{eq:gen_sem_step1} to  $f\in C_b^2(\Rd)$ by approximation.
Take
$\varphi \in C_c^{\infty}(\Rd)$ such that $\varphi(x)=1$ if $|x|\leq 1$, $\varphi(x)=0$ if $|x|\geq 2$
and
set $\varphi_n(x)=\varphi(x/n)$. 
Let $\{\phi_n\}_{n\in \N}$ be 
standard mollifier such that $\supp(\phi_n)\subset B(0,1/n)$.
Then
 $f_n= (f*\phi_n) \cdot \varphi_n \in C_c^{\infty}(\Rd)$ and $f_n\to f$, $\nabla f_n\to \nabla f$  pointwise.
Thus $E_z f_n(x)\to E_z f(x)$ and $F_z f_n(x)\to F_zf(x)$.
Further,
since
$\|\partial_{x}^{\bbbeta}(f*\phi_n)\|_{\infty}\leq \| \partial_{x}^{\beta} f\|_{\infty}$ for every multi-index $|\bbbeta|\leq 2$, 
there is $c>0$ such that for all $x,z\in\Rd$ and $n\in\N$,
$$
|E_z f_n(x)|\leq c\, (|z|\land 1)\,,\qquad  |F_z f_n(x)|\leq c\, |z|^2\,.
$$
Therefore, $\LL^{\kappa} f_n(x)\to \LL^{\kappa}f(x)$ and $\|\LL^{\kappa}f_n \|_{\infty}\leq c<\infty$.
The result follows from \eqref{eq:gen_sem_step1} for $f_n$
 and the dominated convergence theorem (see Lemma~\ref{l:p-kappa-difference}(b) and~\ref{lem:integr_rr}).
\qed
\begin{lemma}\label{lem:p-kappa-final-prop}
The function $p^{\kappa}(t,x,y)$ is non-negative, $\int_{\Rd} p^{\kappa}(t,x,y)dy= 1$ and $p^{\kappa}(t+s,x,y)=\int_{\Rd}p^{\kappa}(t,x,z)p^{\kappa}(s,z,y)dz$
for all $s,t>0$, $x,y\in\Rd$.
\end{lemma}
\pf
By Lemma~\ref{lem:for_max} we can apply Theorem~\ref{t:nonlocal-max-principle} to $u_1(t,x):=P_t^{\kappa}f(x)$, $u_1(0,x):=f(x)$ for any $f\in C_0(\Rd)$ and $T>0$.
The choice of $f\leq 0$ results in $u_1(t,x)\leq 0$ and proves the non-negativity of $p^{\kappa}(t,x,y)$.
Next, given $s>0$, $y\in\Rd$, by Lemma~\ref{l:p-kappa-difference}(a)(b) we can take $f(x)=p^{\kappa}(s,x,y)$.
We also consider $u_2(t,x)=p^{\kappa}(t+s,x,y)$.
It is clear from  
Lemma~\ref{l:p-kappa-difference}(a)(b)(c) and~\eqref{e:fract-der-p-kappa-1b}
that $u_2$ satisfies assumptions of
Corollary~\ref{cor:jedn_max}. Hence $P_t^{\kappa} p(s,\cdot,y) (x)=p^{\kappa}(s+t,x,y)$.
Finally, putting $f=1$ in Proposition~\ref{lem:gen_sem_step1} we get $P_t^{\kappa} 1 - 1 =0$.
\qed

\subsection{Proofs of Theorems~\ref{t:intro-main}--\ref{thm:onC0Lp}}

\noindent
{\bf Proof of Theorem~\ref{thm:onC0Lp}}.
By
Lemma~\ref{lem:bdd_cont} and~\ref{lem:p-kappa-final-prop}
the family
$(P^{\kappa}_t)_{t\geq 0}$ is a strongly continuous positive contraction semigroup on  
$(L^p(\Rd), \|\cdot\|_p)$, $p\in [1,\infty)$, and (additionally contraction) on
$(C_0(\Rd), \|\cdot\|_{\infty})$.
We postpone the proof of the analyticity.
Note that for there exists $c>0$ such that
for every  $g\in C_0^2(\Rd)$ (resp. $g\in C^2(\Rd)$ and $\partial^{\bbbeta}g \in L^p(\Rd)$ for every multi-index $|\bbbeta|\leq 2$),
$$
\left\|\LL^{\kappa}g\right\|_p\leq c \sum_{|\bbbeta|\leq 2} \left\|\partial^{\bbbeta} g \right \|_p\,.
$$
The inequality follows by recovering increments of function $g$ from its partial derivatives, Minkowski's inequality and integrability properties of the measure $J(z)dz$.
We also see that $\LL^{\kappa}f \in C_0(\Rd)$ for $f\in C_0^2(\Rd)$
(resp. $\LL^{\kappa}f \in L^p(\Rd)$ for $f\in C_c^2(\Rd)$).
By Proposition~\ref{lem:gen_sem_step1}, Minkowski's integral inequality and Lemma~\ref{lem:bdd_cont},
we have for $f\in C_0^2(\Rd)$ (resp. $f\in C_c^2(\Rd)$),
$$
\left\|  \frac{P_t^{\kappa}f-f}{t} - \LL^{\kappa}f\right\|_p
\leq \int_0^t \left\| P_s^{\kappa}\LL^{\kappa}f- \LL^{\kappa}f \right\|_p \frac{ds}{t}\,,
$$
which tends to zero as $t\to 0^+$, and ends the proof of 3(a) and 4(a).
In order to prove 3(b) and 4(b)  we investigate $(\bar{\mathcal{A}}_c^{\kappa}, D(\bar{\mathcal{A}}_c^{\kappa}))$ the closure of
$(\mathcal{A}_c^{\kappa}, D(\mathcal{A}_c^{\kappa})):=(\LL^{\kappa},C_c^{\infty}(\Rd))$
in $(C_0(\Rd),\|\cdot\|_{\infty})$ (resp. $(L^p(\Rd), \|\cdot\|_p)$).

\noindent
{\it Step 1}.
We show that $g\in D(\bar{\mathcal{A}}_c^{\kappa})$ and $\bar{\mathcal{A}}_c^{\kappa}g = \LL^{\kappa}g$
if $g\in C_0^{\infty}(\Rd)$ (resp. $g\in C_0^{\infty}(\Rd)$ and $\partial^{\bbbeta}g \in L^p(\Rd)$ for every multi-index $\bbbeta$). 
Take
$\varphi \in C_c^{\infty}(\Rd)$ such that $\varphi(x)=1$ if $|x|\leq 1$, $\varphi(x)=0$ if $|x|\geq 2$
and
set $\varphi_n(x)=\varphi(x/n)$. 
Then $g_n=g\cdot \varphi_n \in C_c^{\infty}(\Rd)$ and for every $|\bbbeta|\leq 2$,
$$
\|\partial^{\bbbeta}(g_n-g)\|_p \leq \|(\partial^{\bbbeta} g)(\varphi_n -1)\|_p + c/n\,,
$$
where $c$ depends only on
$d$, $\|\partial^{\bbbeta} g \|_p$, $|\bbbeta|\leq 1$,
and $\|\partial^{\bbbeta} \varphi \|_{\infty}$, $|\bbbeta|\leq 2$.
Then
$\|g_n-g\|_p \to 0$ and
\begin{equation*}
\|\bar{\mathcal{A}}_c^{\kappa} g_n - \LL^{\kappa}g \|_p=
\|\LL^{\kappa} g_n - \LL^{\kappa}g \|_p
\leq c \sum_{|\bbbeta|\leq 2} \| \partial^{\bbbeta} (g_n-g) \|_p
\leq c \sum_{|\bbbeta|\leq 2} \| (\partial^{\bbbeta}g) (\varphi_n-1) \|_p +c/n\,,
\end{equation*}
which ends the proof by the of that part.

\noindent
{\it Step 2}.
We show that $P_t^{\kappa}f \in D(\bar{\mathcal{A}}_c^{\kappa})$
and $\bar{\mathcal{A}}_c^{\kappa} P_t^{\kappa}f=\LL^{\kappa}P_t^{\kappa}f$
for all $t>0$ and $f\in C_0(\Rd)$ (resp. $f\in L^p(\Rd)$).
Let $\{\phi_n\}_{n\in \N}$ be a 
standard mollifier such that $\supp(\phi_n)\subset B(0,1/n)$.
Then 
by Lemma~\ref{lem:bdd_cont}
$h_n:= (P_t^{\kappa}f) * \phi_n \in C_0^{\infty}(\Rd)$ 
(resp. $h_n\in C_0^{\infty}(\Rd)$ and $\partial^{\bbbeta} h_n \in L^p(\Rd)$ for every~$\bbbeta$)
and 
$
\|h_n-P_t^{\kappa}f\|_{p}\to 0
$ as $n\to\infty$.
By {\it Step 1}., the definition (\eqref{eq:grad_Pt}, \eqref{e:fract-der-p-kappa-2} and Remark~\ref{rem:conv_Lp} in the case~$\Pa$),
\begin{align*}
\bar{\mathcal{A}}_c^{\kappa} h_n(x)
=\LL^{\kappa}h_n(x)
=\int_{\Rd}\int_{\Rd} \left( \int_{\Rd} \delta^{\kappa}(t,x-w,y;z) f(y)dy  \right)  \phi_n(w) \kappa(x,z)J(z)\, dwdz\,.
\end{align*}
Using Fubini's theorem (see \eqref{ineq:some-est_p_kappa}) and
\eqref{eq:LPf} we have
\begin{align*}
\bar{\mathcal{A}}_c^{\kappa}& h_n(x)-(\LL^{\kappa} P_t^{\kappa}f)*\phi_n (x)\\
&= \int_{\Rd}\int_{\Rd} \left( \int_{\Rd} \delta^{\kappa}(t,x-w,y;z) f(y)dy  \right) \big( \kappa(x,z)- \kappa(x-w,z)\big)J(z)dz\,  \phi_n(w) dw\,.
\end{align*}
Let $n$ be large enough so that the support of $\phi_n$ is contained in a ball of radius $\varepsilon>0$.
By \eqref{e:intro-kappa-holder}, 
\eqref{e:intro-kappa},
 \eqref{ineq:some-est_p_kappa}
and Remark~\ref{rem:conv_Lp}
we get $\| \bar{\mathcal{A}}_c^{\kappa} h_n -(\LL^{\kappa} P_t^{\kappa}f)*\phi_n \|_p \leq c \,\varepsilon^{\beta}t^{-1} \|f\|_{p}$  with $c$ independent of large $n\in\N$.
We also have by \eqref{e:L-int-commute-2},  Lemma~\ref{l:p-kappa-difference}(f) and~(d),
Corollary~\ref{cor:small_shift} and
Remark~\ref{rem:conv_Lp} that $\LL^{\kappa} P_t^{\kappa}f \in C_0(\Rd)$
(resp. $\LL^{\kappa} P_t^{\kappa}f \in L^p(\Rd)$ by \eqref{e:LP-p-estimate}).
Thus 
$\|\bar{\mathcal{A}}_c^{\kappa} h_n - \LL^{\kappa} P_t^{\kappa}f\|_p \to 0$ as $n\to\infty$,
which ends the proof.

\noindent
{\it Step 3}.
Obviously, $\bar{\mathcal{A}}_c^{\kappa}\subseteq \mathcal{A}^{\kappa}$
and it remains to show the converse inclusion.
Let $f\in D(\mathcal{A}^{\kappa})$ and define $f_n=P^{\kappa}_{1/n} f$.
By {\it Step 2}. we have $f_n\in D(\bar{\mathcal{A}}_c^{\kappa})$ and
$\|f_n-f\|_p\to 0$. 
Since $\mathcal{A}^{\kappa}$ commutes with $P^{\kappa}_t$ on $D(\mathcal{A^{\kappa}})$,
we also get
$$
\|\bar{\mathcal{A}}_c^{\kappa} f_n - \mathcal{A}^{\kappa} f\|_p
=\|\mathcal{A}^{\kappa} f_n - \mathcal{A}^{\kappa} f\|_p
=\|P_{1/n}^{\kappa}\mathcal{A}^{\kappa} f - \mathcal{A}^{\kappa} f\|_p\to 0\,.
$$
This finally gives $\bar{\mathcal{A}}_c^{\kappa}=\mathcal{A}^{\kappa}$.

\noindent
Now,
by {\it Step 2}. $ P^{\kappa}_t $  is differentiable in $C_0(\Rd)$ (resp. $L^p(\Rd)$), $t>0$,  and 
$\partial_t P_t^{\kappa}=\mathcal{A}^{\kappa} P^{\kappa}_t = \LL^{\kappa} P^{\kappa}_t$  on $C_0(\Rd)$ (resp. $L^p(\Rd)$)
(see \cite[Chapter~1, Theorem~2.4(c)]{MR710486}.
Therefore, since for all $s>0$, $y\in\Rd$, the function $f(x)=p^{\kappa}(s,x,y)$
belongs to  $C_0(\Rd)$ (resp. $L^p(\Rd)$), parts 3(c) and 4(c) follow.
We prove the analyticity.
Take numbers $M\geq 1$ and $\omega \in\R$ such that the operator norm $\|P^{\kappa}_t\| \leq M e^{\omega t}$ (see \cite[Chapter~1, Theorem~2.2]{MR710486}).
Define $T_t:=e^{-\lambda t}P_t^{\kappa}$, $\lambda= \omega +1$. 
It suffices to show the analyticity of $(T_t)_{t\geq 0}$.
Note that $(T_t)_{t\geq 0}$ is generated by $A=(-\lambda+\mathcal{A}^{\kappa})$ and that $T_t$ is differentiable in $C_0(\Rd)$ (resp. $L^p(\Rd)$), $t>0$, and $A T_t = (-\lambda+\LL^{\kappa}) T_t$  on $C_0(\Rd)$ (resp. $L^p(\Rd)$).
Then, by \eqref{e:LP-p-estimate} for $t\in (0,2]$,
\begin{align*}
\|A T_t f \|_p \leq |\lambda| \| T_t f\|_p + e^{2|\lambda|}\|\LL^{\kappa} P_t^{\kappa} f\|_p\leq 
(|\lambda| M + c e^{2|\lambda|}t^{-1} ) \|f\|_p \leq c_1 t^{-1} \|f\|_p\,.
\end{align*}
Next, by
\cite[Chapter~1, Theorem~2.4(c)]{MR710486}
for $t\geq 2$,
\begin{align*}
\|AT_tf  \|_p = \| T_{t-1} AT_1 f \|_p \leq \|T_{t-1}\| \| AT_1 f \|_p
\leq M e^{-(t-1)}c_1 \|f\|_p \leq c_2 t^{-1}\|f\|_p.
\end{align*}
We conclude that $\|AT_t\|\leq C t^{-1}$ for all $t>0$. 
The analyticity follows from
\cite[Chapter~2, Theorem~5.2(d)]{MR710486}.
\qed

\noindent
{\bf Proof of Theorem~\ref{t:intro-further-properties}}.
All the properties are collected in
Lemma~\ref{l:p-kappa-difference}
and~\ref{lem:p-kappa-final-prop}, except for part (8), which is given in Theorem~\ref{thm:onC0Lp} part 3(c).

\qed

\noindent
{\bf Proof of Theorem~\ref{t:intro-main}}.
Suppose there is another function $\tilde{p}^{\kappa}(t,x,y)$ 
that is jointly continuous on $(0,\infty)\times\Rd\times\Rd$ and
satisfies \eqref{e:intro-main-1}, \eqref{e:intro-main-5}, \eqref{e:intro-main-2}, \eqref{e:intro-main-4}.
In the case $\Pa$ we also assume \eqref{e:intro-main-a1}.
Let $T>0$ and $f\in C_c^{\infty}(\Rd)$. For $t\in (0,T]$, $x\in\Rd$ define
$u_1(t,x)=\int_{\Rd} p^{\kappa}(t,x,y) f(y)dy$,
$u_2 (t,x)=\int_{\Rd} \tilde{p}^{\kappa}(t,x,y)f(y)dy$
and $u_1(0,x)=u_2(0,x)=f(x)$. 
We will justify that 
Corollary~\ref{cor:jedn_max} applies to $u_1$ and $u_2$.
For $u_1$ it follows directly from Lemma~\ref{lem:for_max}.
For $u_2$ the proof is the same as the proof of Lemma~\ref{lem:for_max}, except for the last part.
Thus it remains to show \eqref{e:nonlocal-max-principle-4} for $u_2$.
By
the mean value theorem,
\eqref{e:intro-main-1},
\eqref{e:intro-main-4}
and the dominated convergence theorem we get
$\partial_t u(t,x)= \int_{\Rd} \partial_t \tilde{p}^{\kappa}(t,x,y) f(y)dy$.
Now, it suffices to show
\begin{align}\label{eq:aux}
\LL_x^{\kappa,0^+} u_2(t,x)=\int_{\Rd} \LL_x^{\kappa, 0^+} \tilde{p}^{\kappa}(t,x,y) f(y)dy\,.
\end{align}
Note that in the case $\Pa$ by \eqref{e:intro-main-a1} we get $\nabla_x u_2(t,x)=\int_{\Rd}\nabla_x \tilde{p}^{\kappa}(t,x,y) f(y)dy$.
By Fubini's theorem, justified by \eqref{e:intro-main-2}, and \eqref{e:intro-main-a1} in the case $\Pa$,
for $\varepsilon>0$
we have $\LL_x^{\kappa,\varepsilon} u(t,x)=\int_{\Rd} \LL_x^{\kappa,\varepsilon} \tilde{p}^{\kappa}(t,x,y)f(y)dy$.
Using \eqref{e:intro-main-4} and dominated convergence theorem we get \eqref{eq:aux}.
Finally,  by Corollary~\ref{cor:jedn_max} $u_1\equiv u_2$ on $[0,T]\times \Rd$, hence $\tilde{p}^{\kappa}(t,x,y)=p^{\kappa}(t,x,y)$, since $T>0$ and $f\in C_c^{\infty}(\Rd)$ were arbitrary.
\qed

\subsection{Lower bound of $p^{\kappa}(t,x,y)$}

\begin{lemma}\label{lem:lower_extend}
Assume that there exist $T,R, c>0$ such that 
\begin{equation}\label{eq:low_pr_1}
p^{\kappa}(t,x,y)\ge c \left[ h^{-1}(1/t)\right]^{-d}, \qquad t\in(0,T],\, |x-y|\leq R h^{-1}(1/t).
\end{equation}
Then there is $C=C(d,\param,T,R,c)>0$ such that
\begin{equation*}
p^{\kappa}(t,x,y)\ge C\left(\left[ h^{-1}(1/t)\right]^{-d}\wedge t\nu (|x-y|)\right), \qquad  t\in(0,T],\, x,y\in\Rd.
\end{equation*}
\end{lemma}

\pf
Let $X=(X_t)_{t\geq 0}$
the Feller process  corresponding to $(P_t^{\kappa})_{t  \geq 0}$.
By 
Remark~\ref{rem:MP}
(cf. \cite[Theorem~3.21]{MR3156646})
for every $f\in C_0^2(\Rd)$,
\begin{align}\label{e:MG}
M_t^f:=f(X_t)-f(x)-\int_0^t \LL^{\kappa}f(X_{s-})\, ds\,,\qquad t\geq 0\,,
\end{align}
is a martingale with respect to the filtration $\sigma(X_s: s \leq t)$. 
Let $A\subseteq \Rd$  be 
compact and
$f\in C^\infty_c(\Rd)$ such that $\supp (f )\cap A=\emptyset$.
Define a martingale $N_t^f:=\int^t_0\ind_{A}(X_{s-})dM^f_s$. By \cite[Theorem 3.5]{1998_RS} we get that 
$$N_t^f=\sum_{s\leq t}\ind_{A}(X_{s-})f(X_s) - \int^t_0 \ind_A(X_{s-})\int_{\Rd} f(X_{s-}+y)\kappa(X_{s-},y)J(y)dyds.$$
Let $B\subseteq \Rd$ be compact and satisfy $A\cap B =\emptyset$.
Taking $f_n\in C_c^\infty(\Rd)$ such that $0\leq f_n\leq 1$ and $f_n\downarrow \ind_B$ we get that 
$$
\sum_{s\leq t}\ind_{A}(X_{s-})\ind_B(X_s) - \int^t_0 \ind_A(X_{s-})\int_{\Rd} \ind_B(X_{s-}+y)\kappa(X_{s-},y)J(y)dyds
$$
 is a martingale. Therefore, by the optional stopping theorem
for every bounded stopping time $\tau$ and compact sets $A,B\subseteq \Rd$ such that $A\cap B=\emptyset$
we obtain
\begin{align}
\label{e:LSF}
\EE^x \sum_{0<s\le \tau} \ind_{A}(X_{s-})\ind_B(X_s)=\E_x \int_0^\tau \ind_A(X_{s})\int_{\Rd} \ind_B(X_{s}+y)\kappa(X_{s},y)J(y)dyds.
\end{align}
We can and do assume that $R\leq 2$.
Fix $M=h^{-1}(1/T)$ and let $r_t=(R/2) h^{-1}(2/t)$. By Remark \ref{rem:rozciaganie} we stretch the range of scaling in \eqref{eq:intro:wlsc} (and \eqref{eq:intro:wusc} in the case $\Pb$) to $(0,M]$.
For $D \in \mathcal{B}(\Rd)$
we define
$\tau_D:=\inf\{t \geq 0: X_t \notin D\}$ the first exit time of $X$ from $D$.
We claim that there exists $\lambda=\lambda(d,\param,T,R)\in(0,1/2]$ such that for  every $t\in (0,T]$,
\begin{align}\label{e:exit-probability-1}
\sup_{x \in \Rd} \PP^x\left(\tau_{B(x,r_t/4)}\leq \lambda t\right)\leq \frac{1}{2}\, .
\end{align}
By \cite[Theorem 5.1]{MR3156646} there is an absolute constant $c_1$ such that for all $r,s>0$ and $x\in\Rd$
$$
\PP^x\left(\tau_{B(x,r)}\leq s\right)\leq c_1 s \sup_{|x-z|\leq r}\sup_{|\xi|\leq 1/r}|q(z,\xi)|\,,
$$
where $q(z,\xi)$ is the symbol of the operator $\mathcal{L}^\kappa$ (see \cite[Corollary~2.23]{MR3156646} for definition). 
In the case $\Pa$, $\Pb$, $\Pc$ we use, respectively, that for every $\varphi \in \RR$ we have $\left|e^{i\varphi}-1-i\varphi \ind_{|w|\leq 1} \right|\leq  2 (|\varphi|^2\land |\varphi|)\ind_{|w|\leq 1}+2\cdot\ind_{|w|>1}$, $\left|e^{i\varphi}-1\right|\leq 2 (|\varphi|\land 1)$
and $|1-\cos(\varphi)|\leq 2(|\varphi|^2\land 1)$.
Therefore, by \eqref{e:psi1}, \eqref{e:intro-kappa} and Lemma~\ref{lem:int_J} we obtain
$\sup_{|\xi|\leq 1/r} |q(z,\xi)|\leq  c_2 h(r)$
for all $z\in \Rd$, $0<r\leq M$ and some $c_2=c_2(d,\param,T,R)$.
 Hence
$$\sup_{x\in \Rd}\PP^x\left(\tau_{B(x,r)}\leq s\right)\leq c_2 s h(r),\qquad 0<r\leq M,\, s>0\,.$$
By $r_t/4\leq M$ and  $c_2 (\lambda t) h(r_t/4)\leq 2c_2 \lambda  (8/R)^2=c_3^{-1} \lambda$
the inequality 
\eqref{e:exit-probability-1} holds with $\lambda= (1\land c_3)/2$.

We consider $|y-x|\geq Rh^{-1}(1/t)$, which
implies that $|x-y|\geq 2r_t$.
By the strong Markov property and \eqref{e:exit-probability-1} we have
for   $\Ht:=\inf \{ s\geq 0\colon X_s\in B(y, 3 r_t/4)\}$,
\begin{align}
&\PP^x \Big( X_{\lambda t}\in B(y,r_t)) \Big)
\geq \PP^x\Big(\Ht \leq \lambda t, \sup_{s\in [\Ht, \Ht+\lambda t]} |X_s-X_{\Ht}|<r_t/4\Big) \nonumber \\
&=\EE^x\left(\Ht\leq \lambda t; \PP^{X_{\Ht}}\Big(\sup_{s\in[0,\lambda t]} |X_s-X_0|<r_t/4\Big)\right)\nonumber\\
&\geq \PP^x\big(\Ht \leq \lambda t\big)  \inf_{z\in  \Rd} \PP^z\big(\tau_{ B(z,r_t/4)}>\lambda t\big) \nonumber\\
&\geq \frac12 \PP^x\big(\Ht \leq \lambda t\big) \, \geq \, \frac12 \PP^x\left(X_{\lambda t \wedge \tau_{B(x,r_t)}}\in\overline{ B(y, r_t/2)}\right)\, .\label{e:ineque00}
\end{align}
Noticing that 
$X_{s-}\in A:=\overline{B(x,r_t)}$ for $s\leq \tau_{ B(x,r_t)}$
and
$X_s\notin B:=\overline{B(y,r_t/2 )}$ for $s< \tau_{ B(x,r_t)}$ 
we have 
$$
\ind_{\overline{ B(y,r_t/2)}}\left(X_{\lambda t \land \tau_{ B(x,r_t)}}\right)=\sum_{s\le \lambda t \land \tau_{ B(x,r_t)}} \ind_A(X_{s-})\ind_{B}(X_s)\, .
$$
Thus, by the L\'evy system formula  \eqref{e:LSF}we obtain
\begin{align}
\PP^x\left(X_{\lambda t \land \tau_{ B(x,r_t)}}\in \overline{ B(y,r_t/2)}\right)
&=  \EE^x\left[\int_0^{\lambda t \land \tau_{ B(x,r_t)}} \ind_{A}(X_{s-}) \int_{\Rd} \ind_{B}(X_s+z)\kappa(X_s,z)J(z)\, dz ds\right] \nonumber\\
&\geq \kappa_0 \lmCJ^{-1} \EE^x\left[\int_0^{\lambda t \land \tau_{ B(x,r_t/4)}}\int_{B(y,r_t/2)}\nu(|X_s-z|)
\, dz ds\right].\label{e:ineque1}
\end{align}
Let  $z_0$ be the point on the line segment $[x,y]$ such that  
$|z_0-y|=3  r_t/8$.
Then 
$B(z_0,  r_t/8) \subseteq B(y, \,  r_t/2)$ and 
$|X_s-z|<|x-y|$ if $X_s\in B(x, r_t/4 )$, $z\in B(z_0,r_t/8)$.
Hence the monotonicity of $\nu$ and \eqref{e:exit-probability-1} imply that
\begin{align}
&\EE^x\left[ \int_0^{\lambda t \land \tau_{ B(x,r_t/4)}} \int_{B(y,r_t/2)} \nu (|X_s-z|)\, dzds\right] \nonumber\\
&\geq  \EE^x \left[\lambda t \land \tau_{ B(x,r_t/4)}\right] \int_{B(z_0,r_t/8)}   \nu( |x-y|)\, dz \nonumber \\
& \geq \lambda t\, \PP^x \left(\tau_{B(x,r_t/4)}\geq \lambda t\right) |B(z_0,r_t/8)|\, \nu(|x-y|) \nonumber\\
&\geq c(d)\lambda t \,\left[h^{-1}(2/t)\right]^d\, \nu(|x-y|).\label{e:ineque2}
\end{align}
Combining \eqref{e:ineque00},  \eqref{e:ineque1} and \eqref{e:ineque2}
 for $c_4=c_4(d,\param,T,R)$ and all $t\in (0,T]$, $|x-y|\geq R h^{-1}(1/t)$,
\begin{align}\label{e:lower-bound-2}
\PP^x (X_{\lambda t}\in B(y,r_t))\geq  c_4\,t \left[h^{-1}(2/t)\right]^d \nu( |x-y|)\,.
\end{align}
Finally, since $\lambda\in (0,1/2]$ we have  $r_t\leq R h^{-1}(1/[(1-\lambda)t])$.
Therefore,
 by \eqref{eq:low_pr_1}, \eqref{e:lower-bound-2},
Lemma~\ref{lem:equiv_scal_h} and the monotonicity of $h^{-1}$ we get
for all $t\in (0,T]$, $|x-y|\geq R h^{-1}(1/t)$,
\begin{align*}
p^{\kappa}(t,x,y)&\geq  \int_{B(y,r_t)}p^{\kappa}(\lambda t, x,z)p^{\kappa}((1-\lambda)t,z,y)\, dz\\
&\geq \PP^x \left(X_{\lambda t}\in B(y,r_t)\right) \inf_{|z-y|<r_t}p^{\kappa}((1-\lambda)t,z,y)\\
& \geq  c c_4\, \lambda t \left[h^{-1}(2/t)\right]^d\, \nu(|x-y|)  \left[h^{-1}(1/t)\right]^{-d} \geq  C\, t \nu(|x-y|)\,.
\end{align*}
\qed

\noindent
{\bf Proof of Theorem \ref{thm:lower-bound}}.
{\rm (i)} Let  $t\in (0,1]$ and $|x-y|\leq h^{-1}(1/t)$.
By Lemma~\ref{prop:gen_est_low}
 there is $c_1=c_1(d,\nu,\param)$ such that
 \begin{align*}
p^{\mathfrak{K}_y}(t,x,y)\geq c_1 
\left[h^{-1}(1/t)\right]^{-d}.
\end{align*}
By Lemma~\ref{lem:phi_cont_xy} with $c_2=c_2(d,\param,\kappa_2,\beta)$,
\begin{align*}
\left|\phi_y(t,x) \right|&\leq c_2 t \big(\err{0}{\beta}+\err{\beta}{0}\big)(t,x-y)
\leq  c_2 t \left[ h^{-1}(1/t)\right]^{\beta} \err{0}{0}(t,x-y)\leq c_2 \left[ h^{-1}(1/t)\right]^{\beta}   \left[h^{-1}(1/t)\right]^{-d}\, .
\end{align*}
Thus $|\phi_y(t,x)|\leq (c_1/2) [h^{-1}(1/t)]^{-d}$
for all $t\in (0,T_0]$, where $T_0=1\land [h((2c_2/c_1)^{-1/\beta})]^{-1}$.
By \eqref{e:p-kappa} we 
conclude 
that for all $t\in (0,T_0]$ and $|x-y|\leq  h^{-1}(1/t)$ we have
\begin{align}\label{eq:low_pr_1_bis}
p^{\kappa}(t,x,y)\geq (c_1/2) \left[h^{-1}(1/t)\right]^{-d}\,.
\end{align}
This ends the proof of this part due to Lemma~\ref{lem:lower_extend}.

\noindent
{\rm (ii)} 
It suffices to show that if \eqref{eq:low_pr_1} holds for $T>0$ and $R=1$, then it holds for $3T/2$ and $R=1$
(which allows to obtain \eqref{eq:low_pr_1} from \eqref{eq:low_pr_1_bis} and then apply Lemma~\ref{lem:lower_extend}).
Let $t\in [T,3/2T]$ and $|x-y|\leq  h^{-1}(1/t)$, then for $r=h^{-1}(1/(t-T/2))$,
\begin{align*}
p^{\kappa}(t,x,y)&\geq \int_{B(y,r)}p^{\kappa}(T/2,x,z)p^{\kappa}(t-T/2,z,y)\,dz\\
&\geq  \inf_{|x-z|\leq 2h^{-1}(1/T)}p^{\kappa}(T,x,z) |B(y,r)|  c \left[h^{-1}(1/(t-T/2))\right]^{-d} \\
&\geq c (\omega_d/d) \inf_{|x-z|\leq 2h^{-1}(1/T)}p^{\kappa}(T,x,z)  \geq c'=c'(d,\nu,\param,T,c)>0\,.
\end{align*}
We have used Lemma~\ref{lem:lower_extend} and the positivity of $\nu$ in the last inequality.
Finally, we use that $[h^{-1}(1/T)]^{-d} \geq [h^{-1}(1/t)]^{-d}$.

\noindent
{\rm (iii)}
The statement follows from part {\rm (ii)} and Lemma~\ref{lem:comp_uK_uLM}.

\qed

\section{Appendix - unimodal L{\'e}vy processes}\label{sec:appA}

Let $d\in\N$ and
$\nu:[0,\infty)\to[0,\infty]$ be a non-increasing  function  satisfying
$$\int_{\Rd}  (1\land |x|^2) \nu(|x|)dx<\infty\,.$$
For any such $\nu$ there exists a unique
 pure-jump
isotropic unimodal L{\'e}vy process $X$ (see \cite{MR3165234}, \cite{MR705619}).
The characteristic exponent $\uLCh$ of $X$ 
takes the form 
$$\uLCh(x)={\rm Re} [\uLCh(x)]=\int_{\Rd}\big( 1-\cos\left<x,z\right> \big)\nu(|z|)dz\,.$$
For $r>0$ we define 
$h(r)$ and $K(r)$ as in the introduction, and
we let $\uLCh^*(r):=\sup_{|z|\leq r} {\rm Re}[\uLCh(z)]$. Then
(see \cite[Proposition~2]{MR3165234}),
\begin{align}\label{ineq:comp_unimod}
(1/\pi^2) \uLCh^*(|x|) \leq \uLCh(x) \leq \uLCh^*(|x|)\,.
\end{align}
It is also known  that (see \cite[Lemma~4]{MR3225805}), 
\begin{align}\label{eq:hcompPhi}
\frac{1}{8(1+2d)} h(1/r)\leq \uLCh^*(r) \leq 2 h(1/r)\,.
\end{align}

Note that $h(0^+)<\infty$ ($h$ is bounded) if and only if $\nu(\Rd)<\infty$, i.e., the corresponding L\'{e}vy process is a compound Poisson process. 
In the whole section {\bf we assume that} $h(0^+)=\infty$.
We collect and prove general estimates for functions $K$, $h$ and $\rr_t$
(see \cite[Section~2 and~6]{GS-2017}).

\subsection{Properties of $K$ and $h$}

The following properties 
are often used without further comment.

\begin{lemma}\label{lem:basic_prop_K_h}
We have
\begin{enumerate}
\item $K$ and $h$ are continuous and $\lim_{r\to\infty}h(r)=\lim_{r\to\infty}K(r)=0$.
\item $r^2 K(r)$ and $r^2h(r)$ are non-decreasing,
\item $r^{-d}K(r)$ and $h(r)$ are strictly decreasing,
\item $\lambda^2 K(\lambda r) \leq K(r)\leq \lambda^{-d}K(\lambda r)$ 
and $\lambda^2 h(\lambda r)\leq h(r)$, $\lambda\leq 1$, $r>0$,
\item $\sqrt{\lambda} h^{-1}(\lambda u)\leq h^{-1}(u)$, $\lambda\geq 1$, $u>0$,
\item $\nu(r)\leq (\omega_d/(d+2))^{-1}\, r^{-d} K(r)$, 
\item
For all $0<a<b\leq \infty$,
$$
h(b)-h(a)=-\int_a^b 2 K(r)r^{-1}\,dr\,.
$$
\item[\rm 6.] For all $r>0$,
\begin{align*}
&\int_{|z|\geq r }  \nu (dz)\leq  h(r)
\quad \mbox{and} \quad
\int_{|z|< r}  |z|^2 \nu(dz) \leq  r^2 h(r)\,.
\end{align*}
\end{enumerate}
\end{lemma}

\noindent
We consider the scaling conditions: there are $\lah \in(0,2]$, $C_h\in[1,\infty)$ and $\theta_h\in(0,\infty]$ such that
\begin{equation}\label{eq:wlsc:h}
 h(r)\leq C_h\lambda^{\lah }h(\lambda r),\qquad \lambda\leq 1,\, r< \theta_h.
 \end{equation} 
In like manner, 
there are $\uah \in (0,2]$, $c_h\in (0,1]$ and $\theta_h \in (0,\infty]$ 
such that
\begin{equation}\label{eq:wusc:h}
 c_h\,\lambda^{\uah}\,h(\lambda r)\leq h(r)\, ,\quad \lambda\leq 1, \,r< \theta_h.
\end{equation}

\begin{remark}\label{rem:rozciaganie}
If $\theta_h<\infty$ in \eqref{eq:wlsc:h}, 
we can stretch the range of scaling to $r <R<\infty$
at the expense of the constant $C_h$.
Indeed, by continuity of $h$,  for $\theta_h\leq r< R$,
$$h(r)\leq h(\theta_h)\leq C_h \lambda^{\lah} h(\lambda \theta_h)\leq C_h (r/\theta_h)^2 \lambda^{\lah} h(\lambda r)\leq [C_h (R/\theta_h)^2]\lambda^{\lah}  h(\lambda r).$$
Similarly,
if $\theta_h<\infty$  in \eqref{eq:wusc:h} we extend the scaling to $r<R$ as follows, for $\theta_h\leq r <R$,
$$
h(r)\geq (\theta_h/R)^2 h(\theta_h) \geq c_h (\theta_h/R) \lambda^{\uah} h(\lambda \theta_h) \geq [c_h (\theta_h/R)^2] \lambda^{\uah} h(\lambda r)\,.
$$
\end{remark}

\begin{lemma}\label{lem:equiv_scal_h}
Let $\lah \in(0,2]$, $C_h\in[1,\infty)$ and $\theta_h\in(0,\infty]$. 
The following are equivalent.

\begin{enumerate}
\item[\Aa] For all $\lambda\leq 1$ and $r< \theta_h$,
\begin{equation*}
 h(r)\leq C_h\lambda^{\lah }h(\lambda r)\,. 
\end{equation*}
\item[\Ab]  For all $\lambda\geq 1$ and $u>h(\theta_h)$,
\begin{equation*}
 h^{-1}(u)\leq (C_h\lambda)^{1/\lah}\, h^{-1}(\lambda u)\,. 
\end{equation*}
\end{enumerate}
\noindent
Further, consider
\begin{enumerate}
\item[\Ad] There is $\underline{c}\in (0,1]$ such that
for all $\lambda \geq 1$ and $r>1/\theta_h$,
\begin{align*}
\uLCh^* (\lambda r) \geq \underline{c} \lambda^{\lah} \uLCh^*(r)\,.
\end{align*}
\item[\Ac] There is $c>0$ such that for all $r<\theta_h$,
\begin{equation*}
h(r)\leq c K(r)\,. 
\end{equation*}
\end{enumerate}
Then,   
$\Aa$
gives $\Ad$
with $\underline{c}= 1/(c_d C_h)$,   $c_d=16(1+2d)$, while
$\Ad$
gives 
$\Aa$
with $C_h=c_d/\underline{c}$.\\
$\Aa$ 
 implies $\Ac$  
with $c=c(\lah,C_h)$.
$\Ac$
implies 
$\Aa$
with $\lah=2/c$ and $C_h=1$.
\end{lemma}

\begin{lemma}\label{lem:comp_uK_uLM}
The following are equivalent.
\begin{enumerate}
\item[$\Aaa$] There are $T_1\in (0,\infty]$, $c_1>0 $ such that for all $r<T_1$,
\begin{align*}
c_1 r^{-d} K(r) \leq  \nu(r)\,.
\end{align*}
\item[$\Abb$] There are $T_2\in(0,\infty]$, $c_2\in (0,1]$ and $\beta_2\in (0,2)$ such that
for all $\lambda \leq 1$ and $r<T_2$,
\begin{align*}
c_2 \lambda^{\beta_2} K(\lambda r) \leq K(r)\,.
\end{align*}
\item[$\Acc$] There are $T_3\in (0,\infty]$, $c_3\in (0,1]$ and $\beta_3 \in [0,2)$ such that
for all $\lambda\leq 1$ and $r<T_3$,
$$
c_3 \lambda^{d+\beta_3} \nu(\lambda r) \leq \nu(r)\,.
$$
\end{enumerate}
Moreover, $\Aaa$ implies $\Abb$ with
$T_2=T_1$, $c_2=1$ and $\beta_2=\beta_2(d,c_1)$.
From $\Aaa$
we get $\Acc$ with
$T_3=T_1$, $c_3=c_3(d,c_1)$ and $\beta_3=\beta_3(d,c_1)$.
The condition $\Abb$ gives $\Aaa$ with $T_1=(c_2/2)^{1/(2-\beta_2)} T_2$ and $c_1=c_1(d,c_2,\beta_2)$. From $\Acc$ we have $\Aaa$ with $T_1=T_3$ and $c_1=c_1(d,c_3,\beta_3)$.
\end{lemma}

\begin{lemma}\label{lem:int_J}
Let $h$ satisfy \eqref{eq:wlsc:h} with $\lah>1$, then
\begin{align*}
\int_{r \leq |z|<  \theta_h }  |z|\nu(|z|)dz \leq \frac{(d+2) C_h}{\lah-1} \, r h(r)\,, \qquad r>0\,.
\end{align*}
Let $h$ satisfy \eqref{eq:wusc:h} with $\uah<1$, then
\begin{align*}
\int_{|z|< r} |z| \nu(|z|)dz\leq \frac{d+2}{c_h(1-\uah)}\, r h(r)\,,\qquad r< \theta_h \,.
\end{align*}
\end{lemma}

\subsection{Properties of the bound function $\rr_t(x)$} 

We collect properties of the bound function defined in \eqref{e:intro-rho-def}.

\begin{lemma}\label{lem:integr_rr}
We have
$$(\omega_d/2)\leq \int_{\Rd}\rr_t(x)\,dx \leq (\omega_d/2)(1+2/d),\qquad t>0\,.$$
\end{lemma}

\begin{lemma}\label{lem:sol:rho}
Fix $t>0$. There is a unique solution $r_0>0$ of
$$tK(r)r^{-d}=  [h^{-1}(1/t)]^{-d}=\rr_t(r)\,,$$ 
and  $r_0\in [h^{-1}(3/t),h^{-1}(1/t)]$.
\end{lemma}

\begin{proposition}\label{prop:small_shift}
Let $a\geq 1$. There is $c=c(d,a)$ such that
for all $t>0$,
\begin{align*}
\rr_t(x+z)\leq c\, \rr_t(x)\,, \qquad\qquad \mbox{if} \qquad |z|\leq \left[a\,h^{-1}(3/t)\right] \vee \frac{|x|}{2}\,.
\end{align*}
\end{proposition}

\begin{lemma}\label{lem:int_rr_J}
There exists a constant $c=c(d,\lah,C_h)$
such that for all
$t>0$
and $x\in\Rd$,
\begin{align*}
\int_{|z|\geq h^{-1}(1/t)} \rr_t(x-z) \nu(|z|)dz\leq c t^{-1} \rr_t(x)\,.
\end{align*}
\end{lemma}
\pf
We split the integral into parts.
If $|z|\leq |x|/2$, then by  Proposition~\ref{prop:small_shift} we get $\rr_t(x-z) \nu(|z|)\leq c \rr_t(x) \nu(|z|)$
and we apply Lemma~\ref{lem:basic_prop_K_h}.
If $|x|/2 \leq |z|$, then 
we first simply use
$
\rr_t(x-z) \leq  \left[h^{-1}(1/t)\right]^{-d} 
$
and Lemma~\ref{lem:basic_prop_K_h}
to find a bound by $t^{-1}\left[h^{-1}(1/t)\right]^{-d}$.
At the same time we have
$\rr_t(x-z)\nu(|z|)\leq \rr_t(x-z)\nu(|x|/2)\leq c \rr_t(x-z) K(|x|)|x|^{-d}$,
which together with Lemma~\ref{lem:integr_rr} give
a bound by $t^{-1} (tK(|x|)|x|^{-d})$.
Finally, we take the minimum.
\qed

We collect further properties of the bound function
under \eqref{eq:wlsc:h}.

\begin{corollary}\label{cor:small_shift}
Let $h$ satisfy \eqref{eq:wlsc:h}. For every $a\geq 1$ 
there is $c=c(d,a,\lah,C_h)$ such that
\begin{align*}
\rr_t(x+z)\leq c\, \rr_t(x)\,, \qquad\qquad \mbox{if} \qquad |z|\leq \left[a\,h^{-1}(1/t)\right] \vee \frac{|x|}{2}\quad \mbox{and}\quad t<1/h(\theta_h)\,.
\end{align*}
\end{corollary}

\begin{corollary}\label{cor:por_0}
Let $h$ satisfy \eqref{eq:wlsc:h}. For $t>0$, $x\in\Rd$ define 
$$
\varphi_t(x)=
\begin{cases}
[h^{-1}(1/t)]^{-d},\qquad &|x|\leq h^{-1}(1/t)\,,\\
tK(|x|)|x|^{-d}, & |x|>  h^{-1}(1/t)\,.
\end{cases}
$$
Then $\rr_t(x)\leq \varphi_t(x) \leq c\, \rr_t(x)$ for all $t<1/h(\theta_h)$, $x\in\Rd$ and a constant $c=c(\lah,C_h)$.
\end{corollary}

\begin{lemma}\label{lem:int_min_rr}
Let $h$ satisfy \eqref{eq:wlsc:h}. For all
$\beta\in [0,\lah)$ and
$t<1/h(\theta_h)$ we have
\begin{align*}
\int_{\Rd} (|x|^{\beta}\land 1) \rr_t(x)\,dx \leq 2\omega_d \frac{ C_h\big(1+1/\theta_h^{\beta}\big)}{\lah-\beta} \left[h^{-1}(1/t)\right]^{\beta}\,.
\end{align*}
\end{lemma}
\pf
If $\beta=0$, the result follows from Lemma~\ref{lem:integr_rr}. Assume that $\beta>0$.
We have
\begin{align*}
\int_{|x|<h^{-1}(1/t)} |x|^{\beta} \rr_t(x)\,dx 
\leq  
\int_{|x|<h^{-1}(1/t)} \left[h^{-1}(1/t)\right]^{\beta} \left[h^{-1}(1/t)\right]^{-d} \,dx= \frac{\omega_d}{d}
\left[h^{-1}(1/t)\right]^{\beta}\,.
\end{align*}
The integral over the set $|x|\geq h^{-1}(1/t)$ is bounded by the sum
\begin{align*}
\int_{|x|\geq \theta_h} \rr_t(x)\,dx+\int_{h^{-1}(1/t)\leq |x|<\theta_h} |x|^{\beta} \rr_t(x)\,dx\,.
\end{align*}
Further, we have
\begin{align*}
\int_{|x|\geq \theta_h} \rr_t(x)\,dx\leq - \frac{t\omega_d}{2} \int_{\theta_h}^{\infty} h'(r)\,dr=  \frac{t \omega_d}{2} h(\theta_h)
\leq  \frac{\omega_d C_h}{2 \theta_h^{\beta}} \left[h^{-1}(1/t)\right]^{\beta} \,,
\end{align*}
where the last inequality follows from
$
r^{1/\beta}h^{-1}(r)\geq C_h^{-1/\beta} u^{1/\beta} h^{-1}(u)
$
for $r\geq u\geq h(\theta_h)$, which is a consequence of $\Ab$ of Lemma~\ref{lem:equiv_scal_h}, the assumption $0<\beta<\lah$ and continuity of $h^{-1}$.
Now \eqref{eq:wlsc:h} with $\lambda=h^{-1}(1/t)/r$ gives
\begin{align*}
&\int_{h^{-1}(1/t)\leq |x|<\theta_h} |x|^{\beta} \rr_t(x)\,dx
\leq t \int_{h^{-1}(1/t)\leq |x|<\theta_h} |x|^{\beta -d} K(|x|)\,dx
\leq t \omega_d \int_{h^{-1}(1/t)}^{\theta_h} r^{\beta-1} h(r)\,dr\\
&\leq t \omega_d C_h \int_{h^{-1}(1/t)}^{\theta_h} r^{\beta-1} \left[\frac{h^{-1}(1/t)}{r} \right]^{\lah} (1/t)\,dr
\leq \omega_d C_h \left[h^{-1}(1/t)\right]^{\lah} \int_{h^{-1}(1/t)}^{\infty} r^{\beta-1-\lah} \,dr\\
& \leq \frac{\omega_d C_h}{\lah-\beta} 
\left[h^{-1}(1/t)\right]^{\beta}\,.
\end{align*}
\qed

\subsection{3G-type inequalities}\label{sec:3P_ineq}

Let $\hh\colon [0,\infty)\to (0,\infty)$ be non-increasing and such that 
$$\lambda^{\alpha_{\hh}}\hh(\lambda t)\leq c_{\hh}\hh(t),\qquad \lambda\leq 1,\,\ t< \theta_{\hh},$$
for some $\alpha_{\hh}\leq 1$,  $c_{\hh}\in [1,\infty)$ and $\theta_{\hh}\in (0,\infty]$.
For $t>0$ and $x\in\Rd$ we consider
\begin{align}\label{def:hp}
\hrr_t (x)= \hh(t)\rr_t(x)\,.
\end{align}

\begin{proposition}\label{prop:3P}
Let $h$ satisfy \eqref{eq:wlsc:h}. 
There exists a constant $c=c(d,\lah,C_h,\alpha_\hh,c_\hh)$ 
such that
for all $s+t < 1/h(\theta_h) \land \theta_\hh$, $x,y\in \Rd$,
\begin{align}\label{ineq:3Pg_0}
\hrr_s(x)\land \hrr_t(y)\leq c \hrr_{s+t}(x+y)\,.
\end{align}
\end{proposition}
\begin{proof}
First, note that $t\mapsto \hh(t)\big[ h^{-1}(1/t) \big]^{-d}$ and $r\mapsto r^{-d}K(r)$ are non-increasing. Thus
$$
 \hh(s)\big[ h^{-1}(1/s) \big]^{-d}\land  \hh(t)\big[ h^{-1}(1/t) \big]^{-d}\leq  \hh((s+t)/2)\big[ h^{-1}(2/(s+t)) \big]^{-d} \,,
$$
and 
$$
\frac{K(|x|)}{|x|^d}\land \frac{K(|y|)}{|y|^{d}}\leq \frac{K((|x|+|y|)/2)}{\big[(|x|+|y|)/2\big]^d}\leq 2^{d+2}\frac{K(|x+y|)}{|x+y|^d}\,.
$$
Since $\alpha_\hh\leq 1$ for $\lambda \leq 1$
we have $\hh(\lambda t)(\lambda t) \leq c_{\hh} \hh(t)t$
on $(0,\theta_\hh)$. For $s+t\in (0,\theta_\hh)$ we get
$$
\big[\hh(s)\,s\big] \vee \big[\hh(t)\,t \big] \leq c_{\hh}\, \hh(s+t)\, (s+t)\,.
$$
Finally,
\begin{align*}
\hrr_s(x)\land \hrr_t(y)\leq & \, 
\hh(s)\big[ h^{-1}(1/s) \big]^{-d}\land  \hh(t)\big[ h^{-1}(1/t) \big]^{-d}\\ 
 &\land c_{\hh}  \Big(\hh(s+t)(s+t)\Big)\left[\frac{K(|x|)}{|x|^d}\land \frac{K(|y|)}{|y|^{d}} \right]\\
\leq &\, \hh((s+t)/2)\big[ h^{-1}(2/(s+t)) \big]^{-d}\land  2^{d+2} c_{\hh} \Big(\hh(s+t)(s+t)\Big) \frac{K(|x+y|)}{|x+y|^d}\,.
\end{align*}
The inequality follows by scaling conditions for $\hh$ and $h^{-1}$ (see Lemma~\ref{lem:equiv_scal_h}).
\end{proof}

Since we can take $\hh\equiv 1$ with $(\alpha_\hh,\theta_\hh,c_\hh)=(0,0,1)$ we recover the classical 3G inequality.

\begin{corollary}\label{cor:3Pclass}
Let $h$ satisfy 
\eqref{eq:wlsc:h}. 
There exists a constant $c=c(d,\lah,C_h)$ such that
for all $s+t<1/h(\theta_h)$, $x,y\in \Rd$,
\begin{align}\label{ineq:3Pc_0}
\rr_s(x)\wedge \rr_t(y)\leq c \rr_{s+t}(x+y)\,.
\end{align}
\end{corollary}

\subsection{Convolution inequalities}\label{subsec:conv}
Let $B(a,b)$ be the beta function, i.e., $B(a,b)=\int_0^1 s^{a-1} (1-s)^{b-1}ds$, $a,b>0$.

\begin{lemma}\label{l:convoluton-inequality}
Let $\theta,\eta\in\R$. The inequality
\begin{align*}
\int_0^t u^{-\eta}[h^{-1}(1/u)]^{\gamma} (t-u)^{-\theta}[h^{-1}(1/(t-u))]^{\beta}\, du \leq c\, t^{1-\eta-\theta}[h^{-1}(1/t)]^{\gamma+\beta}\,,\quad t>0\,,
\end{align*}
holds in the following cases:
\begin{enumerate}
\item[(i)] for all $\beta, \gamma\geq 0$
such that $\beta/2+1-\theta>0$, $\gamma/2+1-\eta>0$  with $c=B(\beta/2+1-\theta, \gamma/2+1-\eta)$,
\item[(ii)]  under \eqref{eq:intro:wlsc}, for all $t\in (0,T]$, $T>0$, and all $\beta,\gamma\in \R$
such that $(\beta/2) \land (\beta/\lah) +1-\theta>0$, $(\gamma/2) \land (\gamma/\lah) +1-\eta>0$
with
\begin{align*}
c= (C_h [h^{-1}(1/T)\vee 1]^2)^{-(\beta\land 0 +\gamma\land 0)/\lah}
B\Big((\beta/2) \land (\beta/\lah) +1-\theta, (\gamma/2) \land (\gamma/\lah) +1-\eta\Big)\,.
\end{align*}
\end{enumerate}
\end{lemma}
\pf Let $I$ be the above integral. By the change of variable $s=u/t$ we get that 
$$
I=t^{1-\eta-\theta}\int_0^1 s^{-\eta}\left[h^{-1}(t^{-1}s^{-1})\right]^{\gamma}(1-s)^{-\theta}\left[h^{-1}(t^{-1}(1-s)^{-1})\right]^{\beta}\, ds\, .
$$
Since $s^{-1}\geq 1$ and $(1-s)^{-1}\geq 1$ we have 
$h^{-1}(t^{-1}s^{-1})\leq  s^{1/2}h^{-1}(t^{-1})$ and 
$h^{-1}(t^{-1}(1-s)^{-1})\leq  (1-s)^{1/2}h^{-1}(t^{-1})$. Hence,
\begin{align*}
I &\leq  t^{1-\eta-\theta}\left[h^{-1}(1/t)\right]^{\gamma+\beta}\int_0^1  s^{\gamma/2-\eta}(1-s)^{\beta/2-\theta}\, ds \\
&=   B(\beta/2+1-\theta, \gamma/2+1-\eta)t^{1-\eta-\theta}\left[h^{-1}(1/t)\right]^{\gamma+\beta}\, . 
\end{align*}
This proves (i). The cases (ii)  follows from
 Lemma~\ref{lem:equiv_scal_h} and Remark~\ref{rem:rozciaganie} that guarantee
\begin{align*}
\left[h^{-1}(t^{-1}s^{-1})\right]^{\gamma}&\leq (C_h [h^{-1}(1/T)\vee 1]^2)^{-\gamma/\lah} s^{\gamma/\lah}\left[h^{-1}(1/t)\right]^{\gamma}\,,\\
\left[h^{-1}(t^{-1}(1-s)^{-1})\right]^{\beta}&\leq (C_h [h^{-1}(1/T)\vee 1]^2)^{-\beta/\lah} (1-s)^{\beta/\lah}\left[h^{-1}(1/t)\right]^{\beta}\,,
\end{align*}
if $\gamma<0$ or $\beta<0$, respectively.
\qed

For $\gamma,\beta\in \R$
 we consider the function
$\err{\gamma}{\beta}$ defined
in \eqref{def:err}.
\begin{remark}\label{r:rho-decreasing}
The monotonicity of $h^{-1}$ assures the following,
\begin{align}
\err{\gamma_1}{\beta}(t,x)&\leq \left[ h^{-1}(1/T)\right]^{\gamma_1-\gamma_2} \err{\gamma_2}{\beta}(t,x) \,, &&(t,x)\in (0,T]\times \Rd,\quad \gamma_2\leq \gamma_1,
\label{e:nonincrease-gamma} \\
\err{\gamma}{\beta_1}(t,x)&\leq \err{\gamma}{\beta_2}(t,x), &&(t,x)\in (0,\infty)\times\Rd,\quad 0\leq \beta_2\leq \beta_1,
\label{e:nonincrease-beta}\\
\err{0}{0}(\lambda t,x)&\leq \err{0}{0}(t,x) \,, && (t,x)\in (0,\infty)\times \Rd,\quad \lambda\geq 1. \label{e:nonincrease-t}
\end{align}
\end{remark}

\begin{lemma}\label{l:convolution}
Assume \eqref{eq:intro:wlsc} and let $\beta_0\in [0,1]\cap [0,\lah)$.
\begin{itemize}
\item[(a)] For every $T>0$ there exists a constant $c_1=c_1(d,\beta_0,\lah,C_h,h^{-1}(1/T)\vee 1)$ such that for all $t\in(0,T]$ and $\beta\in [0,\beta_0]$,
$$
\int_{\Rd} \err{0}{\beta}(t,x)\,dx \leq c_1 t^{-1} \left[h^{-1}(1/t)\right]^{\beta}\,.
$$
\item[(b)] 
For every $T>0$ there exists a constant $c_2=c_2(d,\beta_0,\lah,C_h,h^{-1}(1/T)\vee 1) \ge 1$
such that
for all $\beta_1,\beta_2,n_1,n_2,m_1,m_2 \in[0,\beta_0]$ with $n_1, n_2 \leq \beta_1+\beta_2$, $m_1\leq \beta_1$, $m_2\leq \beta_2$ and all $0<s<t\leq T$, $x\in\Rd$,
\begin{align*}
\int_{\Rd} \err{0}{\beta_1}(t-s&,x-z)\err{0}{\beta_2}(s,z) \,dz\\
\leq  
c_2 &\Big[ \left( (t-s)^{-1} \left[h^{-1}(1/(t-s))\right]^{n_1} + s^{-1}\left[h^{-1}(1/s)\right]^{n_2}\right) \err{0}{0}(t,x)\\
&+(t-s)^{-1}\left[ h^{-1}(1/(t-s))\right]^{m_1} \err{0}{\beta_2}(t,x) + 
s^{-1}\left[ h^{-1}(1/s)\right]^{m_2} \err{0}{\beta_1}(t,x)\Big].
\end{align*}
\item[(c)] Let $T>0$. For all 
$\gamma_1, \gamma_2\in\RR$, 
$\beta_1,\beta_2,n_1,n_2,m_1,m_2 \in[0,\beta_0]$ with $n_1, n_2 \leq \beta_1+\beta_2$, $m_1\leq \beta_1$, $m_2\leq \beta_2$
and $\theta,\eta \in [0,1]$,
satisfying 
\begin{align*}
(\gamma_1+n_1\land m_1)/2 \land (\gamma_1+n_1\land m_1)/\lah +1-\theta>0\,,\\
(\gamma_2+n_2\land m_2)/2\land (\gamma_2+n_2\land m_2)/\lah+1-\eta>0\,,
\end{align*}
and all
$0<s<t\leq T$, $x\in\Rd$, we have
\begin{align}
\int_0^t\int_{\Rd}& (t-s)^{1-\theta}\, \err{\gamma_1}{\beta_1}(t-s,x-z) \,s^{1-\eta}\,\err{\gamma_2}{\beta_2}(s,z) \,dzds  \nonumber\\
& \leq c_3 \,
t^{2-\eta-\theta}\Big( \err{\gamma_1+\gamma_2+n_1}{0}
+\err{\gamma_1+\gamma_2+n_2}{0}
+\err{\gamma_1+\gamma_2+m_1}{\beta_2} 
+\err{\gamma_1+\gamma_2+m_2}{\beta_1} 
\Big)(t,x)\,,\label{e:convolution-3}
\end{align}
where 
$
c_3= c_2 \, (C_h[h^{-1}(1/T)\vee 1]^2)^{-(\gamma_1\land 0+\gamma_2 \land 0)/\lah}
 B\left( k+1-\theta, \,l+1-\eta\right)
$
and 
\begin{align*}
k=\left(\frac{\gamma_1+n_1\land m_1}{2}\right)\land
\left(\frac{\gamma_1+n_1\land m_1}{\lah}\right),
\quad l=\left(\frac{\gamma_2+n_2\land m_2}{2}\right)\land \left(\frac{\gamma_2+n_2\land m_2}{\lah}\right).
\end{align*}
\end{itemize}
\end{lemma}

\pf
Part $(a)$ follows immediately from Lemma~\ref{lem:int_min_rr} and Remark~\ref{rem:rozciaganie}. We prove part $(b)$.
Proposition~\ref{prop:3P} with $\hh(t)=t^{-1}$ provides
with $c=c(d,\lah,C_h,h^{-1}(1/T)\vee 1)>0$,
\begin{align*}
\frac{\err{0}{0}(t-s,x-z)\err{0}{0}(s,z)}{c\, \err{0}{0}(t,x)}
\leq  \err{0}{0}(t-s,x-z)+\err{0}{0}(s,z).
\end{align*}
Combining with (see formulas following \cite[(2.5)]{MR3500272} or use $(a+b)^{\beta}\leq a^{\beta}+b^{\beta}$, 
$(a+b)\land 1 \leq (a\land 1)+(b\land 1)$ and $(a\land 1)(b\land 1)\leq (ab)\land 1$ for any $\beta\in [0,1]$, $a,b \geq 0$)
\begin{align*}
&\left(|x-z|^{\beta_1} \land 1\right)
\left(|z|^{\beta_2} \land 1\right)
\leq \left(|x-z|^{\beta_1+\beta_2} \land 1\right)
+ \left(|x-z|^{\beta_1} \land 1\right) \left(|x|^{\beta_2} \land 1\right)\,,\\
&\left(|x-z|^{\beta_1} \land 1\right)
\left(|z|^{\beta_2} \land 1\right)
\leq
\left(|z|^{\beta_1+\beta_2} \land 1\right)
+ \left(|x|^{\beta_1} \land 1\right) \left(|z|^{\beta_2} \land 1\right)\,,
\end{align*}
we have by \eqref{e:nonincrease-beta},
\begin{align*}
&\frac{\err{0}{\beta_1}(t-s,x-z)\err{0}{\beta_2}(s,z)}{c\, \err{0}{0}(t,x)}
\leq 
 \Big(\left(|x-z|^{\beta_1+\beta_2} \land 1\right)
+ \left(|x-z|^{\beta_1} \land 1\right) \left(|x|^{\beta_2} \land 1\right)\Big) \err{0}{0}(t-s,x-z)\\
& \hspace{0.28\linewidth}+ \Big(\left(|z|^{\beta_1+\beta_2} \land 1\right)
+ \left(|x|^{\beta_1} \land 1\right) \left(|z|^{\beta_2} \land 1\right)\Big) \err{0}{0}(s,z)\\
& =\, \err{0}{\beta_1+\beta_2}(t-s,x-z)+ \left(|x|^{\beta_2} \land 1\right)
\err{0}{\beta_1}(t-s,x-z)
+ \err{0}{\beta_1+\beta_2}(s,z)+ \left(|x|^{\beta_1} \land 1\right)\err{0}{\beta_2}(s,z)\\
&\leq \, \err{0}{n_1}(t-s,x-z)+ \left(|x|^{\beta_2} \land 1\right)
\err{0}{m_1}(t-s,x-z)
+ \err{0}{n_2}(s,z)+ \left(|x|^{\beta_1} \land 1\right)\err{0}{m_2}(s,z)
\,.
\end{align*}
Integrating  both sides in $z$ and applying $(a)$ we obtain $(b)$. For the proof of $(c)$ we multiply both sides of $(b)$   by
$$
(t-s)^{1-\theta}\left[ h^{-1}(1/(t-s))\right]^{\gamma_1} s^{1-\eta} \left[h^{-1}(1/s)\right]^{\gamma_2}\,,
$$
integrate in $s$ and apply Lemma~\ref{l:convoluton-inequality} to reach $(c)$
with a constant
\begin{align*}
c_2 &\left(C_h\left[h^{-1}(1/T)\vee 1\right]^2\right)^{-(\gamma_1\land 0+\gamma_2\land 0) /\lah} \\
\times & \max  \Bigg\{
B\left(
k_1
+1-\theta, \frac{\gamma_2}{2}\land \frac{\gamma_2}{\lah}+2-\eta \right)\!;
B\left(\frac{\gamma_1}{2}\land \frac{\gamma_1}{\lah}+2-\theta,\, l_1+1-\eta\right)\!;\\
&\quad \qquad B\left(
k_2
+1-\theta, \frac{\gamma_2}{2}\land \frac{\gamma_2}{\lah}+2-\eta\right)\!;
B\left(\frac{\gamma_1}{2}\land \frac{\gamma_1}{\lah}+2-\theta, \,l_2+1-\eta\right)
\Bigg\},
\end{align*}
where 
$k_1=(\frac{\gamma_1+n_1}{2})\land (\frac{\gamma_1+n_1}{\lah})$, $k_2=(\frac{\gamma_1+m_1}{2})\land (\frac{\gamma_1+m_1}{\lah})$
and
$l_1=(\frac{\gamma_2+n_2}{2})\land (\frac{\gamma_2+n_2}{\lah})$
$l_2=(\frac{\gamma_2+m_2}{2})\land (\frac{\gamma_2+m_2}{\lah})$,
which by monotonicity of Beta function is smaller than $c_3$.
\qed

\begin{remark}
When using Lemma~\ref{l:convolution} without specifying the parameters we apply the usual case, i.e., $n_1=n_2=\beta_1+\beta_2$ ($\leq \beta_0$), $m_1=\beta_1$, $m_2=\beta_2$. Similarly, if only $n_1$, $n_2$ are specified, then $m_1=\beta_1$, $m_2=\beta_2$.
\end{remark}

\section{Appendix - general L{\'e}vy process}\label{sec:appB}

Let
$d\in\N$ and $Y=(Y_t)_{t\geq 0}$ be
a L{\'e}vy process  in $\Rd$  (\cite{MR1739520}). 
Recall that there is a well known
one-to-one
 correspondence
between L{\'e}vy processes in $\Rd$ and
the convolution semigroups 
of probability measures $(P_t)_{t\geq 0}$ on $\Rd$.
The characteristic exponent $\LCh$ of $Y$ is defined by 
$$\mathbb{E}
e^{i\left<x,Y_t\right>}
=\int_{\Rd} e^{i\left<x,y\right>} P_t(dy)
=e^{-t\LCh(x)}\,,\qquad x\in\Rd\,,
$$ 
and equals
$$
\LCh(x)=\left<x,Ax\right>-i\left<x,\drf\right> - \int_{\Rd} \left(e^{i\left<x,z\right>}-1 - i\left<x,z\right>\ind_{|z|<1}\right)\LM(dz)\,.
$$
Here 
$A$ is a symmetric non-negative definite matrix, $\drf \in \Rd$ and
$\LM(dz)$ is a L{\'e}vy measure, i.e., 
a measure satisfying
$$
\LM(\{0\})=0\,,
\qquad\quad
\int_{\Rd} (1\land |z|^2)\LM(dz)<\infty \,.
$$
We have $P_t f(x)=\mathbb{E}f(Y_t+x)$ and
 $(P_t)_{t\geq 0}$ is a strongly continuous positive contraction
semigroup on $(C_0(\Rd),\|\cdot\|_{\infty})$ with the infinitesimal generator $(L,D(L))$ such that $C_0^2(\Rd)\subseteq D(L)$
and for $f\in C_0^2(\Rd)$ we have
$$
Lf=\LL f(x):=
\sum_{i,j=1}^d A_{ij} \frac{\partial f(x)}{\partial x_i \partial x_j}
+ \left<\drf,\nabla f(x) \right>+
\int_{\Rd}(f(x+z)-f(x)-\ind_{|z|<1}\left<z,\nabla f(x)\right>)\LM(dz)\,.
$$
Note that the above equality on $C_c^{\infty}(\Rd)$ uniquely determines $(L,D(L))$ 
and the generating triplet $(A,\LM,\drf)$
(see \cite[Theorem~31.5 an~8.1]{MR1739520}).
We make the following assumption on the real part of $\LCh$,
\begin{align}\label{ogolne:zal1}
\lim_{|x|\to \infty} \frac{{\rm Re}[\LCh (x)]}{\log |x|}=\infty\,.
\end{align}
In particular, $\LM(\Rd)=\infty$, thus $Y$ is not a compound Poisson process.
It follows from \cite[Theorem~2.1]{MR3010850} (we only use implication which does not require $A=0$) that $Y_t$ has a density 
$p(t,x)$ for every $t>0$ and
\begin{align}\label{eq:def_p}
p(t,x)=(2\pi)^{-d}\int_{\R^d} e^{-i\left<x,z\right>}e^{-t\LCh (z)}\, dz\,, \quad \qquad p(t,\cdot)\in C_0^{\infty}(\Rd)\,.
\end{align}
We denote $p(t,x,y)=p(t,y-x)$ and observe that $P_t f(x)=\int_{\Rd} p(t,x,y)f(y)\,dy$.

\begin{lemma}\label{l:partial-time}
(a) For all $x,y\in \Rd$, the function $t\mapsto p(t,x,y)$ is differentiable on $(0,\infty)$ and
$$
\frac{\partial p (t,x,y)}{\partial t}
=
-(2\pi)^{-d}\int_{\Rd} e^{-i\left<y-x,z\right>}\LCh(z)e^{-t\LCh (z)}\, dz =\LL_x p(t,x,y)\, .
$$
\noindent
(b) Let $\varepsilon >0$. There is a constant $c>0$
 such that for all $s,t \ge \varepsilon$, $x,x',y\in\Rd$,
$$
|p(t,x,y)-p(s,x',y)|\le c \left(|t-s|+|x-x'|\right)\, .
$$
(c) Let $\varepsilon >0$. There is a constant $c>0$
 such that for all $s,t \ge \varepsilon$, $x,x',y\in\Rd$,
$$
|\nabla p(t,x,y)- \nabla p(s,x',y)|\le c \left(|t-s|+|x-x'|\right)\, .
$$

\end{lemma}
\pf 
(a) Note that for any $t> 0$ and any $h\in \R$ such that $t/2+h> 0$,
$$
\frac{p(t+h,x,y)-p(t,x,y)}{h}=
(2\pi)^{-d}\int_{\Rd} e^{-i\left<y-x,z\right>}e^{-t\LCh(z)}\frac{e^{-h\LCh(z)}-1}{h}\,dz.
$$
The absolute value of the integrand is bounded by $|\LCh(z)|e^{-(t/2){\rm Re}[ \LCh (z)]}$ which is integrable since $|\LCh(z)|\leq c(|z|^2+1)$ (see \cite[Proposition~2.17]{MR3156646}). The claim follows from the dominated convergence theorem. 
The second equality
follows from the semigroup property and \eqref{eq:def_p}. Indeed,
$P_h p(t,\cdot,y)(x)=\int_{\R^d}p(h,x,z)p(t,z,y)dz=p(t+h,x,y)$.
Hence
$\LL_x p(t,x,y)=\lim_{h\to 0^+}(p(t+h,x,y)-p(t,x,y))/h$.

\noindent 
(b) By (a) we have
$$
\sup_{t\geq \varepsilon,\,x,y\in\Rd} 
\left|\frac{\partial p(t,x,y)}{\partial t}\right|
\leq \int_{\Rd}|\LCh(z)|e^{-\varepsilon\, {\rm Re}[\LCh (z)] }\, dz=c_1<\infty.
$$
And 
$$
\sup_{t\geq \varepsilon,\,x,y\in\Rd}\left|\frac{\partial p(t,x,y)}{\partial x_k}\right|
\leq \max_{k=1,\ldots,d} \int_{\Rd}|z_k| e^{-\varepsilon\, {\rm Re} [\LCh(z)] }\,dz=c_2<\infty.$$
These imply the claim with
$c= c_1+ d c_2$.

\noindent
(c)
Like above we have
$\frac{\partial^2}{\partial t \partial x_k } p(t,x,y)=-(2\pi)^{-d}\int_{\Rd} e^{-i\left<y-x,z\right>}\LCh(z)z_k e^{-t\LCh (z)}\, dz$. Thrn we use
 $\sup_{t\geq \varepsilon,\,x,y\in\Rd} 
| \frac{\partial^2}{\partial t \partial x_k } p(t,x,y)|$
and $\sup_{t\geq \varepsilon,\,x,y\in\Rd} 
| \frac{\partial^2}{\partial x_j \partial x_k } p(t,x,y)|$.
\qed

We record a general fact which follows from
\cite[Lemma~4]{MR3225805}
and
Fubini's theorem.
\begin{lemma}\label{lem:log_LM}
Let $\LCh^*(r):=\sup_{|z|\leq r} {\rm Re}[\LCh(z)]$ for $r>0$. Then
$$\int_0^1 \frac{\LCh^*(r)}{r}dr<\infty \qquad \iff \qquad \int_{\Rd} \ln\left(1+ |z|^2\right) \LM(dz)<\infty \,.$$
\end{lemma}

\vspace{.1in}


\small
\bibliographystyle{abbrv}

\begin{thebibliography}{10}

\bibitem{MR2492992}
M.~T. Barlow, A.~Grigor'yan, and T.~Kumagai.
\newblock Heat kernel upper bounds for jump processes and the first exit time.
\newblock {\em J. Reine Angew. Math.}, 626:135--157, 2009.

\bibitem{MR2555009}
R.~F. Bass.
\newblock Regularity results for stable-like operators.
\newblock {\em J. Funct. Anal.}, 257(8):2693--2722, 2009.

\bibitem{MR3165234}
K.~Bogdan, T.~Grzywny, and M.~Ryznar.
\newblock Density and tails of unimodal convolution semigroups.
\newblock {\em J. Funct. Anal.}, 266(6):3543--3571, 2014.

\bibitem{MR2283957}
K.~Bogdan and T.~Jakubowski.
\newblock Estimates of heat kernel of fractional {L}aplacian perturbed by
  gradient operators.
\newblock {\em Comm. Math. Phys.}, 271(1):179--198, 2007.

\bibitem{BKS-2017}
K.~Bogdan, V.~Knopova, and P.~Sztonyk.
\newblock {Heat kernel of anisotropic nonlocal operators}.
\newblock preprint 2017, arXiv:1704.03705.

\bibitem{MR3295773}
K.~Bogdan and S.~Sydor.
\newblock On nonlocal perturbations of integral kernels.
\newblock In {\em Semigroups of operators---theory and applications}, volume
  113 of {\em Springer Proc. Math. Stat.}, pages 27--42. Springer, Cham, 2015.

\bibitem{MR2320691}
K.~Bogdan and P.~Sztonyk.
\newblock Estimates of the potential kernel and {H}arnack's inequality for the
  anisotropic fractional {L}aplacian.
\newblock {\em Studia Math.}, 181(2):101--123, 2007.

\bibitem{MR2163294}
B.~B\"ottcher.
\newblock A parametrix construction for the fundamental solution of the
  evolution equation associated with a pseudo-differential operator generating
  a {M}arkov process.
\newblock {\em Math. Nachr.}, 278(11):1235--1241, 2005.

\bibitem{MR2456894}
B.~B\"ottcher.
\newblock Construction of time-inhomogeneous {M}arkov processes via evolution
  equations using pseudo-differential operators.
\newblock {\em J. Lond. Math. Soc. (2)}, 78(3):605--621, 2008.

\bibitem{MR3156646}
B.~B\"ottcher, R.~Schilling, and J.~Wang.
\newblock {\em L\'evy matters. {III}: L\'evy-type processes: construction,
  approximation and sample path properties, With a short biography of Paul
  L\'evy by Jean Jacod,}, volume 2099 of {\em Lecture Notes in Mathematics}.
\newblock Springer, Cham, 2013.

\bibitem{MR898496}
E.~A. Carlen, S.~Kusuoka, and D.~W. Stroock.
\newblock Upper bounds for symmetric {M}arkov transition functions.
\newblock {\em Ann. Inst. H. Poincar\'e Probab. Statist.}, 23(2,
  suppl.):245--287, 1987.

\bibitem{MR2443765}
Z.-Q. Chen, P.~Kim, and T.~Kumagai.
\newblock Weighted {P}oincar\'e inequality and heat kernel estimates for finite
  range jump processes.
\newblock {\em Math. Ann.}, 342(4):833--883, 2008.

\bibitem{MR2806700}
Z.-Q. Chen, P.~Kim, and T.~Kumagai.
\newblock Global heat kernel estimates for symmetric jump processes.
\newblock {\em Trans. Amer. Math. Soc.}, 363(9):5021--5055, 2011.

\bibitem{MR2008600}
Z.-Q. Chen and T.~Kumagai.
\newblock Heat kernel estimates for stable-like processes on {$d$}-sets.
\newblock {\em Stochastic Process. Appl.}, 108(1):27--62, 2003.

\bibitem{CZ-new}
Z.-Q. Chen and X.~Zhang.
\newblock {Heat kernels for time-dependent non-symmetric stable-like
  operators}.
\newblock preprint 2017, arXiv:1709.04614.

\bibitem{MR3500272}
Z.-Q. Chen and X.~Zhang.
\newblock Heat kernels and analyticity of non-symmetric jump diffusion
  semigroups.
\newblock {\em Probab. Theory Related Fields}, 165(1-2):267--312, 2016.

\bibitem{CZ-survey}
Z.-Q. Chen and X.~Zhang.
\newblock {Heat Kernels for Non-symmetric Non-local Operators}.
\newblock In {\em Recent Developments in Nonlocal Theory}. de Gruyter, 2018.

\bibitem{MR3022725}
A.~Debussche and N.~Fournier.
\newblock Existence of densities for stable-like driven {SDE}'s with {H}\"older
  continuous coefficients.
\newblock {\em J. Funct. Anal.}, 264(8):1757--1778, 2013.

\bibitem{MR0003340}
F.~G. Dressel.
\newblock The fundamental solution of the parabolic equation.
\newblock {\em Duke Math. J.}, 7:186--203, 1940.

\bibitem{MR0492880}
J.~M. Drin'.
\newblock A fundamental solution of the {C}auchy problem for a class of
  parabolic pseudodifferential equations.
\newblock {\em Dokl. Akad. Nauk Ukrain. SSR Ser. A}, (3):198--203, 284, 1977.

\bibitem{MR616459}
J.~M. Drin' and S.~D. Eidelman.
\newblock Construction and investigation of classical fundamental solutions to
  the {C}auchy problem of uniformly parabolic pseudodifferential equations.
  boundary value problems for partial differential equations.
\newblock {\em Mat. Issled.}, (63):18--33, 180--181, 1981.

\bibitem{MR2093219}
S.~D. Eidelman, S.~D. Ivasyshen, and A.~N. Kochubei.
\newblock {\em Analytic methods in the theory of differential and
  pseudo-differential equations of parabolic type}, volume 152 of {\em Operator
  Theory: Advances and Applications}.
\newblock Birkh\"auser Verlag, Basel, 2004.

\bibitem{MR838085}
S.~N. Ethier and T.~G. Kurtz.
\newblock {\em Markov processes. Characterization and convergence}.
\newblock Wiley Series in Probability and Mathematical Statistics: Probability
  and Mathematical Statistics. John Wiley \& Sons, Inc., New York, 1986.

\bibitem{zbMATH03022319}
W.~{Feller}.
\newblock {Zur Theorie der stochastischen Prozesse. (Existenz- und
  Eindeutigkeitss\"atze)}.
\newblock {\em {Math. Ann.}}, 113:113--160, 1936.

\bibitem{MR0181836}
A.~Friedman.
\newblock {\em Partial differential equations of parabolic type}.
\newblock Prentice-Hall, Inc., Englewood Cliffs, N.J., 1964.

\bibitem{MR2778606}
M.~Fukushima, Y.~Oshima, and M.~Takeda.
\newblock {\em Dirichlet forms and symmetric {M}arkov processes}, volume~19 of
  {\em De Gruyter Studies in Mathematics}.
\newblock Walter de Gruyter \& Co., Berlin, extended edition, 2011.

\bibitem{zbMATH02629782}
M.~{Gevrey}.
\newblock {Sur les \'equations aux d\'eriv\'ees partielles du type
  parabolique.}
\newblock {\em {C. R. Acad. Sci., Paris}}, 152:428--431, 1911.

\bibitem{MR3225805}
T.~Grzywny.
\newblock On {H}arnack inequality and {H}\"older regularity for isotropic
  unimodal {L}\'evy processes.
\newblock {\em Potential Anal.}, 41(1):1--29, 2014.

\bibitem{GS-2017}
T.~Grzywny and K.~Szczypkowski.
\newblock {Estimates of heat kernels of non-symmetric L{\'e}vy processes}.
\newblock preprint 2017, arxiv:1710.07793.

\bibitem{MR3713578}
T.~Grzywny and K.~Szczypkowski.
\newblock Kato classes for {L}\'evy processes.
\newblock {\em Potential Anal.}, 47(3):245--276, 2017.

\bibitem{MR1659620}
W.~Hoh.
\newblock A symbolic calculus for pseudo-differential operators generating
  {F}eller semigroups.
\newblock {\em Osaka J. Math.}, 35(4):789--820, 1998.

\bibitem{MR0499861}
C.~T. Iwasaki.
\newblock The fundamental solution for pseudo-differential operators of
  parabolic type.
\newblock {\em Osaka J. Math.}, 14(3):569--592, 1977.

\bibitem{MR1254818}
N.~Jacob.
\newblock A class of {F}eller semigroups generated by pseudo-differential
  operators.
\newblock {\em Math. Z.}, 215(1):151--166, 1994.

\bibitem{MR1873235}
N.~Jacob.
\newblock {\em Pseudo differential operators and {M}arkov processes. {V}ol.
  {I}. Fourier analysis and semigroups}.
\newblock Imperial College Press, London, 2001.

\bibitem{MR1917230}
N.~Jacob.
\newblock {\em Pseudo differential operators \& {M}arkov processes. {V}ol.
  {II}. Generators and their potential theory}.
\newblock Imperial College Press, London, 2002.

\bibitem{MR2158336}
N.~Jacob.
\newblock {\em Pseudo differential operators and {M}arkov processes. {V}ol.
  {III}. Markov processes and applications}.
\newblock Imperial College Press, London, 2005.

\bibitem{MR3550165}
T.~Jakubowski.
\newblock Fundamental solution of the fractional diffusion equation with a
  singular drift.
\newblock {\em J. Math. Sci. (N.Y.)}, 218(2):137--153, 2016.

\bibitem{MR2643799}
T.~Jakubowski and K.~Szczypkowski.
\newblock Time-dependent gradient perturbations of fractional {L}aplacian.
\newblock {\em J. Evol. Equ.}, 10(2):319--339, 2010.

\bibitem{MR2876511}
T.~Jakubowski and K.~Szczypkowski.
\newblock Estimates of gradient perturbation series.
\newblock {\em J. Math. Anal. Appl.}, 389(1):452--460, 2012.

\bibitem{PJ}
P.~Jin.
\newblock {Heat kernel estimates for non-symmetric stable-like processes}.
\newblock preprint 2017, arXiv:1709.02836.

\bibitem{MR3357585}
K.~Kaleta and P.~Sztonyk.
\newblock Estimates of transition densities and their derivatives for jump
  {L}\'evy processes.
\newblock {\em J. Math. Anal. Appl.}, 431(1):260--282, 2015.

\bibitem{MR1876169}
O.~Kallenberg.
\newblock {\em Foundations of modern probability}.
\newblock Probability and its Applications (New York). Springer-Verlag, New
  York, second edition, 2002.

\bibitem{KJ-2018}
P.~Kim and J.~Lee.
\newblock {Heat kernels of non-symmetric jump processes with exponentially
  decaying jumping kernel}.
\newblock preprint 2018, arXiv:1801.00661.

\bibitem{KSV16}
P.~Kim, R.~Song, and Z.~Vondra\v{c}ek.
\newblock Heat kernels of non-symmetric jump processes: Beyond the stable case.
\newblock Potential Anal. (2017), DOI:10.1007/s11118-017-9648-4.

\bibitem{MR3235175}
V.~Knopova.
\newblock Compound kernel estimates for the transition probability density of a
  {L}\'evy process in {$\Bbb R^n$}.
\newblock {\em Teor. \u Imov\=\i r. Mat. Stat.}, 89:51--63, 2013.

\bibitem{MR3139314}
V.~Knopova and A.~Kulik.
\newblock Intrinsic small time estimates for distribution densities of {L}\'evy
  processes.
\newblock {\em Random Oper. Stoch. Equ.}, 21(4):321--344, 2013.

\bibitem{MR3353627}
V.~Knopova and A.~Kulik.
\newblock Parametrix construction for certain {L}\'evy-type processes.
\newblock {\em Random Oper. Stoch. Equ.}, 23(2):111--136, 2015.

\bibitem{MR3652202}
V.~Knopova and A.~Kulik.
\newblock Intrinsic compound kernel estimates for the transition probability
  density of {L}\'evy-type processes and their applications.
\newblock {\em Probab. Math. Statist.}, 37(1):53--100, 2017.

\bibitem{MR3765882}
V.~Knopova and A.~Kulik.
\newblock Parametrix construction of the transition probability density of the
  solution to an {SDE} driven by {$\alpha$}-stable noise.
\newblock {\em Ann. Inst. Henri Poincar\'e Probab. Stat.}, 54(1):100--140,
  2018.

\bibitem{MR3010850}
V.~Knopova and R.~L. Schilling.
\newblock A note on the existence of transition probability densities of
  {L}\'evy processes.
\newblock {\em Forum Math.}, 25(1):125--149, 2013.

\bibitem{MR972089}
A.~N. Kochubei.
\newblock Parabolic pseudodifferential equations, hypersingular integrals and
  {M}arkov processes.
\newblock {\em Izv. Akad. Nauk SSSR Ser. Mat.}, 52(5):909--934, 1118, 1988.

\bibitem{MR3544166}
A.~Kohatsu-Higa and L.~Li.
\newblock Regularity of the density of a stable-like driven {SDE} with
  {H}\"older continuous coefficients.
\newblock {\em Stoch. Anal. Appl.}, 34(6):979--1024, 2016.

\bibitem{MR1744782}
V.~Kolokoltsov.
\newblock Symmetric stable laws and stable-like jump-diffusions.
\newblock {\em Proc. London Math. Soc. (3)}, 80(3):725--768, 2000.

\bibitem{FK-2017}
F.~{K{\"u}hn}.
\newblock {Transition probabilities of {L\'evy}-type processes: Parametrix
  construction}.
\newblock preprint 2017, arXiv:1702.00778.

\bibitem{KR-2017}
T.~Kulczycki and M.~Ryznar.
\newblock {Transition density estimates for diagonal systems of SDEs driven by
  cylindrical $\alpha$-stable processes}.
\newblock preprint 2017, arXiv:1711.07539.

\bibitem{K-2015}
A.~Kulik.
\newblock {On weak uniqueness and distributional properties of a solution to an
  SDE with $\alpha$-stable noise}.
\newblock preprint 2015, arXiv:1511.00106.

\bibitem{MR666870}
H.~Kumano-go.
\newblock {\em Pseudodifferential operators}.
\newblock MIT Press, Cambridge, Mass.-London, 1981.
\newblock Translated from the Japanese by the author, R\'emi Vaillancourt and
  Michihiro Nagase.

\bibitem{zbMATH02644101}
E.~E. {Levi}.
\newblock {Sulle equazioni lineari totalmente ellittiche alle derivate
  parziali.}
\newblock {\em {Rend. Circ. Mat. Palermo}}, 24:275--317, 1907.

\bibitem{MR710486}
A.~Pazy.
\newblock {\em Semigroups of linear operators and applications to partial
  differential equations}, volume~44 of {\em Applied Mathematical Sciences}.
\newblock Springer-Verlag, New York, 1983.

\bibitem{MR1341116}
S.~I. Podolynny and N.~I. Portenko.
\newblock On multidimensional stable processes with locally unbounded drift.
\newblock {\em Random Oper. Stochastic Equations}, 3(2):113--124, 1995.

\bibitem{MR1310558}
N.~I. Portenko.
\newblock Some perturbations of drift-type for symmetric stable processes.
\newblock {\em Random Oper. Stochastic Equations}, 2(3):211--224, 1994.

\bibitem{MR1545225}
E.~Rothe.
\newblock {\"U}ber die {G}rundl\"osung bei parabolischen {G}leichungen.
\newblock {\em Math. Z.}, 33(1):488--504, 1931.

\bibitem{MR924157}
W.~Rudin.
\newblock {\em Real and complex analysis}.
\newblock McGraw-Hill Book Co., New York, third edition, 1987.

\bibitem{MR1739520}
K.-i. Sato.
\newblock {\em L\'evy processes and infinitely divisible distributions},
  volume~68 of {\em Cambridge Studies in Advanced Mathematics}.
\newblock Cambridge University Press, Cambridge, 1999.
\newblock Translated from the 1990 Japanese original, Revised by the author.

\bibitem{1998_RS}
R.~L. Schilling.
\newblock Growth and {H}\"older conditions for the sample paths of {F}eller
  processes.
\newblock {\em Probab. Theory Related Fields}, 112(4):565--611, 1998.

\bibitem{MR0367492}
C.~Tsutsumi.
\newblock The fundamental solution for a degenerate parabolic
  pseudo-differential operator.
\newblock {\em Proc. Japan Acad.}, 50:11--15, 1974.

\bibitem{MR705619}
T.~Watanabe.
\newblock The isoperimetric inequality for isotropic unimodal {L}\'evy
  processes.
\newblock {\em Z. Wahrsch. Verw. Gebiete}, 63(4):487--499, 1983.

\bibitem{MR3294616}
L.~Xie and X.~Zhang.
\newblock Heat kernel estimates for critical fractional diffusion operators.
\newblock {\em Studia Math.}, 224(3):221--263, 2014.

\end{thebibliography}

\end{document}